\documentclass[11pt]{amsart}

\usepackage{lmodern}
\usepackage[margin=1in]{geometry}
\usepackage{amsmath,amssymb,amsthm,mathtools}
\usepackage{booktabs,longtable,array}
\usepackage{enumitem}
\usepackage{microtype}
\usepackage{framed}

\usepackage{xcolor}
\definecolor{darkgreen}{RGB}{0,100,0}
\definecolor{darkpurple}{RGB}{80,0,110}
\definecolor{darkbrown}{RGB}{90,45,15}
\usepackage[normalem]{ulem}
\usepackage{etoolbox}
\usepackage{contour}
\usepackage{graphicx}
\usepackage{float}
\usepackage{tikz}
\usetikzlibrary{arrows.meta}
\usepackage{xurl}
\usepackage{aliascnt}
\usepackage[colorlinks=true,linkcolor=blue,citecolor=blue,urlcolor=blue]{hyperref}
\usepackage[capitalise,noabbrev,nameinlink]{cleveref}

\makeatletter
\def\l@section{\@tocline{1}{2pt plus 1pt}{0pt}{2.8pc}{\bfseries}}
\def\l@subsection{\@tocline{2}{0pt}{2.5pc}{3.2pc}{}}
\def\l@subsubsection{\@tocline{3}{0pt}{4.5pc}{3.8pc}{\small}}
\makeatother

\theoremstyle{plain}
\newtheorem{maintheorem}{Theorem}

\newtheorem{theorem}{Theorem}[section]
\newaliascnt{proposition}{theorem}
\newtheorem{proposition}[proposition]{Proposition}
\aliascntresetthe{proposition}
\newaliascnt{lemma}{theorem}
\newtheorem{lemma}[lemma]{Lemma}
\aliascntresetthe{lemma}
\newaliascnt{corollary}{theorem}
\newtheorem{corollary}[corollary]{Corollary}
\aliascntresetthe{corollary}
\theoremstyle{definition}
\newaliascnt{definition}{theorem}
\newtheorem{definition}[definition]{Definition}
\aliascntresetthe{definition}
\newaliascnt{example}{theorem}
\newtheorem{example}[example]{Example}
\aliascntresetthe{example}
\theoremstyle{remark}
\newaliascnt{remark}{theorem}
\newtheorem{remark}[remark]{Remark}
\aliascntresetthe{remark}
\newaliascnt{question}{theorem}
\newtheorem{question}[question]{Question}
\aliascntresetthe{question}
\numberwithin{equation}{section}

\crefname{theorem}{Theorem}{Theorems}
\crefname{maintheorem}{Theorem}{Theorems}
\crefname{proposition}{Proposition}{Propositions}
\crefname{lemma}{Lemma}{Lemmas}
\crefname{corollary}{Corollary}{Corollaries}
\crefname{definition}{Definition}{Definitions}
\crefname{example}{Example}{Examples}
\crefname{remark}{Remark}{Remarks}
\crefname{question}{Question}{Questions}
\crefname{section}{Section}{Sections}
\crefname{appendix}{Appendix}{Appendices}

\newcommand{\R}{\mathbb R}
\newcommand{\Sph}{\mathbb S}
\newcommand{\conv}{\operatorname{conv}}
\newcommand{\aff}{\operatorname{aff}}
\newcommand{\diam}{\operatorname{diam}}
\newcommand{\convRad}{\operatorname{convRad}}
\newcommand{\diff}{\mathop{}\mspace{-4mu}\mathrm{d}}
\newcommand{\supp}{\operatorname{supp}}
\newcommand{\spann}{\operatorname{span}}
\newcommand{\Stack}{\operatorname{Stack}}
\newcommand{\VR}[2]{\operatorname{VR}(#1;#2)}
\newcommand{\DiamConf}{\operatorname{DiamConf}}
\newcommand{\specC}{\scalebox{0.75}{$\scriptstyle\mathrm{C}$}}
\newcommand{\specWS}{\scalebox{0.75}{$\scriptstyle\mathrm{WS}$}}
\AtBeginEnvironment{thebibliography}{%
  \let\unmarkedbibitem\bibitem
  \renewcommand{\bibitem}[2][]{%
    \par\color{black}\hypersetup{urlcolor=blue}%
    \unmarkedbibitem[#1]{#2}}}

\title[Critical Diameters and Vietoris-Rips Filtrations]{Critical Diameters and Vietoris-Rips Filtrations}

\author{Facundo M\'emoli}
\address{Department of Mathematics, Rutgers University,
Piscataway, NJ 08854, USA}
\email{facundo.memoli@rutgers.edu}

\author{Qingsong Wang}
\address{Hal{\i}c{\i}o\u{g}lu Data Science Institute,
University of California San Diego, La Jolla, CA 92093, USA}
\email{qiw072@ucsd.edu}

\subjclass[2020]{Primary 55N31; Secondary 53C23, 49J52}
\keywords{weak slope, compact metric space, Vietoris--Rips filtration,
Kat\v{e}tov functions, fat realization, Clarke criticality, stationary diameter}
\date{}
\begin{document}
\raggedbottom

\begin{abstract}
We study the relation between critical diameter values and the topology of \mbox{Vietoris–Rips} complexes. Using the notion of weak slope from nonsmooth critical-point theory—a metric analogue of the norm of the gradient—we call a finite labelled configuration weak-slope stationary when the diameter function has zero weak-slope there. For a compact metric space, we prove that the canonical inclusion from scale \(r\) to scale \(s\) is a homotopy equivalence whenever \([r,s)\) contains no diameter of a stationary configuration, regardless of its number of labels.

For a closed connected Riemannian manifold, we show that  weak-slope stationary configurations are also Clarke stationary.
The Clarke diameter spectrum has Hausdorff dimension zero for every smooth metric. 
For real-analytic metrics, this spectrum is countable.
We also construct a smooth nonanalytic metric on \(\Sph^2\) whose Clarke diameter spectrum contains a Cantor set. On Riemannian manifolds,
whenever the distance function is smooth near every pair realizing the diameter, weak-slope stationarity and
Clarke criticality are equivalent to another notion called first-order diameter stationarity which is often easier to interpret. In the case of spheres, this latter notion is equivalent to the existence of an equilibrium
stress: nonnegative weights, not all zero, on the diameter pairs whose
weighted distance gradients sum to zero in the tangent space at each
configuration point.

On the unit round sphere
\(\Sph^m\), \(m\geq1\), the least positive weak-slope critical diameter is
\(\arccos(-1/(m+1))\), and the first accumulation point is
\(\arccos(-1/m)\).  A bound on the number of points needed to realize a
critical diameter gives local finiteness below the latter value.
We adapt a construction of Lov\'asz that places rescaled copies
of a spherical configuration on selected parallels in one higher
dimension and adjoins a pole.  We prove that every nonzero nonnegative
equilibrium stress lifts to the resulting configuration.  Through the
 equivalence between the different notions of stationarity, this produces stationary configurations on
\(\Sph^{m+1}\) from stationary configurations on \(\Sph^m\) of diameter
in \((0,\pi)\).  Increasing the number of layers gives critical diameters
approaching the original diameter from below, making every such value
an accumulation point in the next dimension.  We determine
all Cantor--Bendixson derived sets of the iterated stack diameters and obtain
corresponding inclusions for the derived sets of every finite order
of the full spherical spectrum.

Finally, we establish criteria for transferring diameter stationarity from subspaces to ambient spaces and for excluding positive stationary values by quantitative contractions. One-Lipschitz retractions also transfer failures of canonical Vietoris--Rips inclusions to be homotopy equivalences. Examples illustrate both exact agreement between stationary values and transition scales and the failure of stationarity to imply a topological transition.
\end{abstract}

\maketitle
\newpage

\tableofcontents
\newpage

\section{Introduction}
\label{sec:introduction}
The Vietoris--Rips complexes of a metric space are formed by simplices that correspond to finite subsets with a given diameter bound.
As the diameter bound increases, new simplices appear, but their addition need
not change the homotopy type.  We seek a geometric condition ensuring
that the canonical inclusion between two scales is a homotopy
equivalence.
Our proof starts with deformations that decrease the
diameter of ordered finite configurations.  We then relate the
inclusions between diameter sublevels to the canonical Vietoris--Rips
inclusions through a sequence of comparison maps.

\subsection*{From diameter descent to the canonical Vietoris--Rips map}
For a metric space \((X,d_X)\) and \(t>0\), we use the open
Vietoris--Rips complex \cite[Definition~2.1]{LimMemoliOkutan2024}:
\[
 \VR{X}{t}
 :=\{\sigma\subset X:0<|\sigma|<\infty,\ \diam(\sigma)<t\}.
\]

For \(r<s\), the canonical inclusion \(\VR{X}{r}\to\VR{X}{s}\) is the
identity on vertices.

For \(N\geq1\), consider the labelled diameter
\[
 \diam_N(x_1,\ldots,x_N):=\max_{i<j}d_X(x_i,x_j)
 \qquad\text{on }X^N,
\]
where \(X^N\) has the maximum product metric, repetitions are allowed,
and \(\diam_1:=0\).  The notion of \emph{weak slope} of Degiovanni--Marzocchi and
Katriel \cite{DegiovanniMarzocchi1994,Katriel1994}, recalled in
\cref{def:weak-slope}, measures the
rate of local continuous descent relative to metric displacement.  It
requires no geodesic or differentiable structure.  We call a tuple $x\in X^N$
\emph{weak-slope stationary} when this rate, denoted as $|\diff\diam_N|(x)$, is zero and use the notation
\[
 \Sigma_N^{\specWS}(X)
 :=\{\diam_N(x):x\in X^N,\ |\diff\diam_N|(x)=0\},
 \qquad
 \Sigma^{\specWS}(X):=\bigcup_{N\geq2}\Sigma_N^{\specWS}(X).
\]

Throughout this paper, we understand homotopy statements about a Vietoris--Rips
complex as statements about its geometric realization.
\begin{maintheorem}[First deformation theorem for compact metric spaces]
\label{thm:metric-stationary-chamber-intro}
Let \((X,d_X)\) be a nonempty compact metric space and let \(0<r<s\).
Suppose that
\[
 [r,s)\cap\Sigma^{\specWS}(X)=\varnothing.
\]
Then the canonical inclusion
\[
 \VR{X}{r}\longrightarrow\VR{X}{s}
\]
is a homotopy equivalence.
\end{maintheorem}

Theorem~\ref{thm:metric-stationary-chamber-intro} is a Vietoris--Rips analogue
of the classical noncritical-interval theorem of Morse theory
\cite[Theorem~3.1]{Milnor1963}: the absence of stationary diameter values
guarantees that the canonical inclusion between scales is a homotopy
equivalence. In this sense, Theorem~\ref{thm:metric-stationary-chamber-intro}
serves as a first Morse lemma for the Vietoris--Rips filtration.

The half-open interval comes from the strict convention of working with open Vietoris-Rips complexes: simplices of
diameter \(s\) are absent at scale \(s\), so the upper endpoint is allowed
to be stationary.  
Notably, the conclusion concerns the canonical inclusion,
not merely the existence of a homotopy equivalence between the two
complexes. The proof is given in
Section~\ref*{subsec:configuration-assembly-endpoint}. 
The converse of \cref{thm:metric-stationary-chamber-intro} need not hold: a stationary value need not
correspond to a change in homotopy type; see~\cref{ex:snowflake-stationary-without-transition} for an example.

The proof has two steps: deform the diameter sublevels at
each fixed number of labels, then compare their inclusions with the
canonical Vietoris--Rips inclusion.
Under the hypothesis of Theorem~\ref{thm:metric-stationary-chamber-intro},
Corvellec's deformation theorem \cite[Theorem~2.4]{Corvellec1999},
together with
the argument for passing to the upper endpoint \(s\) in
\cref{cor:half-open-fixed-cardinality}, gives weak homotopy equivalences
between the ordered diameter sublevels at each fixed number of labels,
\[
 \DiamConf_N(X;t):=\{x\in X^N:\diam_N(x)<t\}.
\]
These weak homotopy equivalences between ordered diameter sublevels do not immediately give a map of
Vietoris--Rips complexes, since the chosen homotopies need not agree when
labels are deleted or repeated.
The canonical sublevel inclusions themselves do commute
with deleting and repeating labels, so they form a map of simplicial spaces.
We then turn to the classical theory of fat realization
for simplicial spaces \cite[Appendix~A]{Segal1974} for transferring these homotopy equivalences between ordered diameter sublevels at each fixed number of labels to the canonical
Vietoris--Rips inclusion.
Specifically, we use the homotopy-invariance
and graded connectivity results as presented in the exposition
\cite[Theorem~2.2 and Lemma~2.4]{EbertRandalWilliams2019}.
Together with the incidence resolution in
\cref{prop:natural-configuration-resolution} and the neighborhood model
in \cref{cor:katetov-neighborhood-model}, these results give the
finite-label connectivity criterion in
\cref{thm:configurationwise-chamber}.

More generally, \Cref{thm:configurationwise-chamber} controls the homotopy groups of the canonical Vietoris–Rips inclusion using only finitely many labelled diameter sublevels.   For a nonempty metric space \(X\), scales
\(0<r<s\), and an integer \(N\geq2\), suppose that
\[
 \DiamConf_q(X;r)\longrightarrow\DiamConf_q(X;s)
\]
is \((N-q)\)-connected for every \(1\leq q\leq N\).  Then the
canonical Vietoris--Rips inclusion is \((N-1)\)-connected; see
\Cref{thm:configurationwise-chamber}\textup{(i)}.
This means a bijection on path components, isomorphisms on \(\pi_j\)
for \(1\leq j\leq N-2\), and a surjection on \(\pi_{N-1}\), at
every basepoint.  Thus controlling a fixed range of homotopy groups
requires only finitely many label numbers, with progressively weaker
connectivity hypotheses as the label number increases.  This criterion
requires neither compactness nor a weak-slope condition.

For \(N=2\), the one-label inclusion is the identity, so it suffices
that
the inclusion \(\DiamConf_2(X;r)\hookrightarrow\DiamConf_2(X;s)\) be surjective on path components.  This holds
for every length space and recovers Virk's surjectivity theorem for
the canonical maps on fundamental groups; see
\Cref{cor:path-metric-rips-pi1-surjectivity}.

\subsection*{Critical diameter values on Riemannian manifolds}

On a finite-dimensional Riemannian manifold and given  a locally Lipschitz function $f:M\to \R$, Clarke's subdifferential \cite[Section~2.1]{Clarke1990}, denoted $\partial^{\mathrm C}f$ and  recalled in
\cref{def:clarke-subdifferential}, gives a geometric test for criticality.
Every weak-slope stationary tuple is Clarke critical by
\cref{prop:clarke-first-order}.  Moreover, \cref{prop:clarke-first-order} also proves that at positive
diameter and away from the cut locus of the diameter pairs, either condition
is equivalent to the absence of a common first-order descent direction.
Clarke criticality remains a necessary condition for
weak-slope stationarity at the cut locus, where the smooth first-order
test need not apply.  For a closed connected Riemannian manifold \(M\), put
\[
 \Sigma_N^{\specC}(M)
 :=\{\diam_N(x):x\in M^N,\ 0\in\partial^{\mathrm C}\diam_N(x)\},
 \qquad
 \Sigma^{\specC}(M):=\bigcup_{N\geq2}\Sigma_N^{\specC}(M).
\]
Constant tuples are global minima, so these spectra include zero,
as do the weak-slope spectra.

\begin{maintheorem}[Clarke--Sard theorem for diameter]
\label{thm:fixed-label-clarke-sard-intro}
Let \((M,g)\) be a closed connected \(C^\infty\) Riemannian manifold and
let \(N\geq2\).  Then \(\Sigma_N^{\specC}(M)\) is compact and has Hausdorff
dimension zero.  Moreover, the spectra are nested:
\[
 \Sigma_q^{\specC}(M)\subseteq\Sigma_N^{\specC}(M)
 \qquad(2\leq q\leq N).
\]
These conclusions hold even when an endpoint of a diameter pair lies
in the cut locus of the other.
If \(M\) and \(g\) are real analytic, \(\Sigma_N^{\specC}(M)\) is finite.
\end{maintheorem}

The \hyperref[proof:fixed-label-clarke-sard]{proof in
Section~\ref*{subsec:fixed-label-clarke-spectrum}} represents squared
distance as a minimum of smooth energies over a
compact family of broken paths.  \Cref{lem:marginal-maximum-sard}
then represents the critical diameter values as critical values of smooth
functions,
whose sets of critical values have Hausdorff dimension zero by
\cite[Corollary on p.~169]{Sard1965}.  The analytic refinement applies the
finiteness theorem of Bolte--Daniilidis--Lewis--Shiota
\cite[Corollary~9(ii)]{BolteEtAl2007} after verifying subanalyticity of
the induced distance.  Smoothness alone does not give finiteness: a smooth
nonanalytic rotational metric on \(\Sph^2\) has a Cantor subset in every
\(\Sigma_N^{\specC}\), realized by tuples that are not local minima
(\cref{prop:cantor-clarke-spectrum}).

At each fixed number of labels, compactness and Hausdorff dimension
zero give open intervals containing no Clarke critical diameter values.
A single interval avoiding these values at every label number gives
homotopy equivalences of the canonical Vietoris--Rips inclusions by
\cref{thm:metric-stationary-chamber-intro}; a gap at one fixed label
number does not alone provide this conclusion.
The uniform convexity gap in \cref{prop:uniform-convexity-gap}
guarantees a nonempty initial interval avoiding critical values at every
label number.\par

In particular, we obtain an initial interval on which
the Vietoris--Rips complexes retain the homotopy type of the manifold.
Define \(r_{\mathrm C}(M):=\inf\bigl(\Sigma^{\specC}(M)\setminus\{0\}\bigr)\),
with \(\inf\varnothing:=\infty\).
\Cref{prop:uniform-convexity-gap} gives
\(r_{\mathrm C}(M)\geq\convRad(M)>0\), where \(\convRad(M)\) is the
global convexity radius of \cref{def:global-convexity-radius}.
\Cref{cor:convexity-radius-rips-stability} extends Hausmann's guaranteed
initial range {to \(r_{\mathrm C}(M)\)}: the canonical inclusions are
homotopy equivalences whenever \(0<r<s\leq r_{\mathrm C}(M)\), and
\[
 \VR{M}{t}\simeq M\qquad(0<t\leq r_{\mathrm C}(M)).
\]
Hausmann's theorem \cite[Theorem~3.5]{Hausmann1995} supplies the
small-scale identification; our deformation theorem extends it
through \(r_{\mathrm C}(M)\).
A gap at only \(N\) labels gives
an \((N-1)\)-connected canonical map.  The case of the circle will show that this
connectivity bound is sharp; see
\cref{cor:all-label-clarke-chambers,cor:circle-connectivity-sharpness}.

Notably, stationarity is a \emph{necessary} condition for homological change.  For \(p\geq1\), our deformation theorem shows that every positive homological critical scale in degree \(p\) is the diameter of a weak-slope stationary configuration with at most \(p+2\) labels; see \cref{prop:homological-change-stationarity}.
Combined with the fixed-label Clarke--Sard theorem, this confines those
scales to a compact set of Hausdorff dimension zero.  For real-analytic
Riemannian metrics, it yields a finite barcode in every fixed degree;
see \cref{cor:degreewise-persistence-finiteness}.
Thus an infinite all-label stationary spectrum does not entail
infinitely many homological transitions in any fixed degree.
Stationarity alone need not imply a Vietoris--Rips homotopy transition;
see \cref{ex:snowflake-stationary-without-transition}.
For the circle, however, the positive nonantipodal stationary diameters
coincide exactly with the homotopy transition scales;
see \cref{ex:circle-stationary-rips-intro}.

\subsection*{Round spheres: first values and accumulation}

Every  round sphere \(\Sph^m \subset \R^{m+1}\) in this paper has unit radius and 
geodesic metric \(d_m(x,x') = \arccos(x\cdot x')\).  Write
\[ 
 \zeta_j:=\arccos\!\left(-\frac{1}{j+1}\right)
 \qquad(j\geq0).
\]
Thus \(\zeta_j\) is the diameter of the vertex set of a regular
\((j+1)\)-simplex inscribed in \(\Sph^j\).
An \emph{equilibrium stress} (see \Cref{subsec:spherical-clarke-spectra}) assigns symmetric nonnegative weights, not all zero,
to diameter pairs so that the weighted tangent directions toward diameter neighbors
sum to zero at each vertex; \cref{cor:spherical-gordan-stress}
derives this balance, via Gordan's alternative, from the
absence of a direction simultaneously decreasing all diameter-realizing
distances to first order (\cref{prop:clarke-first-order}).  Applying Euclidean
Jung's theorem \cite{Jung1901} then gives the least positive
value at \(N\) labels in \cref{thm:first-spherical-clarke-value}:
\[
 \min\bigl(\Sigma_N^{\specC}(\Sph^m)\setminus\{0\}\bigr)
 =\min\bigl(\Sigma_N^{\specWS}(\Sph^m)\setminus\{0\}\bigr)
 =\zeta_{\min\{m,N-2\}}
 \qquad(m\geq1,\ N\geq2).
\]
At equality, the distinct entries of the tuple are precisely the vertices of a regular \(\min\{m+1,N-1\}\)-simplex centered at the origin; when \(N>m+2\), the additional labels are repetitions.  In particular, the least positive
value over all label numbers is \(\zeta_m\).

For every unit round sphere,
\(r_{\mathrm C}(\Sph^m)=\zeta_m>\pi/2=\convRad(\Sph^m)\).
\Cref{cor:convexity-radius-rips-stability} therefore enlarges the
convexity-radius range \(0<t\leq\pi/2\) to \(0<t\leq\zeta_m\),
throughout which \(\VR{\Sph^m}{t}\simeq\Sph^m\) and
the canonical Vietoris--Rips inclusions between parameters in this range
are homotopy equivalences; see
\Cref{cor:optimal-spherical-initial-chamber}.
\par

The circle gives a complete model, including the label number needed at
each scale.  Put \(\delta_k:=2\pi k/(2k+1)\) for \(k\geq1\).

\begin{example}[The circle and the Adamaszek--Adams classification]
\label{ex:circle-stationary-rips-intro}
For every \(N\geq2\),
\[
 \Sigma_N^{\specC}(\Sph^1)
 =\Sigma_N^{\specWS}(\Sph^1)
 =\{0,\pi\}\cup
 \left\{\delta_k:
  1\leq k\leq\left\lfloor\frac{N-1}{2}\right\rfloor\right\}.
\]
The value \(\delta_k\) is realized by a regular \((2k+1)\)-gon and cannot be realized with fewer than \(2k+1\) labels.  Thus both positive
nonantipodal spectra are exactly \(\{\delta_k:k\geq1\}\).
For \(0<r\leq s<\pi\), the canonical inclusion
\[
 \VR{\Sph^1}{r}\longrightarrow\VR{\Sph^1}{s}
\]
is a homotopy equivalence if and only if
\([r,s)\cap\{\delta_k:k\geq1\}=\varnothing\).
The assertion about the strict Vietoris--Rips filtration is the classification of
Adamaszek and Adams \cite[Theorem~7.4]{AdamaszekAdams2017};
\cref{sec:circle-rips} identifies its transition scales by a direct
stationary calculation.
\end{example}

For each fixed \(N\), the spectrum \(\Sigma_N^{\specC}(\Sph^m)\) is finite,
but the circle already shows that the union over \(N\) can accumulate.  The next theorem
identifies where accumulation first occurs in every dimension.

For \(k\geq1\), the construction used below places \(k\)
rescaled copies of a finite spherical configuration on \(k\) selected
parallels of a sphere of one higher dimension and adjoins a pole.  We call
the resulting configuration a \emph{\(k\)-layer Lov\'asz stack}, after the
construction introduced by Lov\'asz
\cite[proof of Theorem~1]{Lovasz1983}.  Each application raises the sphere
dimension by one, regardless of \(k\).  A one-layer stack uses just one
copy of the base configuration together with the pole; successive
applications of the one-layer construction give the configurations
appearing in the following theorem.

\begin{maintheorem}[First accumulation of spherical Clarke diameter values]
\label{thm:first-accumulation-spheres-intro}
For every \(m\geq1\) and every \(b<\zeta_{m-1}\), the set
\[
 \Sigma^{\specC}(\Sph^m)\cap(0,b]
\]
is finite.  
Moreover, applying the one-layer Lov\'asz-stack construction
\(m-1\) times to regular odd polygons gives a sequence of distinct Clarke
diameter values increasing to \(\zeta_{m-1}\). Consequently,
the least accumulation point of \(\Sigma^{\specC}(\Sph^m)\) is
\[
 \zeta_{m-1} = \arccos(-\tfrac{1}{m}).
\]
 For example, it is
\(\pi\) for \(m=1\), \(\tfrac{2\pi}{3}\) for \(m=2\), and
\(\arccos(-\tfrac{1}{3})\) for \(m=3\).
\end{maintheorem}

The proof of Theorem~\ref{thm:first-accumulation-spheres-intro},
given in \hyperref[proof:first-accumulation-spheres]{Section~\ref*{subsec:first-accumulation-point}},
establishes local finiteness by reducing to the fixed-label finiteness of
Theorem~\ref{thm:fixed-label-clarke-sard-intro}.
For each fixed \(b<\zeta_{m-1}\), we show that every critical diameter
below that bound can be realized by a uniformly bounded number of
points.  Retaining the endpoints of positive-weight stress
edges gives a configuration with the same diameter whose points are uniformly separated by
\cref{lem:subcritical-held-separation}.  Spherical packing then bounds
their number, as in \cref{prop:bounded-realizing-supports}.  All these values lie in
one finite fixed-label spectrum.  Regular odd polygons and successive
one-layer stacks show that the resulting threshold is sharp.

We prove that every nonzero nonnegative equilibrium stress on a finite
spherical base of diameter \(\alpha\in(\pi/2,\pi)\) lifts to its
\(k\)-layer Lov\'asz stack.  As \(k\) increases, the stack diameters
approach \(\alpha\) from below.  Consequently, every positive
nonantipodal stationary diameter on \(\Sph^m\) becomes an accumulation
point of stationary diameters on \(\Sph^{m+1}\); see
\Cref{thm:stationary-stack-lifting,thm:lovasz-stack-spectrum-transport}.

For \(A\subseteq[0,\pi]\), let \(\operatorname{Acc}(A)\) be its set of
accumulation points in \([0,\pi]\).  Define the finite-order
Cantor--Bendixson derived sets by
\(A^{(0)}:=A\) and \(A^{(\nu+1)}:=\operatorname{Acc}(A^{(\nu)})\),
\(\nu\in\mathbb Z_{\geq0}\);
compare \cite[Section~6.C]{Kechris1995}. 
These successive derivatives record accumulation points,
accumulation points of accumulation points, and so on.

\begin{maintheorem}[Accumulation from Lov\'asz stacks]
\label{thm:lovasz-stack-hierarchy-intro}
For every \(n\geq2\), every
\(\alpha\in\Sigma^{\specC}(\Sph^{n-1})\cap(0,\pi)\), and every integer
\(k\geq1\), the \(k\)-layer Lov\'asz stack over a Clarke-critical
configuration of diameter \(\alpha\) in \(\Sph^{n-1}\) is Clarke critical
in \(\Sph^n\).  Its diameter, denoted by \(\Lambda_k(\alpha)\), depends
only on \(\alpha\) and \(k\).  These stack diameters satisfy
\[
 \Lambda_1(\alpha)<\Lambda_2(\alpha)<\cdots<\alpha,
 \qquad
 \Lambda_k(\alpha)\in\Sigma^{\specC}(\Sph^n)\cap(0,\pi),
 \qquad
 \Lambda_k(\alpha)\longrightarrow\alpha.
\]
Equatorial inclusion and the stack sequences together give
\[
 \Sigma^{\specC}(\Sph^{n-1})\cap(0,\pi)
 \subseteq
 \bigl[\Sigma^{\specC}(\Sph^n)\cap(0,\pi)\bigr]
 \cap
 \operatorname{Acc}\bigl(\Sigma^{\specC}(\Sph^n)\cap(0,\pi)\bigr)
 \qquad(n\geq2).
\]
\end{maintheorem}

The \hyperref[proof:lovasz-stack-hierarchy]{proof in
Section~\ref*{subsec:stress-lifting-spectral-transport}} constructs the
stresses and the approaching sequences.  Iterating the
stack construction gives
\[
 \zeta_{n-\nu}\in[\Sigma^{\specC}(\Sph^n)]^{(\nu)}
 \qquad(0\leq\nu\leq n);
\]
the full derivative inclusions are stated in
\cref{cor:stack-derived-consequences}.  In particular, each sphere of
dimension at least two has countably infinitely many explicitly obtained
accumulation points.  For the family generated from odd polygons by
successive stacks, we compute every derived set exactly
(\cref{thm:explicit-lovasz-stack-tree}).
These equalities determine the least points of the derived
sets of the constructed family. Whether the full spectrum has the same
least points remains open for \(2\leq\nu\leq n\).
See \Cref{ques:derived-zeta-hierarchy}.

\subsection*{Geometric applications and limits of the criterion}
\noindent\emph{Geometric applications.}
A one-Lipschitz retraction from \(X\) onto a metric subspace \(Y\)
transfers weak-slope stationary diameter values from \(Y\) to \(X\).
If a canonical inclusion \(\VR{Y}{a}\to\VR{Y}{b}\), \(0<a<b\), fails
to be a homotopy equivalence, the corresponding inclusion for \(X\)
also fails; see
\Cref{prop:weak-slope-nonexpansive-retract}.
Riemannian products with a circle supply examples; see
\Cref{ex:circle-product-canonical-obstructions}.
For complete connected Riemannian manifolds, smooth distance-preserving
embeddings transfer equilibrium stresses when the ambient diameter pairs
avoid the cut locus; see \Cref{prop:stress-transfer-metric-isometric-embedding}.
The equatorial embedding \(\Sph^1\hookrightarrow\Sph^2\) transfers the
circle's positive nonantipodal stationary values, although \(\Sph^2\)
does not retract onto the equator;
see \Cref{ex:equatorial-stress-without-retraction}.

Quantitative contractions provide another application: control of diameter
decrease relative to displacement excludes positive weak-slope stationary
diameter values.  Radial geodesic contractions satisfy the local criterion
on \(\R\)-trees and \(\operatorname{CAT}(0)\) spaces, without a boundedness
assumption; see \Cref{prop:uniform-contraction-no-stationarity,%
prop:busemann-contraction-no-stationarity,rem:metric-tree-boundedness}.
For comparison, Hausmann's crushing theorem concerns a deformation
retraction onto a subspace during which pairwise distances never increase.
It makes the inclusion of the subspace's Vietoris--Rips complex into the
ambient complex a homotopy equivalence at every positive scale
\cite[Proposition~2.2]{Hausmann1995}, but imposes no quantitative relation
between diameter decrease and displacement; see
\Cref{rem:hausmann-crushing-comparison}.

\smallskip
\noindent\emph{Limitations.}
A stationary diameter value
need not mark a Vietoris--Rips transition.  For the snowflaked interval
\(X=([0,1],|x-y|^\alpha)\), \(0<\alpha<1\), one has
\(\Sigma_N^{\specWS}(X)=[0,1]\) for every \(N\geq2\), although all
positive-scale Vietoris--Rips complexes are contractible and all canonical
inclusions between them are homotopy equivalences; see
\Cref{ex:snowflake-stationary-without-transition}.
This space admits a crushing onto a point, so crushing alone does not
exclude weak-slope stationarity.  A different limitation occurs on spaces
with no nonconstant continuous paths: the weak slope of every continuous
function vanishes; see \Cref{prop:path-rigidity-weak-slope}.

\subsection*{Relation to earlier work}
Weak slope and its deformation theory originate in nonsmooth
critical-point theory
\mbox{\cite{DegiovanniMarzocchi1994,Katriel1994,Corvellec1999}}. Our deformation argument builds on Corvellec's metric deformation
theorem \cite{Corvellec1999} and the comparison of
Lim--M\'emoli--Okutan between Vietoris--Rips complexes and metric
neighborhoods,
whose comparison maps commute up to homotopy with the inclusions between the corresponding complexes and neighborhoods
\cite[Corollary~4.3]{LimMemoliOkutan2024}.
\Cref{prop:natural-configuration-resolution} supplies the comparison
with ordered diameter sublevels.
This comparison allows the
realization theorems of Ebert--Randal-Williams
\cite[Theorem~2.2 and Lemma~2.4]{EbertRandalWilliams2019} to yield the
criterion for the connectivity of Vietoris--Rips inclusions in
\cref{thm:configurationwise-chamber}.
The contribution here is to
connect deformation and connectivity information on labelled
configuration spaces to
the canonical Vietoris--Rips inclusions.

Katz used diameter-decreasing deformations to compute particular
Kuratowski-neighborhood homotopy types for \(\Sph^1\), \(\Sph^2\),
and complex projective spaces
\cite[Theorems~1.1, 7.1, and~8.1]{Katz1991}.
\Cref{thm:metric-stationary-chamber-intro} instead gives a
criterion for homotopy equivalence for
arbitrary compact metric spaces, though it does not identify the underlying homotopy type. On a Riemannian manifold, at positive diameter and
whenever the distance function is smooth near every pair realizing the diameter, \Cref{prop:clarke-first-order} identifies
weak-slope stationarity and Clarke criticality with the absence of a common
first-order descent direction; compare the gradient criteria in
\cite[Section~4.2]{Katz1991} and, for \(\mathbb S^2\),
\cite[Definition~3.1 and Remark~3.2]{KatzMemoliWang2023}.
A different approach is Zaremsky's generalized
Bestvina--Brady Morse theory, whose descending links encode
diameter-lowering deletions and diameter-preserving insertions of
vertices in the closed Vietoris--Rips filtration
\cite[Theorem~3.5]{Zaremsky2022}; its geometric applications also
use suitable dense subspaces \cite[Corollary~4.8]{Zaremsky2022}.
Our criterion instead uses continuous deformations of labelled configurations that decrease their diameter, assembled through
\cref{thm:configurationwise-chamber}.  The mechanisms are
complementary: on discrete spaces weak slope vanishes identically
by \cref{prop:path-rigidity-weak-slope}, whereas descending-link
criteria can still yield homotopy equivalences.

Barbet--Dambrine--Daniilidis--Rifford establish a nonsmooth Sard theorem
for functions obtained by minimizing a smooth family over a compact
parameter manifold \cite[Theorem~1]{BarbetEtAl2016}.
\Cref{lem:marginal-maximum-sard} adapts their critical-value reduction
to finite maxima of such minima.  Applied to diameter, this yields
Hausdorff dimension zero for the fixed-label Clarke spectrum on a
closed smooth Riemannian manifold,
without assuming definability or excluding diameter pairs whose endpoints lie in each other's cut locus; see \Cref{thm:fixed-label-clarke-spectrum}.
The analytic finiteness result instead applies the definable
Clarke--Sard theorem of Bolte--Daniilidis--Lewis--Shiota
\cite[Corollary~9(ii)]{BolteEtAl2007}; see
\Cref{cor:analytic-definable-fixed-label-finiteness}.
The Cantor example in \cref{prop:cantor-clarke-spectrum} shows that
smoothness alone does not imply countability.

The spherical specialization of \cref{cor:convexity-radius-rips-stability},
stated in \cref{cor:optimal-spherical-initial-chamber}, gives an alternative
proof of the optimal Hausmann theorem for spheres, established by
Lim--M\'emoli--Okutan
\cite[Theorem~7.1 and Section~A.4.1]{LimMemoliOkutan2024}.
On the circle, Adamaszek--Adams determine the homotopy types and
canonical inclusions between the circle's Vietoris--Rips complexes
\cite[Theorem~7.4]{AdamaszekAdams2017}; see
\cref{cor:circle-canonical-rips-maps} for the comparison with stationary diameters.
\Cref{cor:fixed-label-circle-spectra} identifies the corresponding
nonantipodal critical scales through both stationarity notions and
determines the smallest number of labels needed to realize each value;
\cref{cor:circle-connectivity-sharpness} then shows that the
finite-label connectivity bound is sharp.
For persistence, Chazal--de Silva--Oudot establish finite-rank maps
for totally bounded metric spaces
\cite[Proposition~5.1]{ChazalDeSilvaOudot2014}.
\Cref{prop:homological-change-stationarity} additionally locates
homological critical scales in each degree \(p\geq1\) in the
weak-slope spectrum with \(p+2\) labels.  On closed connected
real-analytic Riemannian manifolds, combining this with fixed-label
finiteness gives finitely many bars in each fixed positive degree;
see \Cref{cor:degreewise-persistence-finiteness}.

\label{sec:related-walk}
The spherical layer construction is classical: Lov\'asz formulates
it for strongly self-dual bases
\cite[proof of Theorem~1]{Lovasz1983}, while Katz treats triangular
layers on \(\Sph^2\) and records the regular-simplex extension to
higher dimensions in a remark
\cite[Section~1.1 and Remark~1.2, pp.~120--121]{Katz1989}.
The two descriptions use different geometric coordinates.  We prove
an explicit correspondence between Lov\'asz's description and our
extension of Katz's spherical-coordinate description to arbitrary
finite spherical bases; see
\Cref{prop:ellipse-stack-walk-equivalence}.
The relevant extension here is preservation of a nonzero nonnegative
equilibrium stress on any finite spherical base of diameter in
\((\pi/2,\pi)\), proved in
\cref{thm:stationary-stack-lifting}.
\Cref{prop:canonical-lovasz-transform} controls the diameter
transform, and \cref{thm:lovasz-stack-spectrum-transport} shows that every positive
nonantipodal stationary value becomes an accumulation point of stationary
values in the next dimension.
\Cref{cor:stack-derived-consequences} gives inclusions for the
finite-order derived sets of the full spectrum, whereas
\cref{thm:explicit-lovasz-stack-tree} computes the derived sets of
the explicit iterated-stack family.

For the two-sphere, Katz proves local minimality of finite subsets
critical in his Hausdorff-perturbation sense below diameter
\(2\pi/3\), and finiteness up to congruence below each fixed bound
\(d<2\pi/3\) \cite[Theorem~2 and Corollary~1]{Katz1989}.
The value \(\zeta_1 = 2\pi/3\) also appears as the first accumulation point of
critical diameter values on \(\Sph^2\) in
\cite[Theorem~3.10]{KatzMemoliWang2023}.
Still for the two-sphere, the same preprint describes configurations consisting of a rescaled
regular odd polygon on a parallel together with a pole, whose diameters
increase to \(2\pi/3\)
\cite[Section~2.1 and Example~2.7]{KatzMemoliWang2023}.
In our terminology, these are one-layer Lov\'asz stacks over regular
odd polygons---the two-dimensional case of
\cref{thm:stationary-iterated-one-layer-stacks}.
Beyond the two-sphere case, \Cref{thm:first-accumulation-spheres} establishes the first-accumulation
formula \(\min\operatorname{Acc}(\Sigma^{\specC}(\Sph^m))=\zeta_{m-1}\)
in every dimension $m$, with sharpness supplied by the iterated one-layer
constructions of \cref{thm:stationary-iterated-one-layer-stacks}.
Its proof uses
\cref{lem:subcritical-held-separation,prop:bounded-realizing-supports}
to bound the size of realizing supports, then invokes fixed-label
finiteness from \cref{cor:analytic-definable-fixed-label-finiteness};
the one-layer families of
\cref{thm:stationary-iterated-one-layer-stacks} prove sharpness.

\subsection*{Organization}
\Cref{sec:stationary-rips-chambers} proves
\cref{thm:metric-stationary-chamber-intro} by passing from ordered diameter
sublevels to the canonical Vietoris--Rips map.
\Cref{sec:clarke-spectrum} proves
\cref{thm:fixed-label-clarke-sard-intro} and uses intervals
without Clarke critical diameter values to obtain deformations.
The spherical minimum and circle spectrum
in \cref{sec:spherical-diameter-spectra} supply the starting values for
the accumulation and stack theorems in \cref{sec:spherical-stacks}
(\cref{thm:first-accumulation-spheres-intro,thm:lovasz-stack-hierarchy-intro}).
\Cref{sec:spectral-mechanisms} studies how retractions, embeddings, and
contractions affect stationarity and canonical maps.
\Cref{app:weak-clarke-comparison} contains the technical weak-slope proofs,
the smooth counterexample construction, and the antipodal argument;
\cref{app:ellipse-stack-walk-coordinate-bridge,app:stack-walk-shooting-convergence}
establish the coordinate and convergence results used by the stack construction.

\subsection*{Notation and conventions}
\label{subsec:notation-conventions}
\Cref{tab:notation} lists notation shared by several sections.
Our sphere normalization and the index in \(\zeta_m\) agree with
\cite[Section~1.2, p.~3741]{LimMemoliSmith2023}.
A subscript \(N\) on a diameter spectrum fixes the number of ordered
labels, including repetitions.  Omitting \(N\) means the union over
\(N\geq2\), with no bound on the number of labels.

\begingroup
\small
\renewcommand{\arraystretch}{1.12}
\begin{longtable}{@{}>{\raggedright\arraybackslash}p{0.29\textwidth}
                   >{\raggedright\arraybackslash}p{0.67\textwidth}@{}}
\caption{Notation used across sections.}\label{tab:notation}\\
\toprule
Notation & Meaning\\
\midrule
\endfirsthead
\multicolumn{2}{@{}l}{\textit{Table~\thetable\ continued}}\\
\toprule
Notation & Meaning\\
\midrule
\endhead
\bottomrule
\endlastfoot
\(\diam_N(x)\) & Diameter of an ordered \(N\)-tuple; coordinates may repeat and \(\diam_1=0\).\\
\(P\) & Finite point configuration; when \(P=\{x_1,\ldots,x_N\}\) is the  underlying set of a labelled tuple, \(|P|\leq N\).\\
\(d_{X,\infty}\) & Maximum product metric on \(X^N\).\\
\(d_{M,2}\) & Riemannian product distance on \(M^N\).\\
\(\VR{X}{r}\) & Open Vietoris--Rips complex at scale \(r\).\\
\(\DiamConf_N(X;r)\) & Sublevel \(\{x\in X^N:\diam_N(x)<r\}\).\\
\(\Delta_1(X)\), \(d_\infty\) & Kat\v{e}tov functions on \(X\), with the uniform metric.\\
\(|\diff f|\), \(\partial^{\mathrm C}f\) & Weak slope and Clarke subdifferential.\\
\(\Sigma_N^{\specWS}(X)\), \(\Sigma^{\specWS}(X)\) & Diameters of weak-slope stationary \(N\)-tuples, and their union over \(N\geq2\).\\
\(\Sigma_N^{\specC}(M)\), \(\Sigma^{\specC}(M)\) & Diameters of Clarke-critical \(N\)-tuples, and their union over \(N\geq2\).\\
\(\Sph^m\), \(d_m\) & Unit round sphere with intrinsic distance \(d_m(x,x')=\arccos(x\cdot x')\).\\
\(\zeta_j\), \(\delta_k\) & Simplex and circle values: \(\zeta_j=\arccos(-1/(j+1))\), \(\delta_k=2\pi k/(2k+1)\).\\
\(\Stack_k(P)\), \(\Lambda_k(\alpha)\) & The \(k\)-layer Lov\'asz stack and its diameter when \(\diam(P)=\alpha\).\\
\(\operatorname{Acc}(A)\), \(A^{(\nu)}\) & Accumulation points and iterated derived sets, taken in \([0,\pi]\) for spherical spectra.\\
\(\mathcal T_n\) & Diameters obtained by iterating the stack construction to dimension \(n\).\\
\end{longtable}
\endgroup

\subsection*{Acknowledgements}
F.~M\'emoli gratefully acknowledges support from the National Science Foundation through grants CCF-2523653 and DMS-2524362.
F.~M\'emoli thanks the Sydney Mathematical Research Institute
for its hospitality during a visit in which part of this work was developed.
During the development of this work, the authors used ChatGPT Pro and Codex to assist with literature searches, the exploration and testing of ideas, and the auditing of proofs.
The authors checked the mathematical claims, proofs, and references and take
full responsibility for the contents of the paper.

\section{{Metric deformation and Vietoris--Rips inclusions}}
\label{sec:stationary-rips-chambers}
For the proof of \cref{thm:metric-stationary-chamber-intro}, fix a
nonempty compact metric space \((X,d_X)\) and scales \(0<r<s\)
satisfying its hypotheses.  We must pass from deformations of finite
configurations to a comparison of Vietoris--Rips complexes. We work with tuples in \(X^N\), allowing repetitions, so \(N\)
counts labels rather than distinct points.  For each intermediate scale
\(r<u<s\), \cref{thm:fixed-cardinality-metric-descent} supplies the
fixed-label deformations between scales \(r\) and \(u\). 

The canonical sublevel inclusions commute with deletion and repetition
of labels; we assemble them using fat realization.
To compare this realization with the Vietoris--Rips complex, we use the
neighborhood model of Lim--M\'emoli--Okutan
\cite[Section~2.2 and Corollary~4.3]{LimMemoliOkutan2024} in the
space \(\Delta_1(X)\) of Kat\v{e}tov functions on \(X\).  The ball cover of an open neighborhood has
the Vietoris--Rips complex as its nerve.  An incidence space records a
neighborhood point together with labels whose balls contain it, and
relates the same neighborhood to the ordered diameter sublevels.
The incidence projections of
\cref{prop:natural-configuration-resolution} give the two right-hand
squares in \cref{fig:configuration-assembly-overview}, and these commute
strictly.  The neighborhood comparison of
\cref{cor:katetov-neighborhood-model} gives the leftmost square, which
commutes up to homotopy.
They give the connectivity bounds in \cref{thm:configurationwise-chamber}.
The compact-exhaustion argument in \cref{cor:half-open-fixed-cardinality}
permits the upper endpoint \(s\).

\begin{figure}[H]
\centering
\begingroup
\definecolor{assemblyaccent}{HTML}{000000}
\begin{tikzpicture}[
 x=1cm,y=1cm,>={Latex[length=1.5mm]},
 space/.style={font=\small,inner xsep=2pt,inner ysep=3pt},
 comparison/.style={->,draw=black!60,line width=.45pt},
 focus/.style={->,draw=assemblyaccent,line width=.8pt},
 heading/.style={font=\scriptsize,align=center,text=black!75},
 note/.style={heading,anchor=north,text width=2.65cm,inner sep=0pt},
 arrowlabel/.style={font=\scriptsize,fill=white,inner sep=1pt}
]
\node[heading] at (0,.85) {Vietoris--Rips\\complex};
\node[heading] at (3.1,.85) {Open\\neighborhood};
\node[heading] at (6.4,.85) {Incidence\\resolution};
\node[heading] at (10.2,.85) {Ordered diameter\\sublevels};
\foreach \assemblyscale/\assemblyy in {r/0,u/-1.55}{
 \node[space] (R-\assemblyscale) at (0,\assemblyy)
  {\(\bigl|\VR{X}{\assemblyscale}\bigr|\)};
 \node[space] (E-\assemblyscale) at (3.1,\assemblyy)
  {\(B_{\assemblyscale/2}(X)\)};
 \node[space] (N-\assemblyscale) at (6.4,\assemblyy)
  {\(\|\mathcal N_\bullet(X;\assemblyscale)\|\)};
 \node[space] (C-\assemblyscale) at (10.2,\assemblyy)
  {\(\|\DiamConf_\bullet(X;\assemblyscale)\|\)};
 \draw[comparison] (E-\assemblyscale.west)--
  node[arrowlabel,above] {\(h_{\assemblyscale,X}\)}
  node[arrowlabel,below] {\(\simeq\)} (R-\assemblyscale.east);
 \draw[comparison] (N-\assemblyscale.west)--
  node[arrowlabel,above] {\(\simeq\)} (E-\assemblyscale.east);
 \draw[focus] (N-\assemblyscale.east)--
  node[arrowlabel,above] {\(\sim_{\mathrm w}\)} (C-\assemblyscale.west);
}
\foreach \assemblycolumn in {E,N}{
 \draw[comparison] (\assemblycolumn-r.south)--(\assemblycolumn-u.north);
}
\draw[focus] (R-r.south)--
 node[arrowlabel,left] {\(\simeq\)} (R-u.north);
\draw[focus] (C-r.south)--
 node[arrowlabel,left] {\(\sim_{\mathrm w}\)} (C-u.north);
\node[note] (E-note) at (3.1,-2.12)
 {ball cover has\\Vietoris--Rips nerve};
\node[note] (N-note) at (6.4,-2.12)
 {labels carry the\\metric topology};
\node[note] (C-note) at (10.2,-2.12)
 {degree \(k\):\\\(\DiamConf_{k+1}(X;t)\)};
\coordinate (assembly-footer) at ([yshift=-12pt]C-note.south);
\draw[focus,densely dashed] (C-u.center |- assembly-footer)--
 node[arrowlabel,below=5pt,text=assemblyaccent,fill=none]
  {equivalence transfers through the comparison maps}
 (R-u.center |- assembly-footer);
\end{tikzpicture}
\endgroup
\caption{The comparisons at \(r<u<s\), under the hypotheses of
\cref{thm:metric-stationary-chamber-intro}.
The vertical arrows in the first two columns are the
Vietoris--Rips and neighborhood inclusions from \(r\) to \(u\); those
in the last two columns are obtained by applying fat realization to
the inclusions in each degree.  The rightmost one is a weak homotopy
equivalence by deformation for each number of labels and fat realization.
The leftmost square commutes up to homotopy by
\cite[Corollary~4.3]{LimMemoliOkutan2024}; the other squares commute
strictly.  Together they transfer the equivalence to the
canonical Vietoris--Rips inclusion.
Here \(B_{t/2}(X)\) is the open neighborhood in \(\Delta_1(X)\),
and double bars denote fat realization.
The symbols \(\simeq\) and \(\sim_{\mathrm w}\) distinguish homotopy
equivalences from weak homotopy equivalences.  The passage to the upper endpoint \(s\) uses \cref{cor:half-open-fixed-cardinality}.}
\label{fig:configuration-assembly-overview}
\end{figure}

\subsection{Deforming ordered diameter sublevels}
\label{sec:diameter-configuration-sublevels}
Let \((X,d_X)\) be a nonempty metric space.  For \(N\geq1\), equip \(X^N\)
with the maximum product metric
\[
 d_{X,\infty}(x,x'):=\max_{1\leq i\leq N}d_X(x_i,x'_i),
\]
and recall
\begin{equation}
 \diam_N(x):=\max_{1\leq i<j\leq N}d_X(x_i,x_j),
 \qquad
 \DiamConf_N(X;t):=\{x\in X^N:\diam_N(x)<t\}.
 \label{eq:metric-labelled-diameter}
\end{equation}
We set \(\diam_1=0\), so for \(t>0\)
\(\DiamConf_1(X;t)=X\).  These are sublevels of diameter on ordered
tuples.  Coordinates may coincide, unlike in the usual ordered
configuration space \cite[Section~2]{BaryshnikovBubenikKahle2014}; thus
\(\{x_1,\ldots,x_N\}\) may contain fewer than \(N\) points.
Deleting coordinates cannot increase diameter; repeating or permuting
coordinates preserves it.
They are open subspaces of \(X^N\), since
\[
 |\diam_N(x)-\diam_N(x')|\leq2d_{X,\infty}(x,x').
\]

A deformation must decrease diameter continuously for all nearby tuples,
not just for one configuration.  Weak slope expresses this requirement
while controlling the displacement of the points.  We use the
continuous-function definition in \cite[Definition~2.1]{Corvellec1999}.

\begin{definition}[Weak slope~\cite{DegiovanniMarzocchi1994,Katriel1994}]
\label{def:weak-slope}
Let \(f:Z\to\R\) be continuous on a metric space \((Z,d_Z)\).
The \emph{weak slope} \(|\diff f|(z)\) is the supremum of the numbers
\(\sigma\geq0\) for which there are \(\delta>0\) and a continuous map
\[
 H:B_\delta(z)\times[0,\delta]\longrightarrow Z
\]
such that, for every \(y\in B_\delta(z)\) and \(u\in[0,\delta]\),
\[
 d_Z(H(y,u),y)\leq u,
 \qquad
 f(H(y,u))\leq f(y)-\sigma u.
\]
We call \(z\) \emph{weak-slope stationary} if \(|\diff f|(z)=0\).
\end{definition}

The displacement bound gives \(H(y,0)=y\).  A positive admissible
\(\sigma\) therefore supplies a homotopy near \(z\) that decreases \(f\)
at rate at least \(\sigma\), while moving each point by at most the time parameter.

\begin{definition}[Weak-slope diameter spectra]
\label{def:weak-slope-diameter-spectra}
For \(N\geq1\), set
\begin{equation}
 \Sigma_N^{\specWS}(X)
 :=\{\diam_N(x):x\in X^N,\ |\diff\diam_N|(x)=0\},
 \qquad
 \Sigma^{\specWS}(X):=\bigcup_{N\geq2}\Sigma_N^{\specWS}(X).
 \label{eq:weak-slope-diameter-spectra}
\end{equation}
The weak slope is computed with \(d_{X,\infty}\).  Every constant tuple
is stationary, so these spectra include zero. We refer to an element in $\Sigma^{\specWS}(X)$ as a \emph{weak-slope critical value of diameter}.
\end{definition}

Corvellec's theorem converts a uniform positive weak-slope bound into a
homotopy between sublevels.  We use the following form, in which the
homotopy never increases the function and moves the larger sublevel into
the smaller one.

\begin{theorem}[Corvellec's deformation theorem 
{\cite[Theorem~2.4]{Corvellec1999}}]
\label{thm:corvellec-noncritical-interval}
Let \((Z,d_Z)\) be complete, let \(f:Z\to\R\) be continuous, and let
\(\alpha<\beta\) and \(\delta,\sigma>0\).  Suppose
\[
 \alpha-\delta\leq f(z)\leq\beta+\delta
 \quad\Longrightarrow\quad |\diff f|(z)>\sigma.
\]
There is a continuous homotopy \(\eta:Z\times[0,1]\to Z\) with
\(\eta(-,0)=\operatorname{id}_Z\) such that
\[
 f(\eta(z,u))\leq f(z)\qquad(z\in Z,\ 0\leq u\leq1),
\]
and its time-one map sends \(f^{-1}(({-}\infty,\beta])\) into
\(f^{-1}(({-}\infty,\alpha])\).
\end{theorem}

For fixed \(N\), compactness of \(X^N\) and the following lower
semicontinuity property give the required uniform bound.

\begin{lemma}[Lower semicontinuity of the weak slope
{\cite[p.~265]{Corvellec1999}}]
\label{lem:weak-slope-lsc}
For a continuous function \(f:Z\to\R\), the function
\(z\mapsto|\diff f|(z)\) is lower semicontinuous.
\end{lemma}

\begin{proof}
Any admissible rate witnessed on \(B_\delta(z)\times[0,\delta]\)
is also witnessed at every \(z'\in B_{\delta/2}(z)\), by restricting to
\(B_{\delta/2}(z')\times[0,\delta/2]\).  Thus each strict superlevel
set of the weak slope is open.
\end{proof}

\begin{proposition}[Fixed-label application of Corvellec's theorem]
\label{thm:fixed-cardinality-metric-descent}
Let \(X\) be a nonempty compact metric space, let \(N\geq1\), and let
\(0<r<s\).  If
\[
 [r,s]\cap\Sigma_N^{\specWS}(X)=\varnothing,
\]
then the inclusion
\[
 \DiamConf_N(X;r)\longrightarrow\DiamConf_N(X;s)
\]
is a homotopy equivalence.  
More precisely, there is a homotopy from the identity on the larger
sublevel to a map into the smaller sublevel.  At every time, the homotopy
preserves both sublevels.
\end{proposition}

\begin{proof}
The case \(N=1\) is the identity.  Fix \(N\geq2\).
The stationary set is closed by \cref{lem:weak-slope-lsc} and hence compact. Its image under the continuous function \(\diam_N\) is \(\Sigma_N^{\specWS}(X)\), which is therefore compact.  We may therefore choose \(0<\varepsilon<c<r\)
so that \([c-\varepsilon,s+\varepsilon]\) avoids this spectrum.
The set \(K:=\diam_N^{-1}([c-\varepsilon,s+\varepsilon])\) is compact.  If \(K\)
is empty, the two sublevels are equal.  Otherwise lower semicontinuity
gives \(0<\sigma<\min_K|\diff\diam_N|\).

Apply \cref{thm:corvellec-noncritical-interval} on the complete space
\(X^N\), with levels \(c<s\) and margin \(\varepsilon\).
The homotopy never increases diameter and at time one sends every tuple
of diameter in \([c,s]\) to diameter at most \(c<r\); tuples below
\(c\) remain there.  Its time-one restriction therefore defines
\[
 g:\DiamConf_N(X;s)\longrightarrow\DiamConf_N(X;r).
\]
The homotopy stays in each of the two sublevels.  Its restrictions join
the identities to \(i g\) and \(g i\), respectively, where \(i\)
is the inclusion.  Hence \(g\) is its homotopy inverse.
\end{proof}

\subsection{Open neighborhoods in the space of Kat\v{e}tov functions}
\label{subsec:katetov-neighborhoods}

\Cref{thm:fixed-cardinality-metric-descent} gives the fixed-label
homotopy equivalences.  To relate them to the Vietoris--Rips complex,
we use an intermediate metric space: an open neighborhood of \(X\) in
an injective space.  Recall that a metric space \(E\) is
\emph{injective} if every one-Lipschitz map \(A\to E\) extends to a
one-Lipschitz map \(B\to E\) whenever \(A\subseteq B\) is an isometric
inclusion \cite[Section~2]{Lang2013}.
The neighborhood can be compared both with the complex and with the
realized ordered diameter sublevels.

\Cref{cor:katetov-neighborhood-model} identifies the homotopy
type of the neighborhood with that of the Vietoris--Rips complex.
It uses Lim--M\'emoli--Okutan's neighborhood model
\cite[Section~2.2]{LimMemoliOkutan2024}: balls centered at points of
\(X\) form a cover whose nerve is the Vietoris--Rips complex.
To relate the neighborhood to realized ordered diameter sublevels in
\cref{prop:natural-configuration-resolution}, we retain a point in
an intersection and the tuple of centers of its balls.
The space of Kat\v{e}tov functions
is useful here because those intersections are convex, giving explicit
contractions of the fibers.

For a nonempty metric space \(X\), a \emph{Kat\v{e}tov function}
\cite{Melleray2008} is a function \(f:X\to\R\) satisfying
\[
 |f(z)-f(w)|\leq d_X(z,w)\leq f(z)+f(w)
 \qquad(z,w\in X).
\]
Such functions describe the possible distances from points of \(X\) to
a point in a metric extension of \(X\).
Following \cite[Section~3]{Lang2013}, we write
\begin{equation}
 \Delta_1(X):=\{f:X\to\R:f\text{ is a Kat\v{e}tov function}\}
 \label{eq:katetov-space}
\end{equation}
for the space of all Kat\v{e}tov functions.  Its pointwise-minimal
elements form the tight span \cite[Theorem~3.3]{Lang2013}.
Its defining inequalities imply
\(f\geq0\) and are preserved by pointwise convex combinations.
We equip it with the uniform metric
\[
 d_\infty(f,g):=\sup_{z\in X}|f(z)-g(z)|.
\]
\Cref{thm:configurationwise-chamber} will apply even to unbounded spaces.  To check that
this metric is finite in that setting, fix \(x_*\in X\).  Then
the two Kat\v{e}tov inequalities give
\[
 d_X(z,x_*)-f(x_*)\leq f(z)\leq d_X(z,x_*)+f(x_*),
\]
and hence \(|f(z)-g(z)|\leq f(x_*)+g(x_*)\).
Thus convex combinations and their straight-line homotopies are
continuous for \(d_\infty\).
Lang proves that \(\Delta_1(X)\) with this metric is injective
\cite[Proposition~3.2]{Lang2013}, so the neighborhood comparison of
\cite[Propositions~2.25 and~2.27]{LimMemoliOkutan2024} applies.

Write \(d_x(z):=d_X(x,z)\).  The map \(x\mapsto d_x\) is an isometric
embedding of \(X\) into \(\Delta_1(X)\), the Kuratowski embedding.
The standard evaluation identity \cite[Equation~(3.2)]{Lang2013} is,
for every \(f\in\Delta_1(X)\),
\begin{equation}
 \|f-d_x\|_\infty=f(x).
 \label{eq:katetov-distance-to-kuratowski}
\end{equation}
Indeed, the Kat\v{e}tov inequalities give
\(|f(z)-d_X(x,z)|\leq f(x)\), with equality at \(z=x\).

Write \(B_\rho(X)\) for the open \(\rho\)-neighborhood of the
embedded copy of \(X\) in \(\Delta_1(X)\), as in
\cite[Section~2.2]{LimMemoliOkutan2024}.
At Vietoris--Rips scale \(t>0\), the neighborhood we use is
\begin{equation}
 B_{t/2}(X)=\bigcup_{x\in X}V_x(t),
 \qquad
 V_x(t):=\{f\in\Delta_1(X):f(x)<t/2\}.
 \label{eq:katetov-neighborhood-cover}
\end{equation}
By \eqref{eq:katetov-distance-to-kuratowski}, \(V_x(t)\) is exactly
the open ball of radius \(t/2\) about \(d_x\).  The factor \(1/2\)
matches our diameter convention for \(\VR{X}{t}\).

For a nonempty finite \(F\subset X\), put
\(V_F(t):=\bigcap_{x\in F}V_x(t)\).
The cover detects precisely the Vietoris--Rips simplices:
\begin{equation}
 V_F(t)\ne\varnothing\quad\Longleftrightarrow\quad\diam(F)<t.
 \label{eq:katetov-intersection-rips-condition}
\end{equation}
One implication follows from
\(d_X(x,x')\leq f(x)+f(x')<t\) for \(f\in V_F(t)\) and \(x,x'\in F\).
For the converse, put \(D=\diam(F)<t\).  The function
\(z\mapsto D/2+d_X(z,F)\) is one-Lipschitz, and
\[
 d_X(z,w)\leq d_X(z,F)+D+d_X(w,F)
\]
shows that it is Kat\v{e}tov.  Its value on \(F\) is \(D/2<t/2\), so it
belongs to \(V_F(t)\).  Moreover, every nonempty \(V_F(t)\) is convex,
because its additional evaluation inequalities are preserved by convex combinations.  It is therefore contractible.

Thus the cover \(\{V_x(t)\}_{x\in X}\) has nerve \(\VR{X}{t}\)
and contractible nonempty finite intersections.
Lim--M\'emoli--Okutan's functorial nerve lemma
\cite[Corollary~4.3]{LimMemoliOkutan2024} gives the following comparison.

\begin{proposition}[Lim--M\'emoli--Okutan's neighborhood comparison in
\(\Delta_1(X)\)]
\label{cor:katetov-neighborhood-model}
For every nonempty metric space \(X\) and \(0<r\leq s\), there are
homotopy equivalences
\begin{equation}
 h_{t,X}:B_{t/2}(X)\longrightarrow\bigl|\VR{X}{t}\bigr|,
 \qquad
 t\in\{r,s\},
 \label{eq:natural-rips-neighborhood-zigzag}
\end{equation}
such that
\[
 h_{s,X}\circ v_{r,s}\simeq i_{r,s}\circ h_{r,X},
\]
where \(v_{r,s}:B_{r/2}(X)\hookrightarrow B_{s/2}(X)\) and
\(i_{r,s}:|\VR{X}{r}|\hookrightarrow|\VR{X}{s}|\) are
the canonical inclusions.  The comparison is also natural up to homotopy
under isometries of \(X\).
\end{proposition}

\begin{proof}
The spaces \(B_{r/2}(X)\) and \(B_{s/2}(X)\) are metrizable and hence
paracompact.  The covers \(\{V_x(r)\}_{x\in X}\) and
\(\{V_x(s)\}_{x\in X}\) have contractible nonempty finite
intersections, and satisfy \(V_x(r)\subseteq V_x(s)\).
Their nerves are the strict Vietoris--Rips complexes by
\eqref{eq:katetov-intersection-rips-condition}.  The conclusion is
therefore exactly \cite[Corollary~4.3]{LimMemoliOkutan2024}, whose index set may have arbitrary cardinality.

An isometry \(g\) acts on functions by
\((g_*f)(z):=f(g^{-1}z)\), carrying \(V_x(t)\) to \(V_{g(x)}(t)\).
Naturality up to homotopy under this action follows from
\cite[Theorem~4.2]{LimMemoliOkutan2024}.
\end{proof}

It remains to compare this neighborhood with ordered tuples carrying their metric topology, rather than treating the labels as a discrete set.

\subsection{Recovering ordered diameter sublevels}
\label{subsec:incidence-configuration-comparison}

To relate continuous motions in \(X^{k+1}\) to the neighborhood
\(B_{t/2}(X)\), recall that a tuple \((x_0,\ldots,x_k)\) has diameter
less than \(t\) exactly when the balls \(V_{x_i}(t)\) have a common
point, by \eqref{eq:katetov-intersection-rips-condition}.
We retain both the tuple and a point \(f\) in this intersection:

\begin{equation}
 \mathcal N_k(X;t)
 :=\{(f,x_0,\ldots,x_k)\in\Delta_1(X)\times X^{k+1}:
                         f(x_i)<t/2\text{ for every }i\}.
 \label{eq:incidence-nerve-levels}
\end{equation}
All coordinates carry their metric topology.  The inequality
\[
 |f(x)-g(y)|\leq\|f-g\|_\infty+d_X(x,y)
\]
shows that evaluation is continuous.  The projection
\(\mathcal N_0(X;t)\to B_{t/2}(X)\), \((f,x)\mapsto f\), has the
continuous local section \(f\mapsto(f,x)\) on \(V_x(t)\).

The spaces \(\mathcal N_k(X;t)\) form the \v Cech nerve of this
projection: a point in degree \(k\) records \(k+1\) pairs \((f,x_i)\)
with the same image \(f\), or equivalently a point of the
\((k+1)\)-fold fiber product over \(B_{t/2}(X)\)
\cite[Corollary~1.5 and Section~4]{DuggerIsaksen2004}.
Face maps delete a label and degeneracy maps repeat one
\cite[Chapter~1, \S1]{May1967}.
Forgetting \(f\) leaves precisely the ordered diameter sublevel
\(\DiamConf_{k+1}(X;t)\subset X^{k+1}\).  Deleting or repeating a
coordinate preserves the diameter bound, so these sublevels also form
a simplicial space, with

\[
 \bigl(\DiamConf_\bullet(X;t)\bigr)_k:=\DiamConf_{k+1}(X;t)
\]
in degree \(k\).  The incidence spaces therefore project both
to the neighborhood and to the diameter sublevels.  We compare these
projections after geometric realization.
For a simplicial space \(Z_\bullet\), its \emph{fat geometric
realization} is
\begin{equation}
 \|Z_\bullet\|
 :=\left(\coprod_{k\geq0}Z_k\times\Delta^k\right)\big/\!\sim,
 \label{eq:fat-realization-definition}
\end{equation}
where \(\Delta^k\) is the standard topological simplex and only the face
identifications are imposed
\cite[Appendix~A, p.~308]{Segal1974}.  Products and quotients are
taken in compactly generated spaces
\cite[Convention~1.5]{EbertRandalWilliams2019}.
A map of simplicial spaces that is a weak homotopy equivalence in every
degree induces a weak homotopy equivalence on their fat realizations
\cite[Theorem~2.2]{EbertRandalWilliams2019}.

First consider the projection that forgets the labels.  It induces
\[
 \epsilon_{t,X}:\|\mathcal N_\bullet(X;t)\|\longrightarrow B_{t/2}(X),
 \qquad
 \epsilon_{t,X}[(f,x_0,\ldots,x_k);a]:=f,
\]
because the \(f\)-coordinate is unchanged by face identifications.
This is the usual augmentation from a realized \v Cech nerve
to its base.  We prove that it is a homotopy equivalence by adapting
the partition-of-unity argument of
\cite[\S4, Proposition~4.1]{Segal1968}; compare
\cite[Proposition~4G.2]{Hatcher2002}.
The proof keeps the metric topology on the labels throughout.

\begin{lemma}[The incidence resolution]
\label{lem:incidence-cech-resolution}
The augmentation
\[
 \epsilon_{t,X}:\|\mathcal N_\bullet(X;t)\|\longrightarrow B_{t/2}(X)
\]
is a homotopy equivalence.  For \(0<r<s\), the augmentations
at \(r\) and \(s\), together with the inclusion
\(B_{r/2}(X)\hookrightarrow B_{s/2}(X)\) and the map on fat realizations
induced by \(\mathcal N_k(X;r)\hookrightarrow\mathcal N_k(X;s)\),
form a strictly commutative square.
\end{lemma}

\begin{proof}
We construct a continuous section \(s_{t,X}\) of
\(\epsilon_{t,X}\) and a homotopy from the identity to
\(s_{t,X}\epsilon_{t,X}\).
On each \(V_x(t)\), the map \(f\mapsto(f,x)\) chooses a label incident
to \(f\).  A partition of unity combines these local choices into
a point of the realization, using the partition values as barycentric
coordinates.

Since \(B_{t/2}(X)\) is metrizable and hence paracompact, we
can choose a locally finite
partition of unity \((\lambda_\alpha)_{\alpha\in\mathcal A}\)
and points \(x_\alpha\in X\) such that
\(\supp\lambda_\alpha\subset V_{x_\alpha}(t)\), and order
\(\mathcal A\).
For each \(f\), list the finitely many indices with positive weight as
\(\alpha_1<\cdots<\alpha_m\), and write
\[
 \mathbf x(f):=(x_{\alpha_1},\ldots,x_{\alpha_m}),
 \qquad
 \lambda(f):=(\lambda_{\alpha_1}(f),\ldots,\lambda_{\alpha_m}(f)).
\]
The list is nonempty, each selected label \(x_{\alpha_i}\)
satisfies \(f(x_{\alpha_i})<t/2\), and
its weights sum to one.  Hence
\[
 s_{t,X}(f):=[(f,\mathbf x(f));\lambda(f)]
\]
is a section of \(\epsilon_{t,X}\).

We now prove the continuity of {\(s_{t,X}\)}.
Fix \(f_0\in B_{t/2}(X)\). By local finiteness and the
inclusions \(\operatorname{supp}\lambda_\alpha\subset V_{x_\alpha}(t)\),
there are a neighborhood \(U\) of \(f_0\) and finitely many indices
\(\alpha_1<\cdots<\alpha_q\) such that all other weights vanish on \(U\)
and \(U\subset V_{x_{\alpha_i}}(t)\) for every \(i\). Thus, for every
\(f\in U\), we may represent \(s_{t,X}(f)\) using the fixed labels
\(x_{\alpha_1},\ldots,x_{\alpha_q}\) and the continuous weights
\(\lambda_{\alpha_1}(f),\ldots,\lambda_{\alpha_q}(f)\), allowing zero
weights. This representative is continuous, and deleting zero-weight
coordinates recovers the original definition. Hence \(s_{t,X}\) is
continuous.

We now deform each point of the realization to
\(s_{t,X}(f)\) while keeping its image \(f\) under \(\epsilon_{t,X}\)
fixed.  Concatenating the two label lists lets us transfer the
barycentric weight from the original point to the section.
For \(z=[(f,\mathbf y);a]\), concatenate its list with the section's
list and put
\[
 H(z,u):=[(f,\mathbf y,\mathbf x(f));((1-u)a,u\lambda(f))].
\]
All labels remain incident to \(f\), and the weights sum to one.
Deleting an original zero-weight coordinate commutes with this
formula, so \(H\) respects the face identifications.
Using a fixed finite list of labels locally, as above,
proves continuity on each defining simplex;
the compactly generated quotient and product then give continuity
on the realization times \([0,1]\).  At the endpoints the zero
weights are removed, giving
\[
 H(-,0)=\operatorname{id},\qquad
 H(-,1)=s_{t,X}\epsilon_{t,X},\qquad
 \epsilon_{t,X}s_{t,X}=\operatorname{id}.
\]
Finally, for \(0<r<s\), the inclusions
\(\mathcal N_k(X;r)\hookrightarrow\mathcal N_k(X;s)\) induce the map
\([(f,\mathbf x);a]\mapsto[(f,\mathbf x);a]\) on fat realizations.
Applying \(\epsilon_{s,X}\) after this map, or applying
\(\epsilon_{r,X}\) followed by the inclusion
\(B_{r/2}(X)\hookrightarrow B_{s/2}(X)\), sends every representative
to the same function \(f\).
\end{proof}

The other projection is
\(\pi_{k,t,X}:\mathcal N_k(X;t)\to\DiamConf_{k+1}(X;t)\),
\((f,x)\mapsto x\), which forgets the neighborhood point.
For \((f,x_0,\ldots,x_k)\in\mathcal N_k(X;t)\),
\[
 d_X(x_i,x_j)\leq f(x_i)+f(x_j)<t,
\]
so the projection takes values in
\(\DiamConf_{k+1}(X;t)\).
To prove that \(\pi_{k,t,X}\) is a homotopy equivalence, we
construct a continuous section \(x\mapsto(q_x,x)\).
For a fixed tuple \(x\), the possible first coordinates form the convex
intersection \(\bigcap_i V_{x_i}(t)\).  Choosing \(q_x\) in that
intersection lets us contract the fiber to \((q_x,x)\) by straight-line
interpolation.  The choice must vary continuously with \(x\), including
when labels collide.  The function used in
\eqref{eq:katetov-intersection-rips-condition} has this property.
For \(x=(x_0,\ldots,x_k)\), define
\begin{equation}
 D_x:=\diam_{k+1}(x),\qquad
 m_x(z):=\min_i d_X(z,x_i),\qquad
 q_x(z):=\tfrac12D_x+m_x(z).
 \label{eq:configuration-katetov-center}
\end{equation}
The function \(q_x\) is one-Lipschitz, and the triangle inequality gives
\[
 d_X(z,w)\leq m_x(z)+D_x+m_x(w)=q_x(z)+q_x(w).
\]
Thus \(q_x\in\Delta_1(X)\).  At every label,
\begin{equation}
 q_x(x_i)=\tfrac12D_x=\tfrac12\diam_{k+1}(x).
 \label{eq:center-incidence-bound}
\end{equation}
In particular, \(D_x<t\) puts \(q_x\) in every ball \(V_{x_i}(t)\).

\begin{lemma}[Levelwise comparison with diameter sublevels]
\label{lem:cech-level-configuration}
For every \(k\geq0\), the projection
\[
 \pi_{k,t,X}:\mathcal N_k(X;t)\longrightarrow\DiamConf_{k+1}(X;t),
 \qquad
 \pi_{k,t,X}(f,x):=x,
\]
is a homotopy equivalence.
The map \(x\mapsto(q_x,x)\) is a section, and
straight-line interpolation in the first coordinate gives a
fiberwise strong deformation retraction onto its image.
\end{lemma}

\begin{proof}
By \eqref{eq:katetov-intersection-rips-condition}, the projection takes
values in the stated sublevel, and \eqref{eq:center-incidence-bound}
shows that \(\sigma_{k,t,X}(x):=(q_x,x)\) is a section.
For the maximum product metric, the functions
\(D_x\) and \(m_x\) satisfy
\[
 |D_x-D_{x'}|\leq2d_{X,\infty}(x,x'),\qquad
 \|m_x-m_{x'}\|_\infty\leq d_{X,\infty}(x,x').
\]
Consequently,
\begin{equation}
 \|q_x-q_{x'}\|_\infty\leq2d_{X,\infty}(x,x').
 \label{eq:center-collision-continuity}
\end{equation}
This estimate applies to all tuples, including those with repeated entries, and therefore proves continuity of the section.

For a fixed tuple \(x\), the fiber consists of the Kat\v{e}tov
functions satisfying \(f(x_i)<t/2\) for every \(i\).
It is convex, and \(q_x\) belongs to it.  Define
\[
 H_{k,t}((f,x),u):=((1-u)f+u q_x,x).
\]
The interpolated function is Kat\v{e}tov, and
\[
 ((1-u)f+u q_x)(x_i)
 =(1-u)f(x_i)+u q_x(x_i)<t/2.
\]
Thus the homotopy keeps the tuple fixed and remains in its incidence
fiber.  It is continuous by \eqref{eq:center-collision-continuity}
and continuity of pointwise convex combinations in the uniform
metric.  Finally,
\[
 \pi_{k,t,X}\sigma_{k,t,X}=\operatorname{id},\quad
 H_{k,t}(-,0)=\operatorname{id},\quad
 H_{k,t}(-,1)=\sigma_{k,t,X}\pi_{k,t,X},
\]
and \(H_{k,t}(\sigma_{k,t,X}(x),u)=\sigma_{k,t,X}(x)\).
These are the asserted fiberwise deformation-retraction identities.
\end{proof}

\begin{proposition}[Comparing neighborhoods with realized diameter sublevels]
\label{prop:natural-configuration-resolution}
For every \(t>0\), there is a zigzag of weak homotopy equivalences
\begin{equation}
 B_{t/2}(X)
 \xleftarrow[\simeq]{\ \epsilon_{t,X}\ }
 \|\mathcal N_\bullet(X;t)\|
 \xrightarrow[\sim_{\mathrm w}]{\ \|\pi_{\bullet,t,X}\|\ }
 \|\DiamConf_\bullet(X;t)\|.
 \label{eq:natural-configuration-resolution}
\end{equation}
The left-hand map is a homotopy equivalence.
% \sout{Both comparison maps
% commute strictly with increasing scale.}
For \(0<r<s\), the neighborhood inclusion and the maps
obtained by applying fat realization to the inclusions in each degree
give a strictly commutative diagram between the displayed zigzags
at \(r\) and \(s\).
\end{proposition}

\begin{proof}
The augmentation is covered by \cref{lem:incidence-cech-resolution}.
The projections \((f,x_0,\ldots,x_k)\mapsto(x_0,\ldots,x_k)\)
commute with deletion and repetition, so they form a simplicial map.
By \cref{lem:cech-level-configuration}, this map is a homotopy
equivalence in every degree.
Theorem~2.2 of
\cite{EbertRandalWilliams2019} makes its fat realization a weak
homotopy equivalence.
The inclusions in each degree leave all coordinates
unchanged, so they commute with the coordinate projections.
These equalities persist after fat realization.
\end{proof}

For the analogous ordered simplicial-set construction, see
\cite[Definition~4.3]{Otter2022}.  The comparison above instead uses
configuration spaces with their inherited metric topology, as required
for the fixed-label deformations.

\Cref{prop:natural-configuration-resolution} relates the
Kat\v{e}tov neighborhood \(B_{t/2}(X)\) to the realized diameter sublevels
\(\|\DiamConf_\bullet(X;t)\|\).
We use this zigzag of weak homotopy equivalences at \(t=r\) and \(t=s\),
for \(0<r<s\), in \cref{thm:configurationwise-chamber}.
The canonical inclusions of ordered diameter sublevels form a
simplicial map.  Fat realization applies to these inclusions and to
the projections \(\pi_{k,t,X}\), all of which commute with deleting
and repeating labels.  To prove a homotopy equivalence in each degree,
we may choose an inverse and a homotopy separately in that degree.
Thus the centers \(q_x\) in \cref{lem:cech-level-configuration} and the
diameter-decreasing deformations of
\cref{thm:fixed-cardinality-metric-descent} need not be compatible
across label numbers.

\subsection{{Comparing the inclusions and treating the upper endpoint}}
\label{subsec:configuration-assembly-endpoint}

We can now compare the canonical Vietoris--Rips inclusion using the
ordered diameter sublevels.
\Cref{thm:configurationwise-chamber} applies to every nonempty metric
space and does not assume a weak-slope condition.
Its proof follows the maps in \cref{fig:configuration-assembly-overview}.

An \(m\)-connected map induces isomorphisms on homotopy groups in
degrees less than \(m\) and a surjection in degree \(m\), at every
basepoint.  For \(m\geq1\) it also induces a bijection on path
components; a \(0\)-connected map is surjective on path components.
A \emph{weak homotopy equivalence} induces a bijection on path components
and isomorphisms in every positive degree
\cite[Section~4.1]{Hatcher2002}.

\begin{theorem}[{Vietoris--Rips inclusions from ordered diameter sublevels}]
\label{thm:configurationwise-chamber}
Let \(X\) be a nonempty metric space and let \(0<r<s\).
\begin{enumerate}[label=\textup{(\roman*)}]
\item If, for some \(N\geq1\), each inclusion
\[
 \DiamConf_q(X;r)\longrightarrow\DiamConf_q(X;s),
 \qquad 1\leq q\leq N,
\]
is \((N-q)\)-connected, then the canonical inclusion
\(\VR{X}{r}\to\VR{X}{s}\) is \((N-1)\)-connected.
\item If the configuration inclusion is a weak homotopy equivalence
for every \(q\geq1\), then the canonical Vietoris--Rips inclusion is a homotopy
equivalence.
\end{enumerate}
\end{theorem}

\begin{proof}
For \textup{(i)}, we first prove that the map
\[
 \|\DiamConf_\bullet(X;r)\|
 \longrightarrow
 \|\DiamConf_\bullet(X;s)\|
\]
induced by the inclusions in each degree is \((N-1)\)-connected.
\Cref{prop:natural-configuration-resolution,cor:katetov-neighborhood-model}
will then transfer this connectivity to the canonical Vietoris--Rips inclusion.
Let \(f_\bullet\) be
the map of simplicial spaces whose degree-\(k\) map is the inclusion
\[
 f_k:\DiamConf_{k+1}(X;r)\longrightarrow\DiamConf_{k+1}(X;s).
\]
For \(0\leq k<N\), it is \((N-1-k)\)-connected.
The graded realization theorem says that a semisimplicial map
whose degree-\(k\) map is \((m-k)\)-connected in every degree induces
an \(m\)-connected map on realizations
\cite[Lemma~2.4]{EbertRandalWilliams2019}.
Apply it with \(m=N-1\), after truncating the underlying
semisimplicial spaces by putting the empty space in degrees \(k\geq N\).
It follows that the map between the \((N-1)\)-skeleta is
\((N-1)\)-connected.  The inclusions of these skeleta into the full
realizations are also \((N-1)\)-connected
\cite[Lemma~2.1]{EbertRandalWilliams2019}.
Thus \(\|f_\bullet\|\) induces isomorphisms below degree \(N-1\).
A class in degree \(N-1\) lifts first to the target skeleton and
then through the map between the skeleta, proving surjectivity.
For \(N=1\) the same argument is on path components; otherwise it
applies at every basepoint, with the usual change-of-basepoint isomorphisms along paths.

We transfer the connectivity of \(\|f_\bullet\|\) through the
strictly commutative squares of
\cref{prop:natural-configuration-resolution}.  Since their comparison
maps are weak homotopy equivalences, the neighborhood inclusion
\[
 v_{r,s}:B_{r/2}(X)\longrightarrow B_{s/2}(X)
\]
is also \((N-1)\)-connected.
The homotopy equivalences \(h_{r,X}\) and \(h_{s,X}\) in
\cref{cor:katetov-neighborhood-model} give the diagram
\begin{equation}
\begin{array}{ccc}
 B_{r/2}(X)&\xrightarrow{\ h_{r,X}\ }&|\VR{X}{r}|\\[5pt]
 {\scriptstyle v_{r,s}}\Big\downarrow&&
 \Big\downarrow{\scriptstyle i_{r,s}}\\[5pt]
 B_{s/2}(X)&\xrightarrow{\ h_{s,X}\ }&|\VR{X}{s}|.
\end{array}
\label{eq:canonical-rips-scale-diagram}
\end{equation}
It commutes up to homotopy:
\(h_{s,X}v_{r,s}\simeq i_{r,s}h_{r,X}\).
Both horizontal maps are homotopy equivalences, so the connectivity of
\(v_{r,s}\) transfers to the canonical inclusion
\(i_{r,s}\).  This proves \textup{(i)}.

Under the hypothesis of \textup{(ii)}, part~\textup{(i)} applies for
every \(N\).  Hence the same map \(i_{r,s}\) is a weak homotopy
equivalence.  The realizations of abstract simplicial complexes are
CW complexes, so Whitehead's theorem, applied on each path component,
makes it a homotopy equivalence \cite[Theorem~4.5]{Hatcher2002}.
\end{proof}

The endpoint \(s\) may itself be stationary in
\cref{thm:metric-stationary-chamber-intro}: the hypothesis excludes
stationary values only from \([r,s)\).
For each fixed number of labels,
\cref{thm:fixed-cardinality-metric-descent} gives the required
equivalence at every scale \(u\) with \(r<u<s\).
To reach scale \(s\), observe that diameter is uniformly less than
\(s\) on every compact subset of \(\DiamConf_N(X;s)\).
Such a subset therefore lies in \(\DiamConf_N(X;u)\) for some \(u<s\).
 Maps from spheres and their homotopies have compact images, so their
comparison at scale \(s\) reduces to the equivalences at these earlier
scales.  \Cref{lem:compact-exhaustion} makes this passage precise.

\begin{lemma}[ Compact exhaustion at the upper endpoint]
\label{lem:compact-exhaustion}
Let \((Y_t)_{t\leq b}\) be nested subspaces of \(Y_b\), with their
subspace topologies, and suppose every compact subset of \(Y_b\)
is contained in some \(Y_u\), \(u<b\).  Fix \(r<b\).
If each inclusion \(Y_r\to Y_u\), \(r<u<b\), is a weak homotopy
equivalence, then so is \(Y_r\to Y_b\).  If \(Y_r\) and \(Y_b\)
are CW complexes, this last map is a homotopy equivalence.
\end{lemma}

\begin{proof}
 A based map from a sphere into \(Y_b\) has compact image, so it factors
through some \(Y_u\), enlarging \(u\) to ensure \(r<u<b\).
The weak equivalence \(Y_r\to Y_u\) gives surjectivity on homotopy
groups.  A null-homotopy also has compact image and factors through
one such subspace, which gives injectivity.  Applying the same
argument to points and paths gives a bijection on path components.
The CW assertion follows componentwise from Whitehead's theorem
\cite[Theorem~4.5]{Hatcher2002}.
\end{proof}

\begin{corollary}[ Comparison of diameter sublevels over a half-open interval]
\label{cor:half-open-fixed-cardinality}
Let \(X\) be a nonempty compact metric space, \(N\geq1\), and \(0<r<s\).
If \([r,s)\cap\Sigma_N^{\specWS}(X)=\varnothing\), then
\[
 \DiamConf_N(X;r)\longrightarrow\DiamConf_N(X;s)
\]
is a weak homotopy equivalence.
\end{corollary}

\begin{proof}
For every \(r<u<s\), \cref{thm:fixed-cardinality-metric-descent}
gives a homotopy equivalence from scale \(r\) to scale \(u\).
If \(C\subset\DiamConf_N(X;s)\) is nonempty and compact, then
continuity of diameter gives \(\max_C\diam_N<s\).
Choose
\[
 \max\{r,\max_C\diam_N\}<u<s.
\]
Then \(C\subset\DiamConf_N(X;u)\).
Apply \cref{lem:compact-exhaustion}; for \(N=1\) the inclusion
is already the identity.
\end{proof}

\phantomsection
\label{proof:metric-stationary-chamber}
\begin{proof}[\normalfont\bfseries Proof of \cref{thm:metric-stationary-chamber-intro}]
The hypothesis on \(\Sigma^{\specWS}(X)\) implies
\([r,s)\cap\Sigma_N^{\specWS}(X)=\varnothing\) for every \(N\geq2\).
By \cref{cor:half-open-fixed-cardinality}, all configuration
inclusions are weak homotopy equivalences; the one-label inclusion
is the identity on \(X\).
Part~\textup{(ii)} of \cref{thm:configurationwise-chamber}
therefore gives the required homotopy equivalence of the canonical
Vietoris--Rips inclusion.
\end{proof}

A gap at \(N\) labels also controls every smaller label number, because
repeating labels gives \(\Sigma_q^{\specWS}(X)\subseteq\Sigma_N^{\specWS}(X)\) for \(2\leq q\leq N\).  We obtain this from a general
retraction property, which will also be used in
\cref{prop:weak-slope-nonexpansive-retract}.

\begin{lemma}[Weak slope under a one-Lipschitz retraction]
\label{lem:weak-slope-retraction}
Let \(i:Y\to Z\) be an isometric embedding and \(p:Z\to Y\) be
one-Lipschitz with \(p i=\operatorname{id}_Y\).
Suppose \(g:Y\to\R\) and \(f:Z\to\R\) are continuous and
\[
 f\circ i=g,\qquad g\circ p\leq f.
\]
Then \(|\diff g|(y)\geq|\diff f|(i(y))\) for every \(y\in Y\).
\end{lemma}

\begin{proof}
Transfer an admissible local deformation \(H\) for \(f\) by
\(\widetilde H(z,u):=pH(i(z),u)\).
The isometry puts a sufficiently small source ball in the domain,
and
\[
 d_Y(\widetilde H(z,u),z)\leq d_Z(H(i(z),u),i(z))\leq u,\qquad
 g(\widetilde H(z,u))\leq f(H(i(z),u))\leq g(z)-\sigma u.
\]
Every admissible rate for \(f\) is therefore admissible for \(g\).
Taking suprema proves the claim.
\end{proof}

\begin{lemma}[Weak slope under repetition of labels]
\label{lem:weak-slope-repetition-nesting}
Let \(2\leq q\leq N\) and
\(\pi:\{1,\ldots,N\}\twoheadrightarrow\{1,\ldots,q\}\).
For the repetition map \((\iota_\pi(x))_j:=x_{\pi(j)}\),
\[
 |\diff\diam_q|(x)\geq|\diff\diam_N|(\iota_\pi(x)),
 \qquad
 \Sigma_q^{\specWS}(X)\subseteq\Sigma_N^{\specWS}(X).
\]
\end{lemma}

\begin{proof}
Choose a section \(a\) of \(\pi\) and let
\((p_a(y))_i:=y_{a(i)}\).
For the maximum product metrics, \(\iota_\pi\) is isometric,
\(p_a\) is one-Lipschitz, and
\[
 p_a\iota_\pi=\operatorname{id},\qquad
 \diam_N\iota_\pi=\diam_q,\qquad \diam_q p_a\leq\diam_N.
\]
Apply \cref{lem:weak-slope-retraction}.  A zero weak slope on
\(X^q\) forces zero weak slope at its repeated tuple, and repetition
preserves diameter, proving the inclusion of spectra.
\end{proof}

\begin{corollary}[Intervals without stationary values and Vietoris--Rips connectivity]
\label{cor:fixed-label-weak-slope-connectivity}
Let \(X\) be a nonempty compact metric space, \(N\geq2\), and \(0<r<s\).
If \([r,s)\cap\Sigma_N^{\specWS}(X)=\varnothing\), then the
canonical inclusion \(\VR{X}{r}\to\VR{X}{s}\) is \((N-1)\)-connected.
\end{corollary}

\begin{proof}
By \cref{lem:weak-slope-repetition-nesting}, the gap holds for every
\(2\leq q\leq N\).  Then \cref{cor:half-open-fixed-cardinality}
makes each corresponding configuration inclusion a weak homotopy
equivalence.  The one-label map is the identity, so
\cref{thm:configurationwise-chamber}\textup{(i)} applies.
\end{proof}

Fix a field \(\mathbb F\) and an integer \(p\geq1\).
A scale \(c>0\) is \emph{homologically regular in degree \(p\)} if
there is \(0<\varepsilon<c\) such that every canonical inclusion
between scales \(r<s\) in \((c-\varepsilon,c+\varepsilon)\) induces
an isomorphism on \(H_p(-;\mathbb F)\).
Otherwise \(c\) is a
\emph{homological critical scale in degree \(p\)}. We use the
local-constancy convention of Bubenik--Scott
\cite[Definition~4.3]{BubenikScott2014}.

\begin{proposition}[Homological changes force stationary configurations]
\label{prop:homological-change-stationarity}
Let \(X\) be a nonempty compact metric space, let \(p\geq1\), and
let \(\mathbb F\) be a field.  For \(0<r<s\), if the canonical map
\[
 H_p(\VR{X}{r};\mathbb F)
 \longrightarrow H_p(\VR{X}{s};\mathbb F)
\]
is not an isomorphism, then
\[
 [r,s)\cap\Sigma_{p+2}^{\specWS}(X)\neq\varnothing.
\]
If this map is not surjective, then
\[
 [r,s)\cap\Sigma_{p+1}^{\specWS}(X)\neq\varnothing.
\]
Every positive homological critical scale in degree \(p\) belongs to
\(\Sigma_{p+2}^{\specWS}(X)\).  If every neighborhood of a scale
\(c>0\) contains scales \(r<s\) for which the displayed map is not
surjective, then \(c\in\Sigma_{p+1}^{\specWS}(X)\).
\end{proposition}

\begin{proof}
An \((N-1)\)-connected map of CW complexes induces homology
isomorphisms in degrees below \(N-1\) and a surjection in degree
\(N-1\).  Apply \cref{cor:fixed-label-weak-slope-connectivity}
with \(N=p+2\) and \(N=p+1\), respectively, and take
contrapositives.

For each fixed \(N\), lower semicontinuity of the weak slope
in \cref{lem:weak-slope-lsc} makes its zero locus closed in the
compact space \(X^N\).
The image of this zero locus under the continuous function \(\diam_N\) is \(\Sigma_N^{\specWS}(X)\), which is therefore compact.  If \(c>0\) is
outside this spectrum, it has a neighborhood disjoint from the
spectrum, and every half-open scale interval in that neighborhood
satisfies the corresponding gap hypothesis.  The preceding
connectivity conclusions give the two assertions at an exact scale.
\end{proof}

To apply \cref{thm:configurationwise-chamber}\textup{(i)}
with \(N=2\), it is enough that
the inclusion \(\DiamConf_2(X;r)\hookrightarrow\DiamConf_2(X;s)\) be surjective on
path components.  The Vietoris--Rips inclusion is
then \(1\)-connected: it induces a bijection on components and a
surjection on fundamental groups.  New triangles can still kill
loops.

A \emph{length space} is a metric space in
which the distance between two points is the infimum of the lengths of
rectifiable paths joining them \cite[Definition~I.3.1]{BridsonHaefliger1999}.
For a nonempty length space \(X\), Virk proved that the
canonical inclusion induces a surjection
\[
 (i_{r,s})_*:
 \pi_1\bigl(|\VR{X}{r}|,x_*\bigr)
 \longrightarrow
 \pi_1\bigl(|\VR{X}{s}|,x_*\bigr)
\]
for every \(0<r<s\) and every basepoint \(x_*\in X\)
\cite[Proposition~3.2(9) and Section~11.3]{Virk2020}.
\Cref{cor:path-metric-rips-pi1-surjectivity} derives this surjectivity,
together with the bijection on path components, from
\cref{thm:configurationwise-chamber}.

\begin{corollary}[Recovering Virk's fundamental-group surjectivity]
\label{cor:path-metric-rips-pi1-surjectivity}
Let \(X\) be a nonempty length space and \(0<r<s\).
The canonical inclusion
\(\VR{X}{r}\to\VR{X}{s}\) is \(1\)-connected.
\end{corollary}

\begin{proof}
Given \((x_0,x_1)\in\DiamConf_2(X;s)\), choose a rectifiable path
from \(x_1\) to \(x_0\) of length less than \(s\).
Keep the first label at \(x_0\) and move the second along this path.
At every time, the remaining subpath to \(x_0\) has length less than
\(s\), so the pair remains in \(\DiamConf_2(X;s)\).
It ends at \((x_0,x_0)\in\DiamConf_2(X;r)\).
Thus every path component of \(\DiamConf_2(X;s)\) meets \(\DiamConf_2(X;r)\).  The one-label
map is the identity, and
\cref{thm:configurationwise-chamber}\textup{(i)} applies with \(N=2\).
\end{proof}

For the closed filtration, Chazal--de Silva--Oudot proved the
corresponding surjectivity on first homology by approximate midpoint
subdivisions~\cite[Corollary~6.2]{ChazalDeSilvaOudot2014}.

\section{From weak-slope deformation to geometric critical values}
\label{sec:clarke-spectrum}

To apply \cref{thm:metric-stationary-chamber-intro} to a closed connected
Riemannian manifold, it suffices to exclude Clarke-critical diameter
values in the interval under consideration; see
\cref{prop:clarke-first-order}.  We first establish this comparison,
then prove the fixed-label structure theorem announced in
\cref{thm:fixed-label-clarke-sard-intro} and establish the
label-independent gap above zero in \cref{prop:uniform-convexity-gap}.

\subsection{Weak slope, Clarke criticality, and common descent}
\label{subsec:clarke-comparison}

For a finite maximum of smooth functions, only the
functions attaining the maximum at the point in question contribute
to its first-order change.  For example, both functions \(t\) and
\(-t\) attain the maximum in \(f(t)=\max\{t,-t\}\) at zero.
Their derivatives are \(1\) and \(-1\), so no direction decreases
both functions to first order; equivalently, zero lies in the convex
hull of these two derivatives.
Clarke's subdifferential extends this convex-hull description to
locally Lipschitz functions.

\begin{definition}[Clarke subdifferential
{\cite[Section~2.1]{Clarke1990}}]
\label{def:clarke-subdifferential}
Let \(f:M\to\R\) be locally Lipschitz on a finite-dimensional Riemannian
manifold.  Choose a smooth chart \(\phi:U\to\R^m\) about \(x\), and put
\(h=f\circ\phi^{-1}\).  The \emph{Clarke subdifferential} is
\[
 \partial^{\mathrm C}f(x)
 :=(\diff\phi_x)^*\overline{\conv}\left\{
     \lim_{k\to\infty}\diff h_{z_k}:
     z_k\to\phi(x),\ h\text{ is differentiable at }z_k
   \right\}\subset T_x^*M.
\]
It is a nonempty compact convex set, independent of the chart.  The point
\(x\) is \emph{Clarke critical} if \(0\in\partial^{\mathrm C}f(x)\).
\end{definition}

Degiovanni--Marzocchi's Banach-space comparison
\cite[Theorem~2.17]{DegiovanniMarzocchi1994} has the following Riemannian
form.

\begin{theorem}[Degiovanni--Marzocchi comparison in Riemannian form]
\label{thm:weak-implies-clarke}
Let \(f:M\to\R\) be locally Lipschitz on a finite-dimensional Riemannian
manifold.  Then
\begin{equation}
 |\diff f|_{d_M}(x)
 \geq\operatorname{dist}_{g_x^*}
       \bigl(0,\partial^{\mathrm C}f(x)\bigr).
 \label{eq:riemannian-weak-clarke-comparison}
\end{equation}
In particular, weak-slope stationarity implies Clarke criticality.
\end{theorem}

Here \(g_x^*\) denotes the dual metric on \(T_x^*M\).
The numerical weak slope depends on the metric, but its vanishing is
unchanged by a local bi-Lipschitz change.  We use this fact to compare
stationarity of \(\diam_N\) for the Riemannian and maximum product
metrics.

\begin{lemma}[Local bi-Lipschitz invariance of weak-slope stationarity]
\label{lem:weak-slope-bilipschitz}
Let \(f:Z\to\R\) be continuous, and suppose two metrics \(d,d'\) on
\(Z\) are locally bi-Lipschitz equivalent near \(z\).  Then
\[
 |\diff f|_d(z)=0\quad\Longleftrightarrow\quad|\diff f|_{d'}(z)=0.
\]
\end{lemma}

The proofs of \cref{thm:weak-implies-clarke,lem:weak-slope-bilipschitz}
are given in \cref{app:weak-clarke-comparison}.

For a finite maximum of \(C^1\) functions on a finite-dimensional
Riemannian manifold \(M\), the weak-slope bound
\eqref{eq:riemannian-weak-clarke-comparison} is an equality.
The closest point to zero in the convex hull of the differentials of
the branches attaining the maximum determines the best common rate
of descent.  We express this rate using the following directional
derivative.

\begin{definition}[Upper Dini directional derivative;
cf.\ {\cite[Definition~3.1.3 and Remark~3.1.4]{CannarsaSinestrari2004}}]
\label{def:upper-dini-directional}
For a locally Lipschitz function \(f:M\to\R\) and \(v\in T_zM\), put
\[
 D^+f(z;v):=\limsup_{h\downarrow0}
       \frac{f(\exp_z(hv))-f(z)}{h}.
\]
By local Lipschitz continuity, replacing \(v\) by \(v_h\to v\)
changes the quotient by \(o(1)\).
\end{definition}

Suppose, on a neighborhood of \(z\), that
\(f=\max_{\alpha\in\mathcal A}f_\alpha\), where \(\mathcal A\) is
nonempty and finite and each \(f_\alpha\) is \(C^1\).  Define
\[
 I_f(z):=\{\alpha:f_\alpha(z)=f(z)\},
 \qquad
 K_z:=\conv\{\diff(f_\alpha)_z:\alpha\in I_f(z)\}.
\]
The indices in \(I_f(z)\) are called \emph{active}.  Inactive branches
have a positive gap below \(f(z)\), so they cannot attain the maximum
on a sufficiently small neighborhood of \(z\).

\begin{proposition}[Weak slope of a finite maximum of \(C^1\) functions]
\label{prop:weak-slope-finite-smooth-maximum}
With this notation,
\begin{equation}
 D^+f(z;v)=\max_{\alpha\in I_f(z)}\diff(f_\alpha)_z[v],
 \qquad v\in T_zM,
 \label{eq:dini-finite-maximum}
\end{equation}
and \(\partial^{\mathrm C}f(z)=K_z\).  For the Riemannian distance and
its induced cotangent norm,
\begin{equation}
 |\diff f|(z)
 =\operatorname{dist}(0,K_z)
 =\max_{\lVert v\rVert\leq1}\{-D^+f(z;v)\}.
 \label{eq:weak-slope-finite-maximum}
\end{equation}
Consequently, the following are equivalent:
\begin{enumerate}[label=\textup{(\roman*)}]
\item \(|\diff f|(z)=0\);
\item \(0\in K_z\);
\item no vector strictly decreases every active branch to first order;
\item \(D^+f(z;v)\geq0\) for every \(v\in T_zM\).
\end{enumerate}
\end{proposition}

The proof of \cref{prop:weak-slope-finite-smooth-maximum} is given in
\cref{sec:finite-smooth-maxima}.

\begin{remark}[Smooth interpretation]
\label{rem:smooth-weak-slope}
Taking a single branch in
\cref{prop:weak-slope-finite-smooth-maximum} gives
\[
 |\diff f|(z)=\lVert\diff f_z\rVert=\lVert\nabla f(z)\rVert
 \qquad(f\in C^1(M)),
\]
with the weak slope and both norms computed using the Riemannian metric.
Thus the usual gradient criterion for deformation is a special case.  For a finite smooth maximum, the convex hull of
the active differentials takes the place of the single gradient.
\end{remark}

\phantomsection
\label{sec:riemannian-diameter-stationarity}

The diameter is a finite maximum of pair distances.  Only the pairs
attaining that maximum matter locally: every other pair has a strict gap
below it.  In particular, pairs of coincident coordinates are inactive
at positive diameter.

The following notation specializes the active-branch terminology for
finite maxima \cite[Proposition~2.3.12]{Clarke1990} to the diameter.

\begin{definition}[Diameter pairs]
\label{def:active-pairs}
Let \((X,d_X)\) be a metric space and let \(x\in X^N\).
For \(1\leq i<j\leq N\), put
\[
 \ell_{ij}^X:X^N\longrightarrow\R,
 \qquad \ell_{ij}^X(x):=d_X(x_i,x_j).
\]
A \emph{diameter pair} at \(x\)
is a pair realizing \(\diam_N(x)\).  Its indices belong to
\[
 \operatorname{Act}(x):=
 \{ij:1\leq i<j\leq N,\ d_X(x_i,x_j)=\diam_N(x)\}.
\]
For a finite set \(P\subset X\), diameter pairs are defined in the
same way, using distinct points of \(P\).
\end{definition}

For distinct points \(y,y'\) on a complete Riemannian manifold,
\(d_M\) is smooth near \((y,y')\) exactly when
\(y'\notin\operatorname{Cut}_M(y)\)
\cite[Corollary~5.7.11]{Petersen2016}.

\begin{definition}[First-order stationarity for diameter]
\label{def:first-order-stationary}
Let \(M\) be a connected finite-dimensional Riemannian manifold and let
\(x\in M^N\) have positive diameter.  Suppose \(d_M\) is smooth near
each diameter pair of \(x\).  A vector \(V\in T_xM^N\) decreases all
these distances to first order if
\[
 \diff(\ell_{ij}^M)_x[V]<0
 \qquad\text{for every }ij\in\operatorname{Act}(x).
\] In this case we say that $V$ is  a \emph{first-order descent direction} for $x$.
The tuple $x$ is \emph{first-order stationary for diameter} if no such vector
exists.  A nonempty finite set \(P\subset M\) of positive diameter,
with \(d_M\) smooth near its diameter pairs, is \emph{first-order
stationary for diameter} if no family
\((w_y)_{y\in P}\), \(w_y\in T_yM\), satisfies
\[
 \diff(d_M)_{(y,y')}[w_y,w_{y'}]<0
 \qquad\text{for every diameter pair }\{y,y'\}\subset P.
\]
By convention, constant tuples and singleton sets are also first-order
stationary.  Their diameter is already minimal; the differential criteria
above apply only at positive diameter.
\end{definition}

The following comparison applies
\cref{thm:weak-implies-clarke,lem:weak-slope-bilipschitz,prop:weak-slope-finite-smooth-maximum}
to the diameter.  The weak slope uses the maximum product metric specified
in \cref{def:weak-slope-diameter-spectra}.

\begin{proposition}[Comparison of diameter criticality notions]
\label{prop:clarke-first-order}
Let \(M\) be a connected finite-dimensional Riemannian manifold and let
\(N\geq2\). Then for every \(x\in M^N\),
\begin{equation}
 |\diff\diam_N|_{d_{M,\infty}}(x)=0
 \quad\Longrightarrow\quad
 0\in\partial^{\mathrm C}\diam_N(x).
 \label{eq:diameter-weak-clarke-comparison}
\end{equation}
If, in addition, \(x\) has positive diameter and \(d_M\) is smooth
near every diameter pair, then
\begin{equation}
 \partial^{\mathrm C}\diam_N(x)
 =\conv\{\diff(\ell_{ij}^M)_x:ij\in\operatorname{Act}(x)\}.
 \label{eq:clarke-active-convex-hull}
\end{equation}
Under these additional assumptions, the following are equivalent:
\begin{enumerate}[label=\textup{(\roman*)}]
\item \(|\diff\diam_N|_{d_{M,\infty}}(x)=0\);
\item \(0\in\partial^{\mathrm C}\diam_N(x)\);
\item \(x\) has no common first-order descent direction.
\end{enumerate}
If these equivalent stationarity conditions hold, the underlying set
\(\{x_1,\ldots,x_N\}\) is also first-order stationary for diameter.
\end{proposition}

\begin{proof}
The Riemannian product distance
\(d_{M,2}(x,y):=(\sum_{i=1}^N d_M(x_i,y_i)^2)^{1/2}\)
and the maximum product distance satisfy
\[
 d_{M,\infty}\leq d_{M,2}\leq\sqrt N\,d_{M,\infty}.
\]
By \cref{lem:weak-slope-bilipschitz}, a local bi-Lipschitz change
preserves vanishing of the weak slope, though it need not preserve its
numerical value.  Apply the Riemannian form of
Degiovanni--Marzocchi's comparison in \cref{thm:weak-implies-clarke}
to \(\diam_N\) on the product manifold \(M^N\), and then use the
displayed metric bounds.  This gives
\eqref{eq:diameter-weak-clarke-comparison}.

Every inactive distance function has a positive gap below \(\diam_N(x)\).
Continuity and the smoothness assumption give a neighborhood on which
\[
 \diam_N=\max_{ij\in\operatorname{Act}(x)}\ell_{ij}^M
\]
is a finite maximum of smooth functions.  Apply
\cref{prop:weak-slope-finite-smooth-maximum} for the Riemannian product
metric.  This proves \eqref{eq:clarke-active-convex-hull} and the three
equivalences;
\cref{lem:weak-slope-bilipschitz} shows that the weak slope vanishes at exactly the same tuples for both product metrics.

If the support admits a common strict descent family \((w_y)\), set
\(V_i=w_{x_i}\).  A pair of coincident coordinates is inactive at positive diameter,
and each labelled diameter pair gives
\[
 \diff(\ell_{ij}^M)_x[V]
 =\diff(d_M)_{(x_i,x_j)}[w_{x_i},w_{x_j}]<0.
\]
This contradicts (iii), proving the support assertion.
\end{proof}

In particular, an interval containing no Clarke critical diameter value
contains no weak-slope stationary value.  We denote these sets as follows.

\begin{definition}[Clarke diameter spectra]
\label{def:clarke-diameter-spectra}
For a connected finite-dimensional Riemannian manifold and \(N\geq2\),
set
\begin{equation}
 \begin{aligned}
 \Sigma_N^{\specC}(M)
 &:=\{\diam_N(x):x\in M^N,\ 0\in\partial^{\mathrm C}\diam_N(x)\},\\
 \Sigma^{\specC}(M)&:=\bigcup_{N\geq2}\Sigma_N^{\specC}(M).
 \end{aligned}
 \label{eq:clarke-diameter-spectra}
\end{equation}
We call an element of $\Sigma^{\specC}(M)$ a \emph{Clarke critical value of diameter}.
\end{definition}

At diameter zero, weak-slope stationarity follows from global
minimality, Clarke criticality from \cref{prop:clarke-first-order},
and first-order stationarity from the convention in
\cref{def:first-order-stationary}.  Thus both kinds of spectra include zero.
The implication from weak-slope stationarity to Clarke criticality in
\cref{prop:clarke-first-order} also gives the spectral inclusions
\begin{equation}
 \Sigma_N^{\specWS}(M)\subseteq\Sigma_N^{\specC}(M),
 \qquad
 \Sigma^{\specWS}(M)\subseteq\Sigma^{\specC}(M).
 \label{eq:weak-clarke-spectrum-inclusion}
\end{equation}
These inclusions hold even when an endpoint of a diameter pair lies in the cut locus of the other.

\begin{remark}[Why the two spectra are distinguished]
\label{rem:weak-clarke-nonequivalence}
For general locally Lipschitz functions, Clarke criticality need not imply
weak-slope stationarity.  Ribarska, Tsachev, and Krastanov give a
piecewise-linear example on \(\R^2\) with a Clarke-critical origin and
positive weak slope \cite[Example~1.4]{RibarskaTsachevKrastanov1996}.
Convexifying nearby gradients can therefore create criticality even when a
continuous nonsmooth deformation decreases the function.
For diameter, \cref{prop:clarke-first-order} proves the converse when
\(d_M\) is smooth near each diameter pair.
\end{remark}

\subsection{The Clarke spectrum for a fixed number of labels}
\label{subsec:fixed-label-clarke-spectrum}

We prove that, for each fixed number of labels, the Clarke critical
diameter values form a compact set of Hausdorff dimension zero.

Throughout this subsection, let \((M,g)\) be a closed connected
\(C^\infty\) Riemannian manifold, and write \(d_M\) for the distance
induced by \(g\).  If \(\dim M=0\), then \(M\) is a singleton and both
kinds of spectra equal \(\{0\}\).  We may therefore assume positive
dimension in the proofs.

\begin{theorem}[Clarke--Sard theorem for labelled diameter]
\label{thm:fixed-label-clarke-spectrum}
Let \((M,g)\) be a closed connected \(C^\infty\) Riemannian manifold.
For every \(N\geq2\), the set \(\Sigma_N^{\specC}(M)\) is compact and
has Hausdorff dimension zero.  These conclusions hold even when an endpoint
of a diameter pair lies in the cut locus of the other.
These spectra are nested:
\[
 \Sigma_q^{\specC}(M)\subseteq\Sigma_N^{\specC}(M)
 \qquad(2\leq q\leq N).
\]
\end{theorem}

We use that Lipschitz maps do not increase Hausdorff dimension and
that this dimension is stable under countable unions
\cite[Section~2.1]{Falconer1997}.

To prove that the critical diameter values in
\cref{thm:fixed-label-clarke-spectrum} have Hausdorff dimension zero,
we first express squared distance as a minimum of a smooth function
over a fixed compact manifold.  This allows us to apply
\cref{lem:marginal-maximum-sard} even at the cut locus.  We obtain
this representation by the classical broken-geodesic construction for the
energy \cite[\S16]{Milnor1963}, using a smooth cutoff to extend the energy
to all vertex configurations without changing its minimum.
The auxiliary variable in \cref{lem:smooth-distance-marginal} records
the intermediate points of the broken path.  Each minimizing parameter
records one minimizing geodesic, so the derivatives in the two endpoints
come from that same geodesic.

\begin{lemma}[A finite-dimensional smooth energy representation of squared
distance]
\label{lem:smooth-distance-marginal}
There are a compact boundaryless manifold \(K\) and a smooth function
\(\Phi:M\times M\times K\to\mathbb R\) such that
\begin{equation}
 d_M(x,x')^2=\min_{z\in K}\Phi(x,x',z).
 \label{eq:smooth-distance-marginal}
\end{equation}
Every minimizer records an equal subdivision of one minimizing geodesic
from \(x\) to \(x'\).
\end{lemma}

\par\smallskip
\begin{proof}
Choose
\[
 0<r_1<r_2<\operatorname{inj}(M),
 \]
a smooth cutoff \(\chi:M^2\to[0,1]\) equal to one on
\(\{d_M\leq r_1\}\) and supported in \(\{d_M<r_2\}\), and a constant
\(C>\diam(M)^2\).  The function
\[
 \psi(x,x'):=\chi(x,x')d_M(x,x')^2+(1-\chi(x,x'))C
\]
is globally smooth: the first term extends by zero because
\(\operatorname{supp}(\chi)\subset\{d_M<r_2\}\), where squared distance is
smooth.  It satisfies \(\psi\geq d_M^2\), and has equality exactly where
\(\chi=1\).  Choose an integer \(L\geq2\) with
\(\tfrac{\diam(M)}{L}<r_1\), set
\(K:=M^{L-1}\), put \(z_0:=x\) and \(z_L:=x'\), and define
\[
 \Phi(x,x',z):=L\sum_{\nu=1}^{L}\psi(z_{\nu-1},z_\nu).
\]
For every broken chain,
\[
 d_M(x,x')^2
 \leq\left(\sum_\nu d_M(z_{\nu-1},z_\nu)\right)^2
 \leq L\sum_\nu d_M(z_{\nu-1},z_\nu)^2
 \leq\Phi(x,x',z).
\]
Equal subdivision of a minimizing geodesic gives equality.  Conversely,
equality forces equal segment lengths and equality in every cutoff term.
Every segment then has length less than \(r_2\), so consecutive points
are joined by unique minimizing short segments.  Their concatenation
has length \(d_M(x,x')\), hence is a minimizing geodesic with its
equal subdivision.  Write
\[
 \diff_x\Phi(x,x',z)
 :=\diff\bigl(\Phi(\,\cdot\,,x',z)\bigr)_x,
 \qquad
 \diff_{x'}\Phi(x,x',z)
 :=\diff\bigl(\Phi(x,\,\cdot\,,z)\bigr)_{x'}
\]
for the partial differentials in the two endpoint variables.
Each segment of a minimizing chain has length strictly less than
\(r_1\), so the cutoff equals one near that segment's endpoint pair.
In particular, when \(x\neq x'\), the endpoint differentials at a
minimizing parameter \(z\) come from the same geodesic.
Put \(r:=d_M(x,x')\).  For its unit-speed parametrization
\(\gamma:[0,r]\to M\), first variation
\cite[Lemma~5.4.2]{Petersen2016} gives
\[
 \diff_x\Phi(x,x',z)=-2r\dot\gamma(0)^\flat,
 \qquad
 \diff_{x'}\Phi(x,x',z)=2r\dot\gamma(r)^\flat.
\]
Here \(v^\flat:=g_y(v,\,\cdot\,)\in T_y^*M\) denotes the covector
metrically dual to a tangent vector \(v\in T_yM\).
\end{proof}

The critical values of a smooth real-valued function on a finite-dimensional
second-countable smooth manifold without boundary have zero
\(p\)-dimensional Hausdorff measure for every \(p>0\), by a classical
result of Sard \cite[Corollary on p.~169]{Sard1965}.
To apply this result to the energy representation in
\cref{lem:smooth-distance-marginal}, we must turn a Clarke-critical
value into a critical value of a smooth function.
At a Clarke-critical point, zero is a convex combination of the
differentials of the relevant minimizing energies.  We treat the
minimizing parameters and the coefficients in this combination as
additional variables in a smooth weighted sum.  Its derivative in the
original variables vanishes by the convex-combination identity;
its derivatives in the minimizing parameters vanish at interior minima;
and its derivative along variations of the weights with sum zero
vanishes because the active energies have equal values.
This adapts the construction in
\cite[proof of Theorem~1]{BarbetEtAl2016} to a finite maximum of minima.
\Cref{lem:marginal-maximum-sard} makes this reduction precise.

\begin{lemma}[Critical values of a finite maximum of minima]
\label{lem:marginal-maximum-sard}
Let \(U\) be a second-countable smooth manifold and \(\mathcal I\) a
nonempty finite set.  For \(e\in\mathcal I\), let \(K_e\) be a
nonempty compact smooth
manifold, possibly with boundary, and let
\(\psi_e\in C^\infty(U\times K_e)\).
Assume that, for every \(x\in U\), every minimizer of
\(\psi_e(x,\cdot)\) lies in \(\operatorname{int}K_e\).
Define
\[
 g_e(x):=\min_{y\in K_e}\psi_e(x,y),
 \qquad G(x):=\max_{e\in\mathcal I}g_e(x).
\]
Then every \(g_e\), and hence \(G\), is locally Lipschitz, and the set
\[
 \bigl\{G(x):x\in U,\ 0\in\partial^{\mathrm C}G(x)\bigr\}
\]
has Hausdorff dimension zero.
\end{lemma}

\begin{proof}
Choose a coordinate neighborhood \(V\Subset U\) whose coordinate image
is convex.  Compactness of \(\overline V\times K_e\) gives a uniform bound for
\(\lVert\diff_x\psi_e\rVert\).  Thus the functions
\(\psi_e(\,\cdot\,,y)\), \(y\in K_e\), have a common local Lipschitz
constant.  Taking the minimum preserves this bound for each
\(g_e\).  The largest of the finitely many bounds is a local
Lipschitz constant for \(G\), proving the asserted local Lipschitz
continuity.

Suppose \(0\in\partial^{\mathrm C}G(x)\).  We first express zero as
a convex combination of
differentials \(\diff_x\psi_e(x,y)\), where \(y\) minimizes \(\psi_e(x,\cdot)\) and \(g_e(x)=G(x)\).  Apply Clarke's compact-parameter
maximum rule \cite[Theorem~2.8.2]{Clarke1990} to
\(-g_e(x)=\max_{y\in K_e}(-\psi_e(x,y))\).
Its hypotheses hold because \(K_e\) is compact and both \(\psi_e\)
and \(\diff_x\psi_e\) are continuous, with the common local Lipschitz
bound just established.  The negation rule
\cite[Proposition~2.3.1]{Clarke1990}, followed by the finite-maximum rule
\cite[Proposition~2.3.12]{Clarke1990}, gives
\[
 \partial^{\mathrm C}g_e(x)
 \subseteq
 \conv\left\{\diff_x\psi_e(x,y):
      y\in\operatorname*{argmin}\psi_e(x,\cdot)\right\}
\]
and
\[
 \partial^{\mathrm C}G(x)
 \subseteq
 \conv\bigcup_{e\in I(x)}\partial^{\mathrm C}g_e(x),
 \qquad
 I(x):=\{e\in\mathcal I:g_e(x)=G(x)\}.
\]
Consequently, zero belongs to
\[
 \operatorname{conv}
 \bigcup_{e\in I(x)}
 \operatorname{conv}
 \left\{\diff_x\psi_e(x,y):
      y\in\operatorname*{argmin}\psi_e(x,\cdot)\right\},
\]
and Carath\'eodory's theorem
\cite[Theorem~17.1]{Rockafellar1970} gives
an integer \(1\leq\ell\leq\dim(U)+1\), active indices \(e_j\), interior
minimizers \(y_j\), and positive numbers \(\lambda_j\) summing to one, such
that
\begin{equation}
 \sum_{j=1}^{\ell}\lambda_j\diff_x\psi_{e_j}(x,y_j)=0.
 \label{eq:marginal-sard-balance}
\end{equation}
We now regard the minimizing parameters and the coefficients in
\eqref{eq:marginal-sard-balance} as independent variables, together
with the base point.
For each ordered tuple \(\mathbf e=(e_1,\ldots,e_\ell)\), consider the
smooth function
\[
 H_{\mathbf e}(z,y_1,\ldots,y_\ell,\lambda)
 :=\sum_{j=1}^{\ell}\lambda_j\psi_{e_j}(z,y_j)
\]
on
\(U\times\prod_j\operatorname{int}K_{e_j}\times
\operatorname{int}\Delta^{\ell-1}\), where the last factor is the
relative interior of the probability simplex
\(\{\lambda\in[0,1]^\ell:\sum_j\lambda_j=1\}\).
At the point with base coordinate \(x\), the selected minimizers
\(y_j\), and the selected coefficients \(\lambda_j\), the
\(z\)-derivative of \(H_{\mathbf e}\) vanishes by
\eqref{eq:marginal-sard-balance}.  Its \(y_j\)-derivatives vanish
because the minimizers are interior.  Its simplex derivative vanishes
because all active values equal \(G(x)\) and variations tangent to
the simplex have coordinate sum zero.  Moreover, the value of
\(H_{\mathbf e}\) there is \(G(x)\).

Consequently,
\begin{equation}
 \{G(x):x\in U,\ 0\in\partial^{\mathrm C}G(x)\}
 \subseteq
 \bigcup_{\ell=1}^{\dim(U)+1}
 \ \bigcup_{\mathbf e\in\mathcal I^\ell}
 H_{\mathbf e}\bigl(\{p:\diff H_{\mathbf e}(p)=0\}\bigr).
 \label{eq:marginal-sard-critical-image-inclusion}
\end{equation}

Each domain of \(H_{\mathbf e}\) is a finite-dimensional,
second-countable smooth manifold without boundary.  Applying Sard's
corollary \cite[Corollary on p.~169]{Sard1965} to the set where
\(\diff H_{\mathbf e}=0\) shows that its image has zero
\(p\)-dimensional Hausdorff measure for every \(p>0\).
The union in \eqref{eq:marginal-sard-critical-image-inclusion} is
finite, so the same conclusion holds for the Clarke critical values
of \(G\).
\end{proof}

\phantomsection
\label{proof:fixed-label-clarke-sard}
\begin{proof}[\normalfont\bfseries Proof of \cref{thm:fixed-label-clarke-sard-intro} and \cref{thm:fixed-label-clarke-spectrum}]
We first prove the Hausdorff-dimension assertion for squared diameter
and then transfer it to diameter away from zero.  Put
\(G_N:=\diam_N^2\).  With one copy of the compact parameter \(K\) from
\cref{lem:smooth-distance-marginal} for each pair of labels,
\eqref{eq:smooth-distance-marginal} gives
\[
 G_N(x)=\max_{i<j}\min_{z\in K}\Phi(x_i,x_j,z).
\]
By \cref{lem:marginal-maximum-sard}, the Clarke critical values of \(G_N\)
form a set of Hausdorff dimension zero.  At a tuple with
\(r:=\diam_N(x)>0\), Clarke's scalar-composition chain rule
\cite[Theorem~2.3.9]{Clarke1990} applied to the square gives
\[
 \partial^{\mathrm C}G_N(x)
 \subseteq2r\,\partial^{\mathrm C}\diam_N(x).
\]
On a neighborhood of \(r^2>0\), the square root is \(C^1\).  Applying the
same rule to \(\diam_N=\sqrt{G_N}\) gives the reverse inclusion after
multiplication by \(2r\).  Hence
\[
 \partial^{\mathrm C}G_N(x)=2r\,\partial^{\mathrm C}\diam_N(x).
\]
For each integer \(j\geq1\), the square-root map is Lipschitz on
\([j^{-2},\diam(M)^2]\) whenever this interval is nonempty.  The positive
Clarke critical values of \(\diam_N\) are therefore covered by countably
many Lipschitz images of subsets of the critical-value set of \(G_N\).
Each image has Hausdorff dimension zero.  Countable stability and adjoining
the singleton \(0\) prove Hausdorff dimension zero.

It remains to prove compactness and nesting.  The Clarke
subdifferential has closed graph in local cotangent-bundle
trivializations \cite[Proposition~2.1.5(d)]{Clarke1990}.  Its zero locus
in the compact manifold \(M^N\) is therefore closed, and its image under
\(\diam_N\) is exactly \(\Sigma_N^{\specC}(M)\).  This proves compactness.

For nesting, we show that repeating labels preserves Clarke
criticality.  Let
\(\pi:\{1,\ldots,N\}\twoheadrightarrow\{1,\ldots,q\}\) and let
\(\iota_\pi:M^q\to M^N\) be defined by
\[
 (\iota_\pi(x))_j:=x_{\pi(j)}.
\]
Thus it repeats the \(i\)-th label \(k_i:=|\pi^{-1}(i)|\) times.  Since
\(\diam_q=\diam_N\circ\iota_\pi\), Clarke's smooth inner-map chain rule
\cite[Theorem~2.3.10]{Clarke1990} gives
\[
 \partial^{\mathrm C}\diam_q(x)
 \subseteq
 \diff(\iota_\pi)_x^*\partial^{\mathrm C}\diam_N(\iota_\pi x).
\]
If zero belongs to the left side, choose
\(\xi\in\partial^{\mathrm C}\diam_N(\iota_\pi x)\) with
\(\diff(\iota_\pi)_x^*\xi=0\).  Zero pullback does not make \(\xi\)
zero.  We obtain the zero ambient covector by averaging over permutations
of the repeated copies.  Let
\[
 \mathfrak G_\pi
 :=\prod_{i=1}^q\mathfrak S_{\pi^{-1}(i)}
 \leq\mathfrak S_N
\]
be the subgroup that permutes the repeated copies within each fiber of
\(\pi\).  Every \(g\in\mathfrak G_\pi\) fixes \(\iota_\pi x\).  Since \(\diam_N\) is invariant under label permutations, applying
the smooth inner-map chain rule to a permutation and its inverse
shows that permuting the cotangent blocks of \(\xi\) by \(g\)
preserves membership in the subdifferential.  Thus
\[
 g\cdot\xi\in\partial^{\mathrm C}\diam_N(\iota_\pi x).
\]
Moreover, \(\diff(\iota_\pi)_x^*\) adds the cotangent blocks over each
fiber of \(\pi\), so
\[
 \diff(\iota_\pi)_x^*(g\cdot\xi)
 =\diff(\iota_\pi)_x^*\xi=0.
\]
Consequently the average
\[
 \overline\xi
 :=\frac{1}{\lvert\mathfrak G_\pi\rvert}
   \sum_{g\in\mathfrak G_\pi}g\cdot\xi
\]
still belongs to
\(\partial^{\mathrm C}\diam_N(\iota_\pi x)\) and still has zero pullback.
Its \(k_i\) blocks over the \(i\)-th label are equal, while their sum is
zero.  Each block of \(\overline\xi\) is therefore zero.  Thus
\(\overline\xi=0\), and the repeated tuple is Clarke critical.

The analytic finiteness assertion of
\cref{thm:fixed-label-clarke-sard-intro} follows from
\cref{cor:analytic-definable-fixed-label-finiteness}, proved next.
\end{proof}

\subsection{\texorpdfstring{ Finiteness for analytic metrics and a counterexample for smooth metrics.}{Finiteness for analytic metrics and a counterexample for smooth metrics.}}

For real-analytic metrics, we prove that the Clarke and weak-slope
critical diameter spectra are finite for each fixed number of labels.
This applies in particular to the round spheres.

\begin{corollary}[Finiteness for analytic metrics]
\label{cor:analytic-definable-fixed-label-finiteness}
Let \((M,g)\) be a closed connected real-analytic Riemannian manifold.
For every fixed \(N\geq2\), both \(\Sigma_N^{\specC}(M)\) and
\(\Sigma_N^{\specWS}(M)\) are finite.
\end{corollary}

\begin{proof}
A semianalytic set is locally described by finite Boolean combinations
of analytic equalities and inequalities; a subanalytic set is locally a
projection of a relatively compact semianalytic set
\cite[Definitions~2.1 and~3.1]{BierstoneMilman1988}.
We use \(\R_{\mathrm{an}}\), the structure generated by restrictions
of analytic functions to compact boxes on whose neighborhoods they
are analytic.  Its definable sets, also called globally subanalytic
sets, are obtained from these graphs and polynomial equalities and
inequalities by finite Boolean operations, Cartesian products, and
coordinate projections.
This structure is o-minimal: its definable subsets of \(\mathbb R\)
are finite unions of points and intervals.  A function is definable
when its graph belongs to this class; for a function taking the value
\(+\infty\), we use the finite graph
\(\{(x,f(x)):f(x)<+\infty\}\)
\cite[Section~4, Definitions~6--7 and Remark~5(ii)]{BolteEtAl2007}.
In our application, the graphs on compact coordinate boxes will be
globally subanalytic, as verified below.

Bolte--Daniilidis--Lewis--Shiota's theorem gives finitely many Clarke
critical values for a lower-semicontinuous definable function on
Euclidean space, allowing the value \(+\infty\)
\cite[Corollary~9(ii)]{BolteEtAl2007}.
If \(M\) is a singleton, both spectra equal
\(\{0\}\).

Since \(M\) is a closed connected real-analytic Riemannian manifold, \cite[Theorem~3.5.2]{Tamm1981} implies that the two-point distance \(d_M:M\times M\to\R\) is subanalytic.
Analytic coordinate projections and finite maxima preserve
subanalyticity, so
\[
 \diam_N(x_1,\ldots,x_N)=\max_{1\leq i<j\leq N}d_M(x_i,x_j)
\]
is subanalytic. To apply \cite[Corollary~9(ii)]{BolteEtAl2007} in the Riemannian case, choose finitely many
analytic charts \(\phi_a:U_a\to\R^{N\dim M}\)
and closed Euclidean boxes \(K_a\Subset\phi_a(U_a)\) such that
\(\phi_a^{-1}(\operatorname{int}K_a)\) cover \(M^N\).
Extend \(\diam_N\circ\phi_a^{-1}|_{K_a}\) by \(+\infty\) outside
\(K_a\).  Continuity of diameter and closedness of \(K_a\) make
this extension lower semicontinuous, including on the boundary of
the box.  Its finite graph is compact and subanalytic.  To see that
the extension is definable in \(\R_{\mathrm{an}}\),
choose finitely many local representations of the compact graph as projections of relatively compact semianalytic sets. Cover the closures of these sets by sufficiently small coordinate boxes on neighborhoods of which the defining functions are analytic. Their restrictions to the boxes are restricted analytic functions. Finite Boolean operations and projection preserve definability
\cite[Section~4, Definition~6]{BolteEtAl2007}.  The cited finiteness
theorem therefore gives finitely many Clarke critical values for each
extension.  On box interiors it agrees locally with the
original function, so their Clarke subdifferentials coincide there.
Chart invariance and the finite cover prove
finiteness of \(\Sigma_N^{\specC}(M)\); the weak-slope assertion follows
from \eqref{eq:weak-clarke-spectrum-inclusion}.
\end{proof}

The finiteness conclusion fails for smooth metrics, even on the
two-sphere.  The example below uses geodesic parallels whose lengths
vary over a Cantor set.  With one endpoint fixed, approaching the
opposite point of a parallel along its two minimizing semicircles
gives opposite limiting pairs
of endpoint covectors.  Their midpoint is zero, so the endpoint
pair is Clarke critical.  A flat smooth perturbation makes the
critical values distinct.

\begin{example}[Smoothness does not imply finiteness]
\label{prop:cantor-clarke-spectrum}
There is a smooth nonanalytic rotational metric \(g\) on \(\Sph^2\)
such that, for \(M=(\Sph^2,g)\), every \(\Sigma_N^{\specC}(M)\),
\(N\geq2\), contains a Cantor set.  The realizing critical tuples are
not local minima of diameter.
\end{example}

\par\smallskip
This example shows that the conclusion of
\cref{thm:fixed-label-clarke-spectrum} cannot be strengthened from
Hausdorff dimension zero to countability, already for \(N=2\).
The construction establishes Clarke
criticality; it does not establish weak-slope stationarity or a change in
Vietoris--Rips homotopy type.
The construction and the global minimizing-length estimate are proved in
\cref{app:cantor-clarke-spectrum}.

\subsection{Initial gaps and canonical Vietoris--Rips maps}
\label{subsec:all-label-clarke-consequences}

The critical-value theorem concerns each fixed number of labels.  To
obtain an initial deformation interval for the full Vietoris--Rips
filtration, we need a gap valid simultaneously for every label number
\(N\geq2\).  Riemannian
convexity supplies such a gap by moving all points toward one center.

\begin{definition}[Strong convexity and global convexity radius
{\cite[Section~1]{Dibble2017}}]
\label{def:global-convexity-radius}
An open subset \(U\subseteq M\) is \emph{strongly convex} if every two
points of \(U\) are joined by a unique minimizing geodesic in \(M\), and
that geodesic is contained in \(U\).  The \emph{global convexity radius} is
\[
 \convRad(M):=\inf_{x\in M}\sup\{R>0:
 B_r(x)\text{ is strongly convex for every }0<r<R\}.
\]
\end{definition}

Since \(M\) is closed, its global convexity radius is positive
\cite[Section~1]{Dibble2017}.
Following \cite[Section~3, pp.~178--179]{Hausmann1995}, let
\(r_H(M)\) be the supremum of the numbers \(a>0\) satisfying all
three conditions below:
\begin{enumerate}[label=\textup{(\alph*)}]
\item For every \(x,y\in M\) with \(d_M(x,y)<2a\), there is a
      unique minimizing geodesic from \(x\) to \(y\).
\item If \(x,y,u\in M\) satisfy
      \(d_M(x,y)<a\), \(d_M(u,x)<a\), and \(d_M(u,y)<a\), then
      every point \(z\) on the minimizing geodesic from \(x\) to \(y\)
      satisfies
      \[
       d_M(u,z)\leq\max\{d_M(u,x),d_M(u,y)\}.
      \]
\item If \(\gamma\) and \(\eta\) are unit-speed geodesics with
      \(\gamma(0)=\eta(0)\), then, for every \(0\leq s,s'<a\)
      and \(0\leq t\leq1\), whenever the indicated parameters lie in
      their domains,
      \[
       d_M\bigl(\gamma(ts),\eta(ts')\bigr)
       \leq d_M\bigl(\gamma(s),\eta(s')\bigr).
      \]
\end{enumerate}
We prove that \(r_H(M)=\convRad(M)\) and that no positive Clarke
critical diameter is smaller than this radius. To our knowledge, this identification of the two radii has not previously been stated explicitly.
\begin{proposition}[A uniform initial gap]
\label{prop:uniform-convexity-gap}
Let \((M,g)\) be a closed connected Riemannian manifold.  Then
\[
 r_H(M)=\convRad(M),
 \qquad
 \Sigma^{\specC}(M)\cap(0,\convRad(M))=\varnothing.
\]
\end{proposition}

\begin{proof}
If \(M\) is a singleton, both radii are infinite and the spectrum is
\(\{0\}\), so the assertions hold.  Assume \(\dim M>0\).
We first establish the gap and then compare the two radii.
Fix \(N\geq2\) and \(x\in M^N\) with
\(0<D:=\diam_N(x)<\convRad(M)\), and choose
\(D<R<\convRad(M)\).  We move every label down the squared-distance
function
\[
 h(z):=\tfrac12d_M(x_1,z)^2.
\]

We first justify the smoothness and Hessian positivity needed for this
choice of function.  In Dibble's notation, \(r_f(p)\) is the infimum of positive times at
which a nonzero normal Jacobi field along a unit-speed geodesic from
\(p\), initially zero, has vanishing derivative of its norm;
\(r_f(M)=\inf_p r_f(p)\).
An empty defining set has infimum infinity.
Dibble's global convexity-radius formula, in the form
\cite[Corollary~1.5]{Xu2017}, gives
\begin{equation}
 \convRad(M)=\min\{r_f(M),\tfrac12\operatorname{inj}(M)\}.
 \label{eq:convexity-focal-injectivity}
\end{equation}
Consequently, \(R<r_f(x_1)\) and
\(2R<\operatorname{inj}(M)\leq\operatorname{inj}(x_1)\).
The exponential map is therefore injective on the radius-\(R\) tangent
ball, and \cite[Lemma~2.3]{Dibble2017} gives a positive definite
Hessian for \(h\) on the strongly convex ball \(B_R(x_1)\).
The same global injectivity
bound shows that every diameter pair is off the cut locus.

All labels lie in this ball.
Give label \(i\) the velocity \(v_i=-\nabla h(x_i)\).  The unique
unit-speed minimizing segment \(\gamma_{ij}:[0,D]\to M\) of a diameter pair stays inside the ball.  First variation
\cite[Lemma~5.4.2]{Petersen2016} gives
\begin{align*}
 \diff(\ell_{ij}^M)_x[v]
 &=-\langle v_i,\dot\gamma_{ij}(0)\rangle
   +\langle v_j,\dot\gamma_{ij}(D)\rangle\\
 &=(h\circ\gamma_{ij})'(0)-(h\circ\gamma_{ij})'(D)\\
 &=-\int_0^D\operatorname{Hess}h
       (\dot\gamma_{ij}(s),\dot\gamma_{ij}(s))\,\diff s<0.
\end{align*}
Thus the same velocity strictly decreases every active distance.
\Cref{prop:clarke-first-order} excludes Clarke criticality, independently
of the number of labels.

We now identify Hausmann's radius using the same convexity properties.
For \(a<\convRad(M)\), the bound
\(2a<\operatorname{inj}(M)\) gives his condition (a), and strong
convexity of balls gives condition (b).  The Hessian positivity above
holds for \(h_p=\tfrac12d_M(p,\cdot)^2\) on every \(B_a(p)\).
Its negative gradient flow is the radial contraction
\[
 F_\tau(\exp_p v)=\exp_p(e^{-\tau}v),
 \qquad |v|<a,\quad \tau\geq0.
\]
The first variation calculation above, with \(h=h_p\), shows that
the distance between two flow lines is nonincreasing.
Taking \(t=e^{-\tau}\), and then using continuity at \(t=0\), gives
condition (c).  Hence
\(\convRad(M)\leq r_H(M)\).

Conversely, suppose \(a\) satisfies Hausmann's conditions.
Condition (a) implies \(2a\leq\operatorname{inj}(M)\): cut points
joined by more than one minimizing geodesic are dense in each cut
locus; see \cite[proof of Proposition~2.1]{Xu2017}.
For a variation through radial unit-speed geodesics, apply condition
(c) with \(s=s'=v\) and \(t=u/v\), where \(0<u<v<a\).
Dividing by the absolute value of the variation parameter and taking
its limit gives
\[
 |J(u)|\leq|J(v)|\qquad(0<u<v<a)
\]
for every normal Jacobi field \(J\) with \(J(0)=0\).
In Xu's notation, this says
\(a\leq\operatorname{foc}_e(p)\) for every \(p\), where
the extended focal radius \(\operatorname{foc}_e(p)\) is the
supremum of radii on which all these norms are nondecreasing.
\cite[Lemma~2.5]{Xu2017} gives
\(\inf_p\operatorname{foc}_e(p)=r_f(M)\).
Together with \eqref{eq:convexity-focal-injectivity}, this yields
\[
 a\leq\min\{r_f(M),
                  \tfrac12\operatorname{inj}(M)\}
   =\convRad(M).
\]
Taking the supremum over \(a\) proves the reverse inequality.
\end{proof}

An interval containing no Clarke critical diameter values contains no
weak-slope stationary values by \eqref{eq:weak-clarke-spectrum-inclusion}.
Thus \cref{thm:metric-stationary-chamber-intro} and
\cref{cor:fixed-label-weak-slope-connectivity} give the
following consequences for
the canonical Vietoris--Rips inclusion.

\begin{corollary}[Clarke gaps and canonical Vietoris--Rips maps]
\label{cor:all-label-clarke-chambers}
Let \(M\) be a closed connected finite-dimensional Riemannian manifold,
and let \(0<r<s\).
\begin{enumerate}[label=\textup{(\roman*)}]
\item If \([r,s)\cap\Sigma^{\specC}(M)=\varnothing\), the canonical
      inclusion \(\VR{M}{r}\to\VR{M}{s}\) is a homotopy equivalence.
\item For each \(N\geq2\), if \(s\leq\diam(M)\) and
      \([r,s)\cap\Sigma_N^{\specC}(M)=\varnothing\), then
      the canonical inclusion is \((N-1)\)-connected.
\end{enumerate}
\end{corollary}

\begin{proof}
Use \eqref{eq:weak-clarke-spectrum-inclusion} and apply \cref{thm:metric-stationary-chamber-intro} for (i)
and \cref{cor:fixed-label-weak-slope-connectivity} for (ii).
\end{proof}

For a closed connected finite-dimensional \(C^\infty\) Riemannian
manifold \(M\), define
\[
 r_{\mathrm C}(M):=
 \inf\bigl(\Sigma^{\specC}(M)\setminus\{0\}\bigr),
 \qquad \inf\varnothing:=\infty.
\]
This is the lower endpoint of the positive Clarke diameter spectrum;
the definition does not require the infimum to be attained.
If \(\dim M>0\), then
\[
 0<\convRad(M)\leq r_{\mathrm C}(M)\leq\diam(M).
\]
The lower bound follows from \cref{prop:uniform-convexity-gap}.
For the upper bound, a diameter-realizing pair is a global maximum
of the locally Lipschitz function \(\diam_2\), hence is Clarke
critical.  If \(M\) is a singleton, then \(r_{\mathrm C}(M)=\infty\).
All scale parameters below are finite.

\begin{corollary}[Improvement of Hausmann's Theorem: Vietoris--Rips stability up to \(r_{\mathrm C}(M)\)]
\label{cor:convexity-radius-rips-stability}
Let \(M\) be a closed connected finite-dimensional \(C^\infty\)
Riemannian manifold.  For every \(0<r<s\leq r_{\mathrm C}(M)\), the canonical
inclusion
\[
 \VR{M}{r}\longrightarrow\VR{M}{s}
\]
is a homotopy equivalence.  Moreover,
\(\VR{M}{t}\simeq M\) for every \(0<t\leq r_{\mathrm C}(M)\).
\end{corollary}

\begin{proof}
The singleton case is immediate.  Otherwise, for
\(0<r<s\leq r_{\mathrm C}(M)\), the definition gives
\[
 [r,s)\cap\Sigma^{\specC}(M)=\varnothing.
\]
Hence \cref{cor:all-label-clarke-chambers}\textup{(i)} gives the
canonical homotopy equivalence, including when \(s=r_{\mathrm C}(M)\).
Only the half-open interval must avoid the spectrum, so no attainment
or closedness assumption on the all-label spectrum is needed.
For any \(0<t\leq r_{\mathrm C}(M)\), choose \(0<r_0<t\) sufficiently small
that Hausmann's theorem \cite[Theorem~3.5]{Hausmann1995} gives
\(\VR{M}{r_0}\simeq M\).  The canonical inclusion from scale \(r_0\)
to scale \(t\) is a homotopy equivalence by the first assertion,
so \(\VR{M}{t}\simeq M\).
\end{proof}

\begin{remark}
Since \(r_H(M)=\convRad(M)\leq r_{\mathrm C}(M)\),
\cref{cor:convexity-radius-rips-stability} includes stability up to
the convexity radius and extends the guaranteed range whenever this
inequality is strict.  Hausmann's small-scale theorem identifies the
initial homotopy type, and the absence of Clarke critical diameter
values extends that identification through \(r_{\mathrm C}(M)\).
This does not assert that \(r_{\mathrm C}(M)\) is the first actual
Vietoris--Rips transition: stationarity alone does not establish
sharpness of the bound.

For the unit round sphere \(\Sph^m\), \(m\geq1\), one has
\(\convRad(\Sph^m)=\pi/2\), whereas
\cref{thm:first-spherical-clarke-value} gives
\[
 r_{\mathrm C}(\Sph^m)=\zeta_m
 =\arccos\!\left(-\frac1{m+1}\right)>\frac\pi2.
\]
Consequently, \cref{cor:convexity-radius-rips-stability} gives
\(\VR{\Sph^m}{t}\simeq\Sph^m\) for every \(0<t\leq\zeta_m\),
with canonical homotopy equivalences between any two scales in this
range.  This provides an alternative proof of the optimal spherical
initial-range theorem of Lim--M\'emoli--Okutan
\cite[Theorem~7.1 and Section~A.4.1]{LimMemoliOkutan2024}; see
\Cref{cor:optimal-spherical-initial-chamber}.
Hausmann's theorem supplies the small-scale identification, while the
sharp stationary-diameter bound and our deformation theorem extend it
through \(\zeta_m\).  Sharpness of the spherical endpoint uses the
additional topological argument recalled after that corollary.
\end{remark}
\par

Within Hausmann's original range \(0<r<s<r_H(M)\), the canonical-map
conclusion also follows from his construction: choosing one ordering of
\(M\) at all scales makes his maps \(T_t:\VR{M}{t}\to M\) satisfy
\(T_s\circ i_{r,s}=T_r\), where \(i_{r,s}\) is the canonical inclusion
\cite[Section~3, equation~(3.1) and Theorem~3.5]{Hausmann1995}.

The union of the Clarke critical diameter spectra over all label
numbers has Hausdorff dimension zero.

\begin{corollary}[The all-label spectrum has Hausdorff dimension zero]
\label{cor:all-label-thin-initial}
For a closed connected \(C^\infty\) Riemannian manifold \(M\), the union
\(\Sigma^{\specC}(M)\) has Hausdorff dimension zero.
\end{corollary}

\begin{proof}
It is the countable union of the fixed-label spectra in
\cref{thm:fixed-label-clarke-spectrum}.  Apply countable stability of
Hausdorff dimension.
\end{proof}

By \cref{thm:fixed-label-clarke-spectrum}, each
\(\Sigma_N^{\specC}(M)\) is closed and has empty interior.  Hence
\begin{equation}
 (0,\diam(M))\setminus\Sigma^{\specC}(M)
 \label{eq:generic-clarke-regular-values}
\end{equation}
is a countable intersection of open dense sets and is dense by
Baire's theorem.  This need not give an open interval containing no
critical values:
even a countable set can be dense.
On the round sphere \(\Sph^m\), \(m\geq1\), we will bound the
number of points needed to realize critical values in \((0,b]\) for each
\(b<\zeta_{m-1}\); see \cref{prop:bounded-realizing-supports}.
This bound and \cref{cor:analytic-definable-fixed-label-finiteness} imply
that only finitely many critical diameter values lie in \((0,b]\).

For homology in a fixed degree, only one fixed-label spectrum is
needed.  This gives a stronger regularity conclusion than the
all-label statement alone. The following corollary uses persistence modules induced by Vietoris-Rips filtrations and their barcodes;
for background, see \cite[Section~2.1]{LimMemoliOkutan2024},
\cite[Section~2.1.2]{BalitskiyCoskunuzerMemoli2025},
and references therein.

\begin{corollary}[Homological critical scales and finite barcodes for analytic metrics]
\label{cor:degreewise-persistence-finiteness}
Let \((M,g)\) be a closed connected \(C^\infty\) Riemannian manifold,
let \(p\geq1\), and let \(\mathbb F\) be a field.  The set of
positive homological critical scales in degree \(p\) is a compact
subset of \(\Sigma_{p+2}^{\specC}(M)\) and has Hausdorff dimension
at most zero.
If \(M\) and \(g\) are real analytic, then
\(H_p(\VR{M}{t};\mathbb F)\) is finite-dimensional for every
\(t>0\), and the persistence module
\[
 \bigl\{H_p(\VR{M}{t};\mathbb F)\bigr\}_{t>0}
\]
with its canonical maps has finitely many bars.  Every positive
endpoint of these bars belongs to \(\Sigma_{p+2}^{\specC}(M)\).
\end{corollary}

\begin{proof}
\Cref{prop:homological-change-stationarity} and
\eqref{eq:weak-clarke-spectrum-inclusion} place the positive
homological critical scales in \(\Sigma_{p+2}^{\specC}(M)\).
Their set is closed in \((0,\infty)\), since homological regularity
is an open condition.  By \cref{prop:uniform-convexity-gap}, it
is bounded away from zero.  Compactness and Hausdorff dimension zero
therefore follow from \cref{thm:fixed-label-clarke-spectrum}.
If \(M\) is a singleton, the set is empty and all the stated
homology groups vanish.

Now assume the metric is real analytic.  By
\cref{cor:analytic-definable-fixed-label-finiteness}, the spectrum
\(\Sigma_{p+2}^{\specC}(M)\) is finite.  For each \(t>0\), choose
\(0<a<t\) such that
\[
 [a,t)\cap\Sigma_{p+2}^{\specC}(M)=\varnothing.
\]
The weak-slope comparison and
\cref{cor:fixed-label-weak-slope-connectivity} imply that the
canonical map from scale \(a\) to scale \(t\) is an isomorphism
on \(H_p(-;\mathbb F)\).  This also holds when \(t\) itself is
critical, because the gap condition omits its upper endpoint.

To see that this isomorphism has finite rank, use the finite-net
argument of \cite[Proposition~5.1]{ChazalDeSilvaOudot2014}, here
with the strict convention.  Choose \(\varepsilon>0\) with
\(a+2\varepsilon<t\), a finite \(\varepsilon\)-net \(F\subset M\),
and a map \(\nu:M\to F\) satisfying
\(d_M(x,\nu(x))<\varepsilon\).  The induced simplicial maps
\[
 \VR{M}{a}\longrightarrow\VR{F}{a+2\varepsilon}
 \longrightarrow\VR{M}{t}
\]
have composite contiguous to the canonical inclusion: for each
simplex \(\sigma\) at scale \(a\), the union
\(\sigma\cup\nu(\sigma)\) has diameter less than
\(a+2\varepsilon<t\).  The homology map thus factors through the
homology of a finite complex.  Being both an isomorphism and of
finite rank, it proves finite-dimensionality at scale \(t\).

There are only finitely many positive values in
\(\Sigma_{p+2}^{\specC}(M)\).  If \(c<c'\) are successive such values, the  maps between scales \(r,s\) with \(c<r\leq s\leq c'\) are isomorphisms. The upper endpoint is included because the gap condition uses \([r,s)\).  Above \(\diam(M)\), the complex
is a full simplex and its positive-degree homology vanishes.
Consequently the persistence module is determined by a finite
sequence of finite-dimensional vector spaces.  Its decomposition
into interval modules gives finitely many bars, whose positive
endpoints can occur only at the indicated spectral values.
\end{proof}

Thus the total stationary spectrum can be infinite even though,
for an analytic metric, persistence has only finitely many bars
in each fixed homological degree.  No finiteness across all degrees
or of the collection of full homotopy types is asserted.

The round sphere first gives an exact positive minimum and an explicit
circle spectrum in \cref{sec:spherical-diameter-spectra}.  The bound on
realizing supports in \cref{sec:spherical-stacks} will then locate
the first accumulation point when all label numbers are allowed.

\section{Positive stationary diameters of round spheres}
\label{sec:spherical-diameter-spectra}

We now determine the least positive stationary diameter on a round
sphere.  Both the ambient dimension and the number of labels enter the
answer:
\[
 \min\bigl(\Sigma_N^{\specC}(\Sph^m)\setminus\{0\}\bigr)
 =\min\bigl(\Sigma_N^{\specWS}(\Sph^m)\setminus\{0\}\bigr)
 =\zeta_{\min\{m,N-2\}}.
\]
The two restrictions have the same geometric origin: a centered regular
simplex must fit in the ambient sphere and use no more points than the
available labels.  Summing the equilibrium equations puts the origin in
the convex hull of the endpoints of edges with positive weight.
Jung's inequality then gives the lower bound, and its equality case
identifies the realizing configurations
(\cref{thm:first-spherical-clarke-value}).  Repetitions allow larger
tuples with the same simplex support.

The circle provides an exact model beyond this first value.  Its regular
odd polygons give every positive nonantipodal stationary diameter and its
least number of labels.  The known canonical Vietoris--Rips maps then
show that the connectivity estimate is sharp.  Those polygons also
provide the bases for the accumulation construction in the next section.

\subsection{First-order stationarity and equilibrium stresses}
\label{subsec:spherical-clarke-spectra}

We characterize stationarity at nonantipodal positive diameters by a
weighted balance among the pairs realizing the diameter.

Let \(P=\{x_1,\ldots,x_q\}\subset\Sph^m\), with \(q=|P|\), have diameter
\(D\in(0,\pi)\), and write \(x:=(x_1,\ldots,x_q)\).
Its \emph{diameter graph} has vertex set \(P\) and
an edge \(ij\) precisely when \(d_m(x_i,x_j)=D\).
These are the diameter pairs of \cref{def:active-pairs}.  For each diameter pair \(ij\), the function
\(\ell_{ij}^{\Sph^m}\), defined in \cref{def:active-pairs}, is smooth near \(x\), because \(x_i\) and
\(x_j\) are distinct and nonantipodal.  For
\(V=(V_1,\ldots,V_q)\), with \(V_i\in T_{x_i}\Sph^m\), its differential is
\[
 \diff(\ell_{ij}^{\Sph^m})_x[V]
 =-\frac{V_i\cdot x_j+x_i\cdot V_j}{\sin D}.
\]
By \cref{prop:clarke-first-order}, common strict descent
fails exactly when zero belongs to the convex hull of these
differentials.  Equivalently, a linear combination of them with nonnegative coefficients, not all zero, vanishes.  This is Gordan's theorem of the alternative
\cite[Section~22]{Rockafellar1970}, in the form of the
convex-hull identity \eqref{eq:clarke-active-convex-hull}.
Reading the vanishing combination
one tangent space at a time gives the following equilibrium equations.

\begin{corollary}[Equilibrium stresses on a round sphere]
\label{cor:spherical-gordan-stress}
Let \(P=\{x_1,\ldots,x_q\}\subset\Sph^m\), with \(q=|P|\), have diameter
\(D\in(0,\pi)\).  Then \(P\) is first-order stationary for
diameter if and only if there are symmetric coefficients
\(w_{ij}=w_{ji}\geq0\), not all zero and vanishing on
nonedges of the diameter graph, with \(w_{ii}:=0\), such that
\begin{equation}
 \sum_{j\ne i}w_{ij}(x_j-\cos(D)x_i)=0
 \qquad(1\leq i\leq q).
 \label{eq:spherical-gordan-equilibrium}
\end{equation}
Such a family \(\bigl(w_{ij}\bigr)_{1\leq i,j\leq q}\) is called a
\emph{nonzero nonnegative equilibrium stress}.
For this stress, the \emph{stressed degree} of \(x_i\) is
\(h_i:=\sum_{j\ne i}w_{ij}\), the sum of its incident stress weights.
These degrees satisfy
\begin{equation}
 (1-\cos D)\sum_{i=1}^q h_i x_i=0,
 \label{eq:spherical-stress-weighted-center}
\end{equation}
and hence \(0\in\conv\{x_i:h_i>0\}\).
\end{corollary}

\begin{proof}
Put \(x:=(x_1,\ldots,x_q)\).  Apply the convex-hull criterion in
\cref{prop:clarke-first-order} to the active distance
differentials on \(T_x(\Sph^m)^q\), and normalize the coefficients.
  The gradient of the \(ij\)-branch
at coordinate \(x_i\) is
\[
 -\frac{x_j-\cos(D)x_i}{\sin D}.
\]
The vanishing of a linear combination of the full differentials with nonnegative coefficients, not all zero, is therefore equivalent, coordinate by coordinate, to
\eqref{eq:spherical-gordan-equilibrium}.

Sum these equations over \(i\).  Symmetry of the weights gives
\(\sum_i\sum_{j\ne i}w_{ij}x_j=\sum_jh_jx_j\), proving
\eqref{eq:spherical-stress-weighted-center}.  Since
\[
 \sum_i h_i=2\sum_{i<j}w_{ij}>0
 \quad\text{and}\quad 1-\cos D>0,
\]
division by \(\sum_i h_i\) expresses zero as a convex combination of
the vertices of positive stressed degree.
\end{proof}

The vector \(x_j-\cos(D)x_i\) is the tangent projection of \(x_j\)
at \(x_i\).  Thus \eqref{eq:spherical-gordan-equilibrium} says that
the weighted directions toward diameter neighbors balance at each point.
Summing these equations gives
\eqref{eq:spherical-stress-weighted-center}, which places the origin in
the convex hull and allows us to apply Jung's inequality.

The nonnegativity requirement permits vertices of zero
stressed degree, even among
vertices incident to edges of the diameter graph.  They carry no positive-weight edge and
can be discarded without changing the equilibrium equations on the
remaining points.  The remaining set still has diameter \(D\), because
the stress has at least one positive-weight diameter edge.
Thus, by the corollary, the origin lies in the convex
hull of the vertices with positive stressed degree.

\begin{proposition}[Coincidence below the antipodal scale]
\label{prop:spherical-spectrum-coincidence}
Let \(m\geq1\), \(N\geq2\), and \(x\in(\Sph^m)^N\), with
\(0<\diam_N(x)<\pi\).  The following are equivalent: \(x\) is
weak-slope stationary; \(x\) is Clarke critical; \(x\) has no common
strict first-order descent direction; and its underlying set
\(\{x_1,\ldots,x_N\}\) carries a nonzero nonnegative equilibrium stress.
\end{proposition}

\begin{proof}
Distance is smooth near every diameter pair, so
\cref{prop:clarke-first-order} gives the first three equivalences.

To pass between a labelled tuple and its  underlying set,
we copy descent vectors to repeated labels in one direction and lift
stress coefficients in the other.  A pair of coincident coordinates
is inactive at positive diameter.  If the support admitted a common strict descent
family, copying its vector at each point to every corresponding label
would give a common strict descent direction for \(x\).  Thus
\cref{cor:spherical-gordan-stress} gives a stress on the support whenever
\(x\) is stationary.  Conversely, select one label for each support
point.  A stress on the support lifts to those labels, with zero weights
on all other labelled edges.  This is
a linear relation among the active labelled distance differentials with nonnegative coefficients, not all zero, so
\cref{prop:clarke-first-order} excludes common strict
descent for \(x\).
This proves the pointwise assertion and
\eqref{eq:fixed-spherical-spectrum-coincidence}.  Taking unions over
\(N\), or ordering the points of any stressed finite set, proves
\eqref{eq:all-spherical-spectrum-coincidence}.
\end{proof}

Consequently,
\begin{equation}
 \Sigma_N^{\specWS}(\Sph^m)\cap(0,\pi)
 =\Sigma_N^{\specC}(\Sph^m)\cap(0,\pi),
 \label{eq:fixed-spherical-spectrum-coincidence}
\end{equation}
and the unions over all label numbers have the description
\begin{equation}
\begin{split}
 \Sigma^{\specWS}(\Sph^m)\cap(0,\pi)
 &=\Sigma^{\specC}(\Sph^m)\cap(0,\pi)\\
 &=\left\{\diam(P):
   \begin{array}{l}
    P\subset\Sph^m\text{ is finite and carries a nonzero}\\
    \text{nonnegative equilibrium stress},\quad 0<\diam(P)<\pi
   \end{array}\right\}.
\end{split}
\label{eq:all-spherical-spectrum-coincidence}
\end{equation}

The antipodal scale requires a separate argument because distance is
not smooth at antipodal pairs.

\begin{lemma}[Two-label spectrum of a round sphere]
\label{lem:two-label-spherical-spectrum}
For every \(m\geq1\),
\[
 \Sigma_2^{\specC}(\Sph^m)
 =\Sigma_2^{\specWS}(\Sph^m)=\{0,\pi\}.
\]
\end{lemma}

The smooth case follows by moving one endpoint toward the other.  Near an
antipodal pair, distance has the local form \(\pi-\lVert w\rVert\).
A weak-slope deformation with positive descent rate would extend a
degree-one map from a sphere
across a ball.  The complete endpoint argument is given in
\cref{subsec:antipodal-weak-slope}.

The two-label result shows that \(\pi\) belongs to both spectra for \(N=2\). Repeating entries preserves stationarity by \cref{lem:weak-slope-repetition-nesting,thm:fixed-label-clarke-spectrum}, so \(\pi\) belongs to both spectra for every \(N\geq2\). Constant tuples show that zero also belongs to both spectra.
Since no diameter exceeds \(\pi\),
\eqref{eq:fixed-spherical-spectrum-coincidence} yields
\begin{equation}
\begin{split}
 \Sigma_N^{\specWS}(\Sph^m)
 &=\Sigma_N^{\specC}(\Sph^m),\\
 \Sigma^{\specWS}(\Sph^m)
 &=\Sigma^{\specC}(\Sph^m).
\end{split}
\label{eq:full-fixed-spherical-spectrum-coincidence}
\end{equation}
These are equalities of value sets.  They do not assert pointwise
equivalence of weak-slope stationarity and Clarke criticality for all
tuples of diameter \(\pi\).

\subsection{The sharp positive minimum at \texorpdfstring{\(N\)}{N} labels}

We now compute the least positive stationary diameter for a prescribed
number of labels and identify every configuration attaining it.

\begin{lemma}[Euclidean Jung's theorem in spherical form {\cite{Jung1901}}]
\label{lem:spherical-jung}
Let \(P\subset\Sph^m\), \(m\geq1\), be finite.  If
\(0\in\conv(P)\), then \(\diam(P)\geq\zeta_m\).
Equality holds if and only if \(P\) is exactly the vertex set of a
regular \((m+1)\)-simplex centered at the origin.
\end{lemma}

\begin{proof}
Put \(D:=\diam(P)\) and let
\[
 \Delta_E:=\max_{y,y'\in P}\|y-y'\|=2\sin(D/2)
\]
be its Euclidean chordal diameter.  Choose an inclusion-minimal set
\(\{w_1,\ldots,w_q\}\subseteq P\) containing zero in its convex
hull, and write \(\sum_i a_iw_i=0\), with \(a_i>0\) and
\(\sum_i a_i=1\).  Minimality implies affine independence: an affine
dependence would allow a variation of these coefficients preserving
their sum and barycenter until one coefficient became zero.  Hence
\(q\leq m+2\).

For any Euclidean center \(c\),
\[
 \sum_i a_i\|w_i-c\|^2=1+\|c\|^2.
\]
The Euclidean circumradius of \(P\) is therefore exactly one.  The
Jung estimate for this configuration follows directly from
\[
 1=\frac12\sum_{i,j}a_i a_j\|w_i-w_j\|^2
 \leq\frac{\Delta_E^2}{2}\left(1-\sum_i a_i^2\right)
 \leq\frac{\Delta_E^2}{2}\left(1-\frac1q\right)
 \leq\frac{m+1}{2(m+2)}\Delta_E^2.
\]
It follows that
\(\cos D=1-\Delta_E^2/2\leq-1/(m+1)\), or
\(D\geq\zeta_m\).

If \(D=\zeta_m\), equality holds throughout the displayed chain.
Thus \(q=m+2\), all coefficients are \(1/(m+2)\), and every distinct
pair \(w_i,w_j\) has chordal distance \(\Delta_E\).  In particular,
\[
 \sum_iw_i=0,
 \qquad w_i\cdot w_j=-\frac1{m+1}\quad(i\ne j).
\]
These points form a centered regular \((m+1)\)-simplex.
For any \(y\in P\), the diameter bound gives
\(y\cdot w_i\geq-1/(m+1)\) for every \(i\).  The barycentric
coordinates of \(y\in\R^{m+1}\) with respect to this simplex are
\[
 b_i=\frac{1+(m+1)y\cdot w_i}{m+2}.
\]
They are nonnegative and sum to one, so \(y\in\conv\{w_i\}\).
Strict convexity of the Euclidean unit ball forces a unit vector in
this simplex to be one of its vertices.  Hence \(P=\{w_i\}\).
The displayed inner products also prove the converse.
\end{proof}

A stress can be supported on fewer points than the original tuple.
Applying \cref{lem:spherical-jung} in the linear span of
an inclusion-minimal subset whose convex hull contains
the origin accounts for both restrictions: the ambient
dimension and the available number of labels.

\begin{theorem}[Sharp spherical diameter gap for \(N\)-tuples]
\label{thm:first-spherical-clarke-value}
For every \(m\geq1\) and \(N\geq2\),
\begin{equation}
 \min\bigl(\Sigma_N^{\specC}(\Sph^m)\setminus\{0\}\bigr)
 =\min\bigl(\Sigma_N^{\specWS}(\Sph^m)\setminus\{0\}\bigr)
 =\zeta_{\min\{m,N-2\}}.
 \label{eq:fixed-label-spherical-minimum}
\end{equation}
Consequently,
\begin{equation}
 \min\bigl(\Sigma^{\specC}(\Sph^m)\setminus\{0\}\bigr)
 =\min\bigl(\Sigma^{\specWS}(\Sph^m)\setminus\{0\}\bigr)
 =\zeta_m.
 \label{eq:first-spherical-clarke-value}
\end{equation}
If a Clarke-critical or weak-slope stationary \(N\)-tuple has diameter
\(\zeta_{\min\{m,N-2\}}\), then its distinct entries are precisely the
vertices of a regular \(\min\{m+1,N-1\}\)-simplex centered at the origin.
Repeated entries occur only when \(N>m+2\).
\end{theorem}

\begin{proof}
\noindent\emph{Lower bound.}
The case \(N=2\), including the equality configuration, follows from
\cref{lem:two-label-spherical-spectrum}.  Assume \(N\geq3\), and let
\(x\in(\Sph^m)^N\) be Clarke critical with positive diameter \(D\).
If \(D=\pi\), it is strictly larger than the claimed minimum.  If
\(D<\pi\), \cref{prop:spherical-spectrum-coincidence} gives a
nonzero nonnegative stress on its  underlying set
\(P:=\{x_1,\ldots,x_N\}\).  For this stress, discard the vertices of zero
stressed degree and denote
the remaining set by \(S\).  By
\eqref{eq:spherical-stress-weighted-center}, \(0\in\conv(S)\).

Choose an inclusion-minimal subset
\(Q=\{z_1,\ldots,z_q\}\subseteq S\) with \(0\in\conv(Q)\).
It is affinely independent by the coefficient-variation argument in
\cref{lem:spherical-jung}.  Since its affine span contains zero, its
linear span has dimension \(q-1\).  We have \(q\geq2\), because a
unit vector is nonzero.  The case \(q=2\) would force an antipodal
pair and hence \(D=\pi\).  Therefore
\[
 3\leq q\leq\min\{N,m+2\}.
\]
The unit sphere in \(\spann(Q)\) is an equatorial \(\Sph^{q-2}\).
Applying \cref{lem:spherical-jung} there gives
\[
 D\geq\diam(Q)\geq\zeta_{q-2}
 \geq\zeta_{\min\{m,N-2\}}.
\]
By \eqref{eq:full-fixed-spherical-spectrum-coincidence}, this also
bounds the positive weak-slope spectrum from below.

\par\smallskip\noindent\emph{Attainment.}
To attain the bound, put \(q:=\min\{N,m+2\}\) and take the vertices
\(v_1,\ldots,v_q\) of a centered regular \((q-1)\)-simplex in an
equatorial \(\Sph^{q-2}\subseteq\Sph^m\).  Their diameter is
\(\zeta_{q-2}<\pi\), and
\[
 \sum_i v_i=0,
 \qquad v_i\cdot v_j=-\frac1{q-1}\quad(i\ne j).
\]
Giving every edge the same positive weight satisfies equilibrium:
\[
 \sum_{j\ne i}\left(v_j+\frac1{q-1}v_i\right)=0.
\]
Thus \cref{prop:spherical-spectrum-coincidence} makes the ordered
vertex tuple stationary for both notions.  If \(N>q\), repeat
vertices; the characterization by an equilibrium stress
on the distinct points in that proposition
gives stationarity of the resulting \(N\)-tuple as well.  This proves
\eqref{eq:fixed-label-spherical-minimum}.

\par\smallskip\noindent\emph{Equality configurations.}
At equality, strict decrease of \(j\mapsto\zeta_j\) forces
\(q=\min\{N,m+2\}\) in the lower-bound chain.  The equality case of
\cref{lem:spherical-jung} makes \(Q\) a centered regular
\((q-1)\)-simplex.  If \(N\leq m+2\), its \(q=N\) distinct
vertices already use every label.  If \(N>m+2\), it is a
full-dimensional regular simplex in \(\R^{m+1}\).  Each additional
support point \(y\) satisfies \(y\cdot z_i\geq-1/(m+1)\) for
every vertex \(z_i\).  The same barycentric and strict-convexity
argument as in \cref{lem:spherical-jung} forces \(y\in Q\).
All additional labels are therefore repetitions.  Taking
\(N\geq m+2\) proves \eqref{eq:first-spherical-clarke-value}.
\end{proof}

\subsection{The initial Vietoris--Rips range}

For \(m\geq1\), \cref{thm:first-spherical-clarke-value} gives
\[
 r_{\mathrm C}(\Sph^m)=\zeta_m
 =\arccos\!\left(-\frac1{m+1}\right)
 >\frac\pi2=\convRad(\Sph^m).
\]
Thus \cref{cor:convexity-radius-rips-stability} extends the
convexity-radius range strictly.  The strict Vietoris--Rips convention
includes the regular-simplex scale \(\zeta_m\) as the upper endpoint.

\begin{corollary}[The initial spherical Vietoris--Rips interval]
\label{cor:optimal-spherical-initial-chamber}
\label{cor:stationary-chamber-intro}
Let \(m\geq1\).  For every \(0<r\leq s\leq\zeta_m\), the
canonical inclusion
\[
 \VR{\Sph^m}{r}\longrightarrow\VR{\Sph^m}{s}
\]
is a homotopy equivalence.  Moreover,
\(\VR{\Sph^m}{t}\simeq\Sph^m\) for every \(0<t\leq\zeta_m\).
\end{corollary}

\begin{proof}
Apply \cref{cor:convexity-radius-rips-stability} with
\(r_{\mathrm C}(\Sph^m)=\zeta_m\), as computed in
\cref{thm:first-spherical-clarke-value}.  The case \(r=s\) is the identity.
\end{proof}

The initial spherical homotopy range in
\cref{cor:optimal-spherical-initial-chamber},
including the assertion that the canonical Vietoris--Rips inclusions in this range are homotopy equivalences,
was established by Lim--M\'emoli--Okutan
\cite[Theorem~7.1, p.~1046, and Section~A.4.1]{LimMemoliOkutan2024}.
Their proof constructs homotopy equivalences to the sphere
such that the map on a larger Vietoris--Rips complex restricts to the map
on each smaller complex. Hence the canonical inclusions are homotopy
equivalences.

Our argument gives an alternative proof of this spherical initial-range
theorem: Hausmann's small-scale theorem identifies the homotopy type,
and \cref{cor:convexity-radius-rips-stability}, together with the sharp
identity \(r_{\mathrm C}(\Sph^m)=\zeta_m\), extends that identification
throughout \(0<t\leq\zeta_m\).  It does not use the spherical
initial-range theorem of Lim--M\'emoli--Okutan, but still uses
Hausmann's small-scale theorem \cite[Theorem~3.5]{Hausmann1995}.

The endpoint is sharp for the canonical maps as well, but this requires
an additional topological argument.  Lim--M\'emoli--Okutan show that the
fundamental class of the initial \(\Sph^m\), with coefficients in any
field, dies at scales larger than \(\zeta_m\)
\cite[Remark~9.27, Proposition~9.28, and Remark~9.29]{LimMemoliOkutan2024}.
Consequently, a canonical inclusion from the initial range to a scale
larger than \(\zeta_m\) cannot be a homotopy equivalence.  This conclusion
uses filling radius and persistence, not stationarity alone.

\subsection{The intrinsic circle and its canonical Vietoris--Rips maps}
\label{sec:circle-rips}

We classify the circle's stationary diameter values and compare its
positive nonantipodal scales with the known Vietoris--Rips transitions.

The intrinsic unit circle has circumference \(2\pi\).  Put
\begin{equation}
 \delta_k:=\frac{2\pi k}{2k+1}\quad(k\geq1),
 \qquad \delta_0:=0.
 \label{eq:circle-critical-scales}
\end{equation}
For a circle configuration of diameter \(D\in(0,\pi)\), at each point
\(y\) of positive stressed degree, the equilibrium equation
\eqref{eq:spherical-gordan-equilibrium} forces diameter neighbors in both
circular directions.  These neighbors bound an empty
arc of length \(2(\pi-D)\) centered at \(-y\).  Counting these arcs
shows that all gaps between stressed points have the same length;
their lengths determine the diameter and the number of points.

\Cref{fig:circle-antipodal-gaps} illustrates the gaps for the regular
triangle and pentagon.  Equal weights on their diameter edges give
balance, while their vertex counts give the least label numbers below.

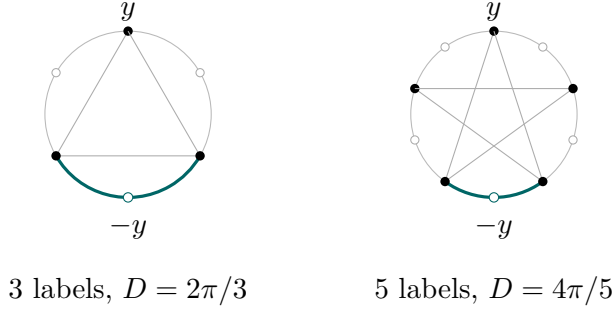
\begin{figure}[htbp]
\centering
\begin{tikzpicture}[scale=1.1]
 \begin{scope}[xshift=-2.2cm]
  \draw[gray!55] (0,0) circle (1);
  \foreach \i in {0,1,2} {
   \coordinate (a\i) at ({90+120*\i}:1);
  }
  \draw[teal!80!black,very thick] (210:1)
   arc[start angle=210,end angle=330,radius=1];
  \foreach \i in {0,1,2} {
   \pgfmathtruncatemacro{\j}{mod(\i+1,3)}
   \draw[gray!65] (a\i) -- (a\j);
   \fill (a\i) circle (1.6pt);
   \draw[gray!65,fill=white] ({270+120*\i}:1) circle (1.4pt);
  }
  \draw[teal!80!black,fill=white] (270:1) circle (1.6pt);
  \node[above] at (a0) {\(y\)};
  \node[below=3pt] at (270:1) {\(-y\)};
  \node[anchor=north] at (0,-1.82) {\(3\) labels, \(D=2\pi/3\)};
 \end{scope}
 \begin{scope}[xshift=2.2cm]
  \draw[gray!55] (0,0) circle (1);
  \foreach \i in {0,1,2,3,4} {
   \coordinate (b\i) at ({90+72*\i}:1);
  }
  \draw[teal!80!black,very thick] (234:1)
   arc[start angle=234,end angle=306,radius=1];
  \foreach \i in {0,1,2,3,4} {
   \pgfmathtruncatemacro{\j}{mod(\i+2,5)}
   \draw[gray!65] (b\i) -- (b\j);
   \fill (b\i) circle (1.6pt);
   \draw[gray!65,fill=white] ({270+72*\i}:1) circle (1.4pt);
  }
  \draw[teal!80!black,fill=white] (270:1) circle (1.6pt);
  \node[above] at (b0) {\(y\)};
  \node[below=3pt] at (270:1) {\(-y\)};
  \node[anchor=north] at (0,-1.82) {\(5\) labels, \(D=4\pi/5\)};
 \end{scope}
\end{tikzpicture}
\caption{The regular triangle and pentagon on the intrinsic unit circle.
Filled points are vertices; hollow points are their antipodes.  The
highlighted gap is centered at \(-y\) and bounded by the two diameter
neighbors of \(y\).  Its length is \(2(\pi-D)\): \(2\pi/3\) for
the triangle and \(2\pi/5\) for the pentagon.  Thin chords represent
diameter pairs; equal weights on these edges balance the opposite
initial tangent directions at each vertex.}
\label{fig:circle-antipodal-gaps}
\end{figure}

\begin{theorem}[Circle spectra and least label numbers]
\label{cor:fixed-label-circle-spectra}
For every \(N\geq2\),
\[
\begin{aligned}
 \Sigma_N^{\specC}(\Sph^1)
 &=\Sigma_N^{\specWS}(\Sph^1)\\
 &=\{0,\pi\}\cup\{\delta_k:k\geq1,\ 2k+1\leq N\}.
\end{aligned}
\]
The set indexed by \(k\) is empty when \(N=2\).
If a finite \(P\subset\Sph^1\) is first-order stationary with
diameter \(D\in(0,\pi)\), then \(D=\delta_k\) for some \(k\geq1\),
and \(P\) is the vertex set of a regular \((2k+1)\)-gon.  Equal
positive weights on its diameter edges give an equilibrium stress.
Conversely, every such polygon is first-order stationary, and its
ordered vertex tuple is both weak-slope stationary
and Clarke critical.  In particular, the smallest number of labels needed to realize
\(\delta_k\) in either spectrum is \(2k+1\).
\end{theorem}

Taking the union of $\Sigma_N^{\specC}(\Sph^1)
=\Sigma_N^{\specWS}(\Sph^1)$ over \(N\) gives
 $\Sigma^{\specC}(\Sph^1)=\Sigma^{\specWS}(\Sph^1)
 =\{0,\pi\}\cup\{\delta_k:k\geq1\}.$

 See \cite[Lemma 4.3]{Katz1991} and \cite[Proposition~2.5]{KatzMemoliWang2023} for a related
classification, and Moy \cite[Section~3]{Moy2022}
for a related argument using excluded arcs.

\begin{proof}[Proof of \Cref{cor:fixed-label-circle-spectra}]
Let \(P=\{y_1,\ldots,y_\ell\}\) be first-order stationary with
diameter \(D\in(0,\pi)\).  By
\cref{cor:spherical-gordan-stress}, choose a nonzero nonnegative
equilibrium stress \(\bigl(w_{ij}\bigr)_{1\leq i,j\leq\ell}\), and put
\[
 h_i:=\sum_jw_{ij},
 \qquad S:=\{y_i:h_i>0\},
 \qquad q:=|S|.
\]
Every positive-weight edge has both endpoints in \(S\), so
\(\diam(S)=D\).  At each \(y_i\in S\), division of
\eqref{eq:spherical-gordan-equilibrium} by \(h_i\sin D\) gives
\[
 \sum_{j:w_{ij}>0}\frac{w_{ij}}{h_i}
       \frac{y_j-\cos(D)y_i}{\sin D}=0,
 \qquad
 \sum_{j:w_{ij}>0}\frac{w_{ij}}{h_i}=1.
\]
The normalized vectors are the initial unit tangents toward the
diameter neighbors.  In the one-dimensional tangent space, convex
balance requires both directions.  Thus both points at distance
\(D\) from \(y_i\) belong to \(S\).

For each \(y\in S\), consider the excluded open arc
\[
 I_y:=\{x\in\Sph^1:d_1(x,-y)<\pi-D\}.
\]
The identity \(d_1(x,y)+d_1(x,-y)=\pi\) shows that every
\(x\in I_y\) satisfies \(d_1(x,y)>D\); hence \(I_y\cap P=\varnothing\).
The endpoints of \(I_y\) are the two diameter neighbors of \(y\),
which belong to \(S\).  Therefore \(I_y\) is a whole gap between
consecutive points of \(S\).  It has length \(2(\pi-D)\) and
midpoint \(-y\), so distinct points of \(S\) give distinct gaps.

There are \(q\) arcs \(I_y\) and exactly \(q\) gaps in
\(\Sph^1\setminus S\); the arcs therefore exhaust the gaps.
Summing their lengths gives
\begin{equation}
 2q(\pi-D)=2\pi,
 \qquad D=\pi-\frac\pi q.
 \label{eq:circle-gap-sum}
\end{equation}
All gaps have the same length, so \(S\) is the vertex set of a
regular \(q\)-gon.  Since every gap avoids \(P\), we also have
\(P=S\).  An even regular polygon contains antipodes, contradicting
\(D<\pi\).  Thus \(q=2k+1\) for some \(k\geq1\), and
\eqref{eq:circle-gap-sum} gives \(D=\delta_k\).

Conversely, in a regular \((2k+1)\)-gon the longest pairs join
cyclic indices differing by \(k\).  They have distance \(\delta_k\).
The two diameter neighbors at each vertex have opposite initial unit
tangents, so assigning the same positive weight to every diameter
edge gives equilibrium.  This proves first-order stationarity by
\cref{cor:spherical-gordan-stress} and
the two other stationarity assertions by
\cref{prop:spherical-spectrum-coincidence}.

If a value \(D\in(0,\pi)\) occurs in either spectrum with \(N\)
labels, \cref{prop:spherical-spectrum-coincidence} makes the
 underlying set \(P\) a first-order stationary set.  The polygon
classification gives \(D=\delta_k\) and \(|P|=2k+1\leq N\).
Conversely, the regular polygon realizes \(\delta_k\) with
\(2k+1\) labels; repeating its
vertices preserves stationarity by the characterization
by an equilibrium stress on the distinct points in
that proposition.  The antipodal two-label example and repetition
give \(\pi\) for every \(N\geq2\), while constant tuples give
zero in both spectra.  These possibilities exhaust all stationary diameter values in \([0,\pi]\), proving the formulas and the assertion about the minimum number of labels.
\end{proof}

Below \(\pi\), the circle's positive stationary values also describe
its homotopy transitions.  This additional conclusion uses the
topological classification of
Adamaszek and Adams, including their assertion about canonical maps.

\begin{remark}[Stationary diameters and circle transitions]
\label{cor:circle-canonical-rips-maps}
By Adamaszek--Adams \cite[Theorem~7.4]{AdamaszekAdams2017}, for
\(0<r\leq s<\pi\), the canonical inclusion
\(\VR{\Sph^1}{r}\to\VR{\Sph^1}{s}\) is a homotopy equivalence
exactly when \([r,s)\) contains none of the stationary diameters
\(\delta_j\) classified in \cref{cor:fixed-label-circle-spectra}.
Thus, for \(\Sph^1\), the positive weak-slope and Clarke critical
diameter values below \(\pi\) are exactly the Vietoris--Rips homotopy
transition scales.
\end{remark}

This proves \cref{ex:circle-stationary-rips-intro}.
Among intervals contained in \((0,\pi)\), the maximal intervals on which every canonical inclusion is a homotopy equivalence are \((\delta_k,\delta_{k+1}]\), \(k\geq0\).

\begin{remark}[The strict antipodal scale]
\label{rem:circle-antipodal-scale}
At scale \(\pi\), the strict circle complex consists of finite sets
without an antipodal pair.  Every finite subcomplex lies in a cone inside
this complex: choose a vertex nonantipodal to all its finitely many
vertices.  Every continuous map from a sphere into the
geometric realization of this complex has image contained in a finite subcomplex
\cite[Proposition~A.1]{Hatcher2002}, so it is nullhomotopic.
The connected CW realization is therefore contractible by Whitehead's
theorem \cite[Theorem~4.5]{Hatcher2002}.  At larger scales the complex is a full simplex.
The same statements hold at half the circumference after rescaling.
\end{remark}

\begin{remark}[Minimum label number and connectivity at a critical scale]
\label{rem:critical-cardinality-transition-degree}
Let \(M\) be a closed connected Riemannian manifold, let \(c\in\Sigma^{\specC}(M)\cap(0,\diam(M))\), and suppose that \(q\geq3\) is
the least label number for which \(c\in\Sigma_q^{\specC}(M)\).  Minimality of \(q\) gives
\(c\notin\Sigma_{q-1}^{\specC}(M)\).  This fixed-label spectrum is
compact by \cref{thm:fixed-label-clarke-spectrum}, so there is
\(\varepsilon>0\) such that
\[
 [c-\varepsilon,c+\varepsilon]
 \cap\Sigma_{q-1}^{\specC}(M)=\varnothing.
\]
After decreasing \(\varepsilon\) so that the displayed endpoints lie in
\((0,\diam(M))\), \cref{cor:all-label-clarke-chambers}\textup{(ii)} gives a
\((q-2)\)-connected map
\[
 \VR{M}{c-\varepsilon}
 \longrightarrow
 \VR{M}{c+\varepsilon}.
\]
Thus a Clarke value requiring \(q\) labels cannot change homotopy groups
in degrees below \(q-2\).  It may kill classes in degree \(q-2\), and it
need not be a Vietoris--Rips transition at all.  The following corollary shows that this bound is sharp for the circle.
\end{remark}

\begin{corollary}[Sharpness of the connectivity estimate]
\label{cor:circle-connectivity-sharpness}
Let \(k\geq1\) and
\(\delta_{k-1}<r<\delta_k<s<\delta_{k+1}\).  The canonical inclusion
\(\VR{\Sph^1}{r}\to\VR{\Sph^1}{s}\) is \((2k-1)\)-connected and is
not \(2k\)-connected.
\end{corollary}

\begin{proof}
By \cref{cor:fixed-label-circle-spectra}, the value \(\delta_k\)
first occurs with \(2k+1\) labels, and
\([r,s)\cap\Sigma_{2k}^{\specC}(\Sph^1)=\varnothing\).
Hence \cref{cor:all-label-clarke-chambers}\textup{(ii)} makes the
canonical inclusion \((2k-1)\)-connected.  Its source and target
have homotopy types \(\Sph^{2k-1}\) and \(\Sph^{2k+1}\),
respectively, by \cite[Theorem~7.4]{AdamaszekAdams2017}.
The induced map on \(\pi_{2k-1}\) is therefore a map from
\(\mathbb Z\) to zero, and is not an isomorphism.  The inclusion
cannot be \(2k\)-connected.  Thus the bound in
\cref{rem:critical-cardinality-transition-degree} is sharp.
\end{proof}

The odd polygons will also serve as bases for the higher-dimensional
Lov\'asz stacks in \cref{sec:spherical-stacks}.  Those constructions
produce accumulating stationary values.  Unlike the circle calculation,
they do not by themselves identify Vietoris--Rips transitions.

\section{Accumulation of spherical stationary diameters}
\label{sec:spherical-stacks}

We now determine where stationary diameters on \(\Sph^m\) first
accumulate as the number of labels grows.  For every fixed \(N\), the
spectrum \(\Sigma_N^{\specC}(\Sph^m)\) is finite by
\cref{cor:analytic-definable-fixed-label-finiteness}, so a sequence of
distinct stationary diameters requires unboundedly many labels.
Recall that \(\zeta_j=\arccos(-1/(j+1))\) and that \(\zeta_m\) is the
least positive stationary diameter on \(\Sph^m\).  For \(m\geq1\), the
next theorem identifies \(\zeta_{m-1}\) as the first accumulation point
of the full spectrum \(\Sigma^{\specC}(\Sph^m)\).

\begin{theorem}[First accumulation of spherical Clarke diameter values]
\label{thm:first-accumulation-spheres}
For \(m\geq1\) and every \(b<\zeta_{m-1}\), the set
\[
 \Sigma^{\specC}(\Sph^m)\cap(0,b]
\]
is finite, and
\[
 \min\operatorname{Acc}\bigl(\Sigma^{\specC}(\Sph^m)\bigr)=\zeta_{m-1}.
\]
For \(m=1\), derived sets are taken in \([0,\pi]\), so the right side is
the endpoint \(\zeta_0=\pi\).
\end{theorem}

There are two parts to the proof.
For each fixed \(b<\zeta_{m-1}\), every stationary value in \((0,b]\) can be realized by a configuration whose number of points is bounded in terms of \(m\) and \(b\). Since the spectrum for a fixed number of labels is finite, only finitely many such values occur.  For sharpness, we lift
regular odd polygons one layer at a time to obtain values approaching
the threshold.  The general construction with an arbitrary number of
layers will then give accumulation at every positive nonantipodal
stationary value from the preceding dimension.

\subsection{Realizing critical values with a bounded number of points}
\label{subsec:bounded-realizing-supports}

Removing vertices of zero stressed degree preserves the diameter of a
configuration carrying a nonzero nonnegative equilibrium stress.
We show that the remaining
vertices are uniformly separated when the diameter is bounded by
\(b<\zeta_{m-1}\).  At each vertex, the equilibrium equation places zero
in the convex hull of the unit tangent directions toward its positive-weight
neighbors.  The diameter bound forces this convex hull to contain a ball
of uniformly positive radius, which prevents another vertex from
approaching arbitrarily closely.  Spherical packing then bounds the
number of retained vertices.

\begin{lemma}[Separation of vertices supporting an equilibrium stress]
\label{lem:subcritical-held-separation}
Let \(m\geq1\) and \(b<\zeta_{m-1}\).  Suppose that a finite set
\(P\subset\Sph^m\) has diameter \(D\in(0,b]\) and carries a nonzero
nonnegative equilibrium stress {\(\bigl(w_{yz}\bigr)_{y,z\in P}\)}.  Put
\[
 h_y:=\sum_z w_{yz},\qquad S:=\{y\in P:h_y>0\}.
\]
Then \(D>\pi/2\), \(\diam(S)=D\), and the restricted weights give an
equilibrium stress on \(S\).  For distinct \(y,y'\in S\),
\begin{equation}
 d_m(y,y')\geq\delta(m,b)
 :=\arccos\!\left(
 \frac{2m\cos^2 b}{1+(m-1)\cos b}-1
 \right)>0.
 \label{eq:stress-support-separation}
\end{equation}
Consequently, \(|S|\) has an upper bound depending only on \(m\) and \(b\).
\end{lemma}

Estimate \eqref{eq:stress-support-separation} provides an explicit
separation bound in every dimension.  Its two-dimensional specialization
coincides with the bound obtained by Katz
\cite[Lemmas~4.1 and~4.3, pp.~126--127]{Katz1989}.

\begin{proof}[Proof of \Cref{lem:subcritical-held-separation}]
Every positive-weight edge has both endpoints in \(S\).  Removing the
vertices of stressed degree zero therefore preserves all remaining equilibrium
equations and at least one edge of length \(D\).  Thus the restricted
stress is nonzero and \(\diam(S)=D\).

Fix \(y\in S\), and write \(c:=\cos D\), \(s:=\sin D\).  For each
neighbor \(z\) with \(w_{yz}>0\), write
\[
 z=cy+su_z,\qquad u_z\in\Sph(T_y\Sph^m).
\]
Dividing \eqref{eq:spherical-gordan-equilibrium} at \(y\) by \(h_y s\)
gives
\begin{equation}
 \sum_{z:w_{yz}>0}\frac{w_{yz}}{h_y}u_z=0,
 \qquad
 \sum_{z:w_{yz}>0}\frac{w_{yz}}{h_y}=1.
 \label{eq:held-core}
\end{equation}
The diameter bound also gives, for distinct such neighbors,
\begin{equation}
 u_z\cdot u_{z'}
 \geq\frac{c-c^2}{1-c^2}
 =\frac{c}{1+c}.
 \label{eq:held-direction-gram}
\end{equation}
If \(c\geq0\), all these pairwise products are nonnegative, so the
squared norm of the convex combination in \eqref{eq:held-core} is
positive.  Hence \(D>\pi/2\), and therefore \(b>\pi/2\).

Set
\[
 \gamma_b:=\frac{-\cos b}{1+\cos b},
 \qquad
 \rho:=\sqrt{\frac{1-(m-1)\gamma_b}{m}}
       =\sqrt{\frac{1+m\cos b}{m(1+\cos b)}}>0.
\]
Here positivity follows from \(b<\zeta_{m-1}=\arccos(-1/m)\).
Since \(D\leq b\), \eqref{eq:held-direction-gram} gives
\(u_z\cdot u_{z'}\geq-\gamma_b\).  For any convex combination of
\(q\leq m\) of these directions, with \(a_i\geq0\) and \(\sum_i a_i=1\),
\begin{equation}
 \left\lVert\sum_{i=1}^q a_i u_i\right\rVert^2
 \geq(1+\gamma_b)\sum_{i=1}^q a_i^2-\gamma_b
 \geq\frac{1+\gamma_b}{m}-\gamma_b
 =\rho^2.
 \label{eq:stress-small-face-bound}
\end{equation}

Let
\[
 \mathcal K_y:=\conv\{u_z:w_{yz}>0\}\subset T_y\Sph^m.
\]
Zero is an interior point of \(\mathcal K_y\) in the \(m\)-dimensional tangent
space.
Otherwise, a supporting hyperplane through zero would contain every \(u_z\) with \(w_{yz}>0\), by \eqref{eq:held-core}.  Carath\'eodory's theorem
in that hyperplane \cite[Theorem~17.1]{Rockafellar1970} would express zero
using at most \(m\) directions, contrary to
\eqref{eq:stress-small-face-bound}.

Every boundary point of \(\mathcal K_y\) lies in a supporting face of affine
dimension at most \(m-1\).  Applying Carath\'eodory in that face and
\eqref{eq:stress-small-face-bound} shows that its norm is at least
\(\rho\).  Thus \(\mathcal K_y\) contains the closed ball of radius \(\rho\)
centered at zero.  In particular, for every unit \(v\in T_y\Sph^m\),
there is a neighbor \(z\) with \(w_{yz}>0\) whose tangent direction satisfies
\begin{equation}
 v\cdot u_z\leq-\rho.
 \label{eq:uniform-tangent-cover}
\end{equation}

For \(y'\in S\setminus\{y\}\), write
\[
 y'=\cos r\,y+\sin r\,v,\qquad
 r:=d_m(y,y')\in(0,D]\subset(0,\pi).
\]
Choose \(u_z\) as in \eqref{eq:uniform-tangent-cover}.  Since
\(d_m(y',z)\leq D\),
\[
 c\leq y'\cdot z
 \leq c\cos r-s\rho\sin r.
\]
Rearranging this inequality gives
\[
 s\rho\sin r\leq(-c)(1-\cos r).
\]
Since \(-c>0\) and \(\sin r>0\), division by
\((-c)\sin r\), together with
\((1-\cos r)/\sin r=\tan(r/2)\), yields the first inequality below.

The second inequality uses \(D\leq b\) and the fact that
\(-\tan u\) is decreasing on \((\pi/2,\pi)\):
\[
 \tan\frac{r}{2}\geq-\rho\tan D\geq-\rho\tan b>0.
\]
Applying the increasing function \(2\arctan\) gives the desired separation bound:
\[
 r\geq2\arctan\!\bigl(-\rho\tan b\bigr)=\delta(m,b).
\]
To verify the equality, put
\(t:=\rho\sin b/(-\cos b)>0\).  The minus sign makes the denominator
positive, since \(\pi/2<b<\pi\).  Substituting the value of \(\rho^2\)
and cancelling \(1+\cos b\) gives
\[
 t^2
 =\frac{1+m\cos b}{m(1+\cos b)}
   \frac{\sin^2 b}{\cos^2 b}
 =\frac{(1+m\cos b)(1-\cos b)}{m\cos^2 b}.
\]
Consequently,
\begin{align*}
 1+t^2
 =\frac{m\cos^2 b+(1+m\cos b)(1-\cos b)}{m\cos^2 b}
 =\frac{1+(m-1)\cos b}{m\cos^2 b},
\end{align*}
and hence
\[
 \cos(2\arctan t)
 =\frac{2}{1+t^2}-1
 =\frac{2m\cos^2 b}{1+(m-1)\cos b}-1.
\]
The denominator in the last expression is also positive:
\(b<\arccos(-1/m)\) implies
\(1+(m-1)\cos b\geq1/m>0\).
Since \(2\arctan t\in(0,\pi)\), taking arccosines proves the
claimed equality with \(\delta(m,b)\).
For \(m=1\), the same argument gives \(\rho=1\) and
\(\delta(1,b)=2(\pi-b)\).
Finally, disjoint spherical caps of radius \(\delta(m,b)/2\), centered
at the points of \(S\), give the cardinality bound.
\end{proof}

\begin{proposition}[Realizations with a uniformly bounded number of points]
\label{prop:bounded-realizing-supports}
For every \(m\geq1\) and \(b<\zeta_{m-1}\), there is an integer
\(N_0=N_0(m,b)\geq2\) such that each value in
\(\Sigma^{\specC}(\Sph^m)\cap(0,b]\) is realized by a stationary
configuration with at most \(N_0\) distinct points.  Repeating labels as
needed, we obtain
\[
 \Sigma^{\specC}(\Sph^m)\cap(0,b]
 \subseteq\Sigma_{N_0}^{\specC}(\Sph^m).
\]
Consequently, the set on the left is finite.
\end{proposition}

\begin{proof}
Let \(D\in\Sigma^{\specC}(\Sph^m)\cap(0,b]\).  Since \(D<\pi\),
\cref{prop:spherical-spectrum-coincidence} supplies a finite realization
with a nonzero nonnegative equilibrium stress.
\Cref{lem:subcritical-held-separation} gives a subset \(S\) of the same
diameter, carrying the restricted stress, with
\(2\leq|S|\leq N_0(m,b)\), independently of \(D\).
This stress makes its ordered vertex tuple Clarke critical by
\cref{prop:spherical-spectrum-coincidence}.
Repeating vertices up to \(N_0\) labels proves the displayed inclusion
by \cref{thm:fixed-label-clarke-spectrum}.
The containing spectrum is finite by
\cref{cor:analytic-definable-fixed-label-finiteness}.
\end{proof}

\subsection{One-layer lifts and the first accumulation point}
\label{subsec:first-accumulation-point}

To complete the proof of \cref{thm:first-accumulation-spheres}, we construct
stationary configurations in \(\Sph^m\) whose diameters increase to
\(\zeta_{m-1}\).  We obtain them by repeatedly applying the following
one-layer construction  that is a special case of a method due to Lovasz (see \Cref{sec:general-stacks}): place a rescaled copy of a stationary configuration
on a parallel in a sphere of one higher dimension and add the north pole;
see \cref{fig:one-layer-lovasz-stack}.

\begin{proposition}[One-layer stress lifting]
\label{cor:one-layer-lovasz-transform}
Let \(P=\{p_\xi:\xi\in A\}\subset\Sph^m\) be finite with
\(\alpha:=\diam(P)\in(\tfrac{\pi}{2},\pi)\), and put
\(c:=\cos\alpha\), \(a:=c/(1-c)\), and \(s:=\sqrt{1-a^2}\).
Identify \(\R^{m+2}=\R^{m+1}\times\R\), and define
\[
\begin{gathered}
 \Sph^{m+1}\supset\Stack_1(P):=\{Z\}\cup\{Y_\xi:\xi\in A\},\\
 Z:=(0,1),\qquad Y_\xi:=(s p_\xi,a).
\end{gathered}
\]
Here \(0\) denotes the zero vector in \(\R^{m+1}\).
Each \(Y_\xi\) is obtained by multiplying the coordinates of \(p_\xi\)
by \(s\) and appending \(a\) as the last coordinate.  Thus the lifted
points lie on the parallel of height \(a\), while \(Z\) is the north pole.
Then \(\diam(\Stack_1(P))=\arccos a\).
Every nonzero nonnegative equilibrium stress
\(\bigl(w_{\xi\eta}\bigr)_{\xi,\eta\in A}\) on \(P\) induces a nonzero
nonnegative equilibrium stress on \(\Stack_1(P)\): assign \(Y_\xi Y_\eta\)
weight \(w_{\xi\eta}\) and \(Y_\xi Z\) weight \((-c)h_\xi\), where
\(h_\xi:=\sum_\eta w_{\xi\eta}\).
\end{proposition}

\begin{proof}
Every pair of lifted base points has inner product
\[
 Y_\xi\cdot Y_\eta
 =(1-a^2)(p_\xi\cdot p_\eta)+a^2
 \geq(1-a^2)c+a^2=a,
\]
with equality on the lifted base-diameter pairs.  Every spoke also has
inner product  \(a\), so the diameter is
\(\arccos a\).

For the base stress, the equilibrium equation
\eqref{eq:spherical-gordan-equilibrium} and its summed form
\eqref{eq:spherical-stress-weighted-center} give, respectively,
\[
 \sum_\eta w_{\xi\eta}p_\eta=c h_\xi p_\xi,
 \qquad \sum_\xi h_\xi p_\xi=0.
\]
At a lifted vertex, the required balance is
{
\begin{align*}
 \sum_\eta w_{\xi\eta}(Y_\eta-aY_\xi)
   +(-c)h_\xi(Z-aY_\xi)
 &=\bigl(s h_\xi(c-a+ca)p_\xi,
          h_\xi[a(1-a)-c(1-a^2)]\bigr)\\
 &=0,
\end{align*}
}
because\(a=c/(1-c)\).
At the apex it is
{
\[
 \sum_\xi(-c)h_\xi(Y_\xi-aZ)
 =\bigl((-c)s\sum_\xi h_\xi p_\xi,0\bigr)=0.
\]}
The lifted weights are nonnegative and not all zero, since the base weights are retained. The construction also allows zero weights and vertices of zero stressed degree.
\end{proof}

\begin{figure}[htbp]
\centering
\begingroup
\contourlength{.65pt}
\newcommand{\stackfigmath}[1]{\contour{white}{\ensuremath{#1}}}
\begin{tikzpicture}[
  >={Latex[length=1.6mm,width=1.2mm]},
  every node/.style={font=\small,text=black,inner sep=2pt},
  outline/.style={draw=black!55,line width=.55pt},
  edge/.style={draw=black!75,line width=.8pt},
  guide/.style={draw=black!35,line width=.4pt,densely dashed},
  back/.style={draw=black!55,line width=.55pt,densely dashed},
  dot/.style={circle,fill=black!75,inner sep=1.4pt},
  panel/.style={font=\small\bfseries,anchor=base},
  formula/.style={font=\small,align=center},
  line cap=round,line join=round
]
\def\stackfigR{1.28}
\pgfmathsetmacro{\stackfiga}{-1/3}
\pgfmathsetmacro{\stackfigs}{sqrt(1-(\stackfiga)^2)}
\pgfmathsetmacro{\stackfiglat}{asin(\stackfiga)}

\begin{scope}[shift={(1.65,0)}]
  \node[panel] at (0,1.95) {(a) Start with the base};
  \draw[outline] (0,0) circle[radius=\stackfigR];
  \foreach \i/\a in {1/90,2/210,3/330}{
    \coordinate (base\i) at
      ({\stackfigR*cos(\a)},{\stackfigR*sin(\a)});
  }
  \draw[edge] (base1)--(base2)--(base3)--cycle;
  \foreach \i in {1,2,3}{\node[dot] at (base\i) {};}
  \node[above] at (base1) {\stackfigmath{p_1}};
  \node[below left] at (base2) {\stackfigmath{p_2}};
  \node[below right] at (base3) {\stackfigmath{p_3}};
  \node[formula] at (0,-1.80) {\(P\subset\Sph^1\)};
  \node[formula] at (0,-2.32)
    {\(\alpha=\diam(P),\quad c=\cos\alpha\)};
\end{scope}

\draw[edge,->] (3.28,0.2)--(4.18,0.2);

\begin{scope}[shift={(6.20,0)}]
  \node[panel] at (0,1.95) {(b) Choose the parallel};
  \draw[outline] (0,0) circle[radius=\stackfigR];
  \draw[guide] (-\stackfigR,0)--(\stackfigR,0);
  \draw[guide] (0,-\stackfigR)--(0,\stackfigR);
  \coordinate (meridianY) at
    ({\stackfigR*\stackfigs},{\stackfigR*\stackfiga});
  \draw[outline]
    ({-\stackfigR*\stackfigs},{\stackfigR*\stackfiga})--(meridianY);
  \draw[outline] (0,\stackfigR)--(0,0)--(meridianY);
  \draw[edge,->] (0,{\stackfigR*\stackfiga})--(meridianY)
    node[midway,above] {\stackfigmath{s}};
  \draw[edge] (0,.43)
    arc[start angle=90,end angle=\stackfiglat,radius=.43];
  \node at (.59,.45) {\stackfigmath{D}};
  \node[left] at (0,.05) {\stackfigmath{0}};
  \node[below left] at (0,{\stackfigR*\stackfiga})
    {\stackfigmath{a}};
  \node[dot] at (0,\stackfigR) {};
  \node[above] at (0,\stackfigR) {\stackfigmath{(0,1)}};
  \node[dot] at (meridianY) {};
  \node[below right] at (meridianY) {\stackfigmath{(s,a)}};
  \node[formula] at (0,-1.80)
    {\(a=\dfrac{c}{1-c},\quad s=\sqrt{1-a^2}\)};
  \node[formula] at (0,-2.32) {\(D=\arccos a\)};
\end{scope}

\draw[edge,->] (8.28,0.2)--(9.18,0.2);

\begin{scope}[shift={(11.50,0)}]
  \node[panel] at (0,1.95) {(c) Adjoin the pole};
  \pgfmathsetmacro{\stackfigcy}{\stackfigR*\stackfiga*cos(18)}
  \pgfmathsetmacro{\stackfigrx}{\stackfigR*\stackfigs}
  \pgfmathsetmacro{\stackfigry}{\stackfigrx*sin(18)}
  \coordinate (liftZ) at (0,{\stackfigR*cos(18)});
  \foreach \i/\a in {1/90,2/210,3/330}{
    \coordinate (lift\i) at
      ({\stackfigrx*cos(\a)},
       {\stackfigcy-\stackfigry*sin(\a)});
  }
  \draw[outline] (0,0) circle[radius=\stackfigR];
  \draw[guide] (0,-\stackfigR)--(liftZ);
  \draw[back] (-\stackfigrx,\stackfigcy)
    arc[start angle=180,end angle=0,
        x radius=\stackfigrx,y radius=\stackfigry];
  \draw[outline] (-\stackfigrx,\stackfigcy)
    arc[start angle=180,end angle=360,
        x radius=\stackfigrx,y radius=\stackfigry];
  \draw[edge] (lift1)--(lift2)--(lift3)--cycle;
  \foreach \i in {1,2,3}{\draw[edge] (liftZ)--(lift\i);}
  \foreach \i in {1,2,3}{\node[dot] at (lift\i) {};}
  \node[dot] at (liftZ) {};
  \node[above] at (liftZ) {\stackfigmath{Z=(0,1)}};
  \node[below] at (lift1) {\stackfigmath{Y_1}};
  \node[left] at (lift2) {\stackfigmath{Y_2}};
  \node[right] at (lift3) {\stackfigmath{Y_3}};
  \node[formula] at (0,-1.80)
    {\(Y_\xi=(s p_\xi,a)\in\Sph^2\)};
  \node[formula] at (0,-2.32)
    {\(\Stack_1(P)=\{Z\}\cup\{Y_\xi:\xi\in A\}\)};
\end{scope}
\end{tikzpicture}
\endgroup
\caption{The one-layer Lov\'asz construction, illustrated by
a regular triangle in \(\Sph^1\), which lifts to a regular tetrahedron
in \(\Sph^2\).  The parallel has height \(a=c/(1-c)\) and Euclidean
radius \(s\).  Every pole--vertex pair and every lifted base-diameter
pair has spherical distance \(D=\arccos a\).  Straight segments in
panels~(a) and~(c) indicate diameter pairs, not spherical geodesics.}
\label{fig:one-layer-lovasz-stack}
\end{figure}
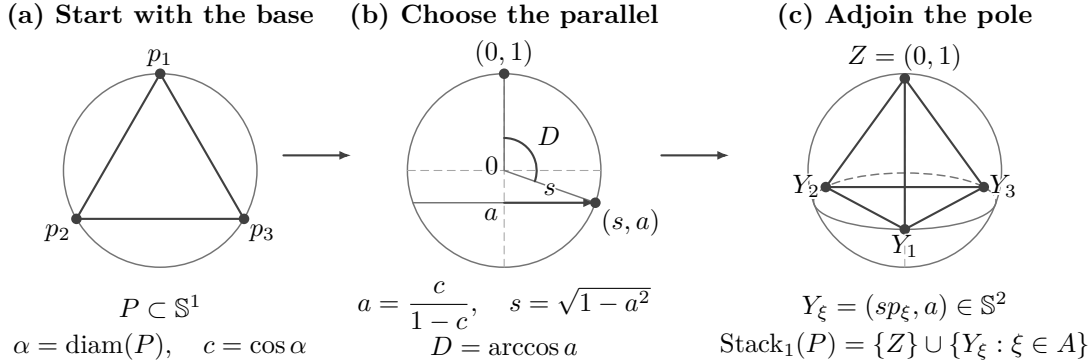

For \(k\geq1\), let \(R_{2k+1}\subset\Sph^1\) be the vertex set
of a regular \((2k+1)\)-gon.

Iterate the one-layer construction from the regular odd polygons by setting
\[
 C_{1,k}:=R_{2k+1},
 \qquad
 C_{m+1,k}:=\Stack_1(C_{m,k})\subset\Sph^{m+1}
 \qquad(m,k\geq1),
\]
and write \(\theta_{m,k}:=\diam(C_{m,k})\).

\begin{corollary}[Iterated lifts of odd polygons and the first accumulation point]
\label{thm:stationary-iterated-one-layer-stacks}
For all \(m,k\geq1\),
\[
 |C_{m,k}|=2k+m,\qquad
 \theta_{m,k}\in\Sigma_{2k+m}^{\specC}(\Sph^m).
\]
\begin{equation}
 \cos\theta_{m,k}
 =-\frac{\cos(\frac{\pi}{2k+1})}
 {1+(m-1)\cos(\frac{\pi}{2k+1})}.
 \label{eq:iterated-one-layer-stack-diameter}
\end{equation}
For each fixed \(m\), the diameters \(\theta_{m,k}\) increase strictly
with \(k\), and
\[
 \theta_{m,k}\nearrow\zeta_{m-1}=\arccos\left(-\frac{1}{m}\right).
\]
\end{corollary}

\begin{proof}
Put \(a_k:=\cos(\tfrac{\pi}{2k+1})\).
The diameter pairs of \(R_{2k+1}\) form one cycle, and equal weights on that
cycle give a nonzero nonnegative equilibrium stress.  Repeated application
of \cref{cor:one-layer-lovasz-transform} preserves such a stress and adds one
point at each step.  Since all displayed diameters lie below \(\pi\),
\cref{prop:clarke-first-order} gives
\(\theta_{m,k}\in\Sigma_{2k+m}^{\specC}(\Sph^m)\).

By \cref{cor:one-layer-lovasz-transform}, if
{\(c=\cos\diam(P)\in(-1,0)\)}, one application of the one-layer
construction replaces {\(c\)} by
{\(c/(1-c)\)}.  After \(r\) successive applications of the
one-layer construction, its value is
{\(c/(1-rc)\)}, by induction.
For the regular polygon,
\(\cos\diam(R_{2k+1})=-a_k\); taking \(r=m-1\) gives
\eqref{eq:iterated-one-layer-stack-diameter}.  Since \(a_k\nearrow1\),
the diameters increase strictly to \(\arccos(-\tfrac{1}{m})\).
\end{proof}

\phantomsection
\label{proof:first-accumulation-spheres}
\begin{proof}[\normalfont\bfseries Proof of \cref{thm:first-accumulation-spheres-intro} and \cref{thm:first-accumulation-spheres}]
For every \(b<\zeta_{m-1}\), \cref{prop:bounded-realizing-supports}
gives finiteness of \(\Sigma^{\specC}(\Sph^m)\cap(0,b]\).
The distinct values in \cref{thm:stationary-iterated-one-layer-stacks}
increase to \(\zeta_{m-1}\), so this is an accumulation point and there
is no smaller one.  Zero cannot be an accumulation point, by the positive
gap in \cref{thm:first-spherical-clarke-value}.
\end{proof}

\smallskip\noindent\emph{Topological consequence.}
Combined with \cref{cor:all-label-clarke-chambers}, this local finiteness
divides \((0,b]\), for \(0<b<\zeta_{m-1}\), into finitely many
intervals on which the canonical Vietoris--Rips inclusions are homotopy
equivalences.  Indeed, write the possibly empty set
\(\Sigma^{\specC}(\Sph^m)\cap(0,b)\) as
\[
 c_1<\cdots<c_J,
\]
and put \(c_0:=0\) and \(c_{J+1}:=b\).  Whenever
\(c_i<r<s\leq c_{i+1}\) for \(0\leq i\leq J\), the half-open interval
\([r,s)\) avoids the Clarke spectrum, so the canonical inclusion
\[
\VR{\Sph^m}{r}\longrightarrow\VR{\Sph^m}{s}
\]
is a homotopy equivalence.  Possible changes in homotopy type are thus
confined to the listed values \(c_1,\ldots,c_J\).

\subsection{General Lov\'asz stacks}\label{sec:general-stacks}

The one-layer construction identifies the first accumulation point.
We next ask whether a prescribed positive nonantipodal stationary
diameter can itself be approached by distinct stationary diameters in
the next dimension. For this we use Lov\'asz's layer construction
\cite[proof of Theorem~1]{Lovasz1983} in a spherical-coordinate
formulation that extends Katz's description; see comments on page~\pageref{sec:related-walk}.
For a base \(P\) of diameter \(\alpha\), the stack has \(k|P|+1\)
points and diameter \(\Lambda_k(\alpha)<\alpha\), tending to \(\alpha\)
as the layer number grows.  We prove that every nonzero nonnegative
equilibrium stress on the base lifts to this stack.
\par

Geometrically, we place rescaled copies of the base on \(k\) parallels
and add the north pole.  An equal-step walk between meridians separated
by \(\pi-\alpha\) determines the parallels: every second colatitude
carries a copy of the base.  Reflection symmetry ensures that the pairs listed in \eqref{eq:general-stack-active-pairs} all have the same distance.
\Cref{fig:stack-walk-to-layers} illustrates how the walk determines the
layers in the case \(k=2\).

Let \(P\) be indexed without repetition by a finite set \(A\), and put
\[
 P:=\{p_\xi:\xi\in A\}\subset\Sph^m,
 \qquad
 \alpha:=\diam(P)\in\left(\frac{\pi}{2},\pi\right),
 \qquad c:=\cos\alpha,
 \qquad \beta:=\pi-\alpha.
\]
For an integer \(k\geq1\), a \emph{symmetric \(k\)-layer stack walk} is a
strictly increasing symmetric colatitude sequence
\[
 \varphi_0:=0,\qquad \varphi_{2k+1}:=\pi,
 \qquad \varphi_0<\varphi_1<\cdots<\varphi_{2k+1}
\]
satisfying
\begin{align}
 \cos\varphi_r\cos\varphi_{r+1}
 -c\sin\varphi_r\sin\varphi_{r+1}
 &=\cos\varphi_1
 \qquad(0\leq r\leq2k),
 \label{eq:general-stack-walk}\\
 \varphi_r+\varphi_{2k+1-r}&=\pi
 \qquad(0\leq r\leq2k+1).
 \label{eq:general-stack-symmetry}
\end{align}

\begin{proposition}[Canonical symmetric stack walk]
\label{prop:canonical-symmetric-stack-walk}
For every
\[
 \alpha\in\left(\frac{\pi}{2},\pi\right)
 \qquad\text{and}\qquad
 k\geq1,
\]
equations \eqref{eq:general-stack-walk}--
\eqref{eq:general-stack-symmetry}, with \(c=\cos\alpha\), admit a unique
symmetric \(k\)-layer stack walk.
\end{proposition}

To construct the walk, we vary the first step and solve \eqref{eq:general-stack-walk} forward with increasing colatitudes. There is a unique choice for which the two middle colatitudes satisfy \(\varphi_k+\varphi_{k+1}=\pi\); reflection across the equator then completes the walk. Appendix~\ref{app:stack-walk-shooting-convergence} gives the proof.

For the canonical walk, define
\[
 \Lambda_k(\alpha):=\pi-\varphi_1=\varphi_{2k}.
\]
We call \(\Lambda_k\) the \emph{diameter transform} of the
\(k\)-layer Lov\'asz construction.
The value \(\Lambda_k(\alpha)\) depends only on \((\alpha,k)\).
\Cref{thm:stationary-stack-lifting} identifies it with the diameter
of the \(k\)-layer stack over every base of diameter \(\alpha\).

For \(1\leq i\leq k\), put
\[
 \ell_i:=\varphi_{2i},\qquad c_i:=\cos\ell_i,\qquad s_i:=\sin\ell_i,
\]
and define the parallel lift \(p_\xi\mapsto v_i^\xi\) to
colatitude \(\ell_i\), and the north-pole apex in
\(\Sph^{m+1}\subset\R^{m+1}\times\R\), by
\[
 Z:=(0,1),\qquad v_i^\xi:=(s_ip_\xi,c_i),
\]
\begin{equation}
 \Stack_k(P):=\{Z\}\cup
 \{v_i^\xi:1\leq i\leq k,\ \xi\in A\}.
 \label{eq:general-spherical-stack}
\end{equation}
Write
\[
 L_0:=\{Z\},
 \qquad
 L_i:=\{v_i^\xi:\xi\in A\}\quad(1\leq i\leq k)
\]
for the apex and the \(k\) lifted layers.
The pole \(L_0=\{Z\}\) is not counted as a layer. Panel~(b) of \cref{fig:stack-walk-to-layers} shows \(L_1,L_2\) and
\(Z\) for a triangular base.
By \cref{prop:canonical-symmetric-stack-walk}, the walk and hence the layer
coordinates are uniquely determined by \((\alpha,k)\), whereas the resulting
configuration is determined by \((P,k)\).  We call \(\Stack_k(P)\) the
\(k\)-layer Lov\'asz stack over \(P\).
For a fixed base, the layer positions are determined separately
for each \(k\), in the same sphere of one higher dimension.  Increasing
\(k\) does not mean appending a layer while retaining the previous layer
positions.
For \(k=1\), writing \(a=\cos\varphi_2\), the recurrence gives
\((a-1)((1-c)a-c)=0\).  The ordered walk selects \(a=c/(1-c)\), so this definition agrees with the direct one-layer
construction in \cref{cor:one-layer-lovasz-transform}.

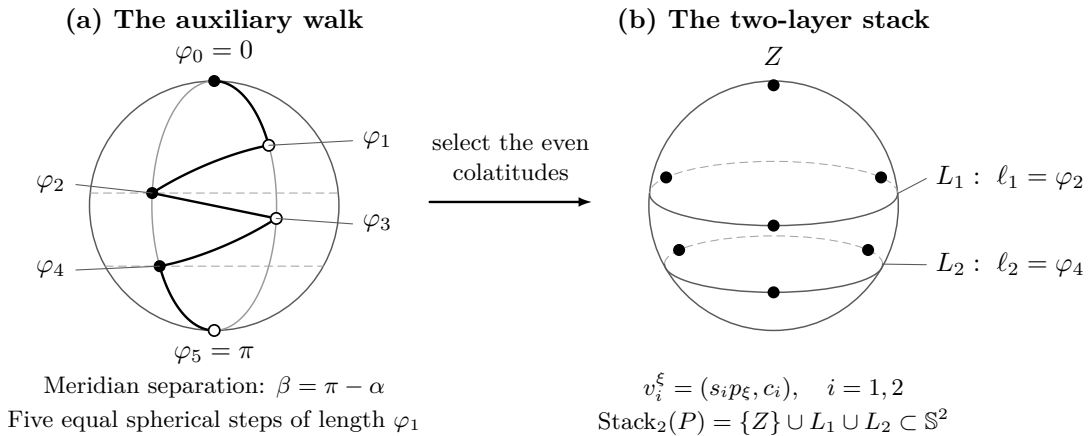
\begin{figure}[H]
\centering
 \begingroup
\color{black}
\contourlength{.55pt}
\newcommand{\stackwalkmath}[1]{\contour{white}{\ensuremath{#1}}}
\begin{tikzpicture}[
  >={Latex[length=1.8mm,width=1.3mm]},
  every node/.style={font=\small,text=black,inner sep=2pt},
  outline/.style={draw=black!65,line width=.55pt},
  walk/.style={draw=black,line width=.9pt},
  guide/.style={draw=black!35,line width=.4pt,densely dashed},
  leader/.style={draw=black!65,line width=.4pt},
  selected/.style={circle,fill=black,inner sep=1.6pt},
  auxiliary/.style={circle,draw=black,fill=white,
                    line width=.65pt,inner sep=1.5pt},
  panel/.style={font=\small\bfseries,anchor=base},
  line cap=round,line join=round
]
\def\walkfigR{1.65}
\def\walkfigstep{61.023267768851}
\def\walkfigsone{.994804636746036}
\def\walkfigcone{.101802429777426}
\def\walkfigstwo{.874816515801676}
\def\walkfigctwo{-.484454397937118}

\begin{scope}[shift={(3.05,0)}]
  \node[panel] at (0,2.35) {(a) The auxiliary walk};
  \draw[outline] (0,0) circle[radius=\walkfigR];
  \foreach \side in {-1,1}{
    \draw[outline,draw=black!40]
      plot[domain=0:180,samples=81,variable=\t]
        ({\side*.5*\walkfigR*sin(\t)},{\walkfigR*cos(\t)});
  }
  \foreach \s/\c in {\walkfigsone/\walkfigcone,\walkfigstwo/\walkfigctwo}{
    \draw[guide] ({-\walkfigR*\s},{\walkfigR*\c})
      --({\walkfigR*\s},{\walkfigR*\c});
  }
  \foreach \xa/\za/\xb/\zb in {
    0/1/.437408257900838/.484454397937118,
    .437408257900838/.484454397937118/-.497402318373018/.101802429777426,
    -.497402318373018/.101802429777426/.497402318373018/-.101802429777426,
    .497402318373018/-.101802429777426/-.437408257900838/-.484454397937118,
    -.437408257900838/-.484454397937118/0/-1}{
    \draw[walk] plot[domain=0:1,samples=35,variable=\t]
      ({\walkfigR*(sin((1-\t)*\walkfigstep)*\xa
          +sin(\t*\walkfigstep)*\xb)/sin(\walkfigstep)},
       {\walkfigR*(sin((1-\t)*\walkfigstep)*\za
          +sin(\t*\walkfigstep)*\zb)/sin(\walkfigstep)});
  }
  \foreach \j/\x/\z in {
    0/0/1,1/.437408257900838/.484454397937118,
    2/-.497402318373018/.101802429777426,
    3/.497402318373018/-.101802429777426,
    4/-.437408257900838/-.484454397937118,5/0/-1}{
    \coordinate (walkpoint\j) at ({\walkfigR*\x},{\walkfigR*\z});
  }
  \foreach \j in {0,2,4}{\node[selected] at (walkpoint\j) {};}
  \foreach \j in {1,3,5}{\node[auxiliary] at (walkpoint\j) {};}
  \node[above=4pt] at (walkpoint0) {\stackwalkmath{\varphi_0=0}};
  \node[below=4pt] at (walkpoint5) {\stackwalkmath{\varphi_5=\pi}};
  \draw[leader] (walkpoint2)--(-1.83,.28);
  \node[anchor=east] at (-1.89,.28) {\stackwalkmath{\varphi_2}};
  \draw[leader] (walkpoint4)--(-1.83,-.85);
  \node[anchor=east] at (-1.89,-.85) {\stackwalkmath{\varphi_4}};
  \draw[leader] (walkpoint1)--(1.83,.86);
  \node[anchor=west] at (1.89,.86) {\stackwalkmath{\varphi_1}};
  \draw[leader] (walkpoint3)--(1.83,-.24);
  \node[anchor=west] at (1.89,-.24) {\stackwalkmath{\varphi_3}};
  \node[font=\footnotesize,align=center] at (0,-2.62)
    {Meridian separation: \(\beta=\pi-\alpha\)\\[2pt]
     Five equal spherical steps of length \(\varphi_1\)};
\end{scope}

\draw[walk,->] (5.90,.05)--(8.05,.05)
  node[midway,above=6pt,font=\footnotesize,align=center]
    {select the even\\colatitudes};

\begin{scope}[shift={(10.45,0)}]
  \node[panel] at (0,2.35) {(b) The two-layer stack};
  \draw[outline] (0,0) circle[radius=\walkfigR];
  \foreach \i/\s/\c in {
    1/\walkfigsone/\walkfigcone,2/\walkfigstwo/\walkfigctwo}{
    \pgfmathsetmacro{\layercy}{\walkfigR*\c*cos(15)}
    \pgfmathsetmacro{\layerrx}{\walkfigR*\s}
    \pgfmathsetmacro{\layerry}{\layerrx*sin(15)}
    \draw[guide] (\layerrx,\layercy)
      arc[start angle=0,end angle=180,x radius=\layerrx,y radius=\layerry];
    \draw[outline] (-\layerrx,\layercy)
      arc[start angle=180,end angle=360,x radius=\layerrx,y radius=\layerry];
    \foreach \a in {90,210,330}{
      \node[selected] at
        ({\layerrx*cos(\a)},{\layercy-\layerry*sin(\a)}) {};
    }
    \coordinate (layeredge\i) at (\layerrx,\layercy);
  }
  \node[selected] (stackpole) at (0,{\walkfigR*cos(15)}) {};
  \node[above=5pt] at (stackpole) {\stackwalkmath{Z}};
  \draw[leader] (layeredge1)--(2.02,.36);
  \node[anchor=west] at (2.07,.36)
    {\stackwalkmath{L_1:\ \ell_1=\varphi_2}};
  \draw[leader] (layeredge2)--(2.02,-.78);
  \node[anchor=west] at (2.07,-.78)
    {\stackwalkmath{L_2:\ \ell_2=\varphi_4}};
  \node[font=\footnotesize,align=center] at (0,-2.62)
    {\(v_i^\xi=(s_ip_\xi,c_i),\quad i=1,2\)\\[2pt]
     \(\Stack_2(P)=\{Z\}\cup L_1\cup L_2\subset\Sph^2\)};
\end{scope}
\end{tikzpicture}
\caption{The two-layer Lov\'asz stack over a regular triangle
\(P\subset\Sph^1\) of diameter \(\alpha=2\pi/3\).
In (a), the auxiliary walk alternates between two meridians, taking five
equal spherical steps of length \(\varphi_1\).  Filled markers indicate
the pole and the selected layer colatitudes; open markers are auxiliary.
In (b), rescaled copies \(L_1,L_2\) of \(P\) lie at
\(\ell_i=\varphi_{2i}\), and the pole \(Z\) is adjoined.
The stack diameter is \(\Lambda_2(\alpha)=\varphi_4=\pi-\varphi_1\).}
\label{fig:stack-walk-to-layers}
\endgroup
\end{figure}

The relation with Lov\'asz's ellipse construction is proved in
\cref{prop:ellipse-stack-walk-equivalence}.
That coordinate equivalence is independent of the stress-lifting
argument in \cref{thm:stationary-stack-lifting}.

\subsection{Lifting equilibrium stresses to Lov\'asz stacks}
\label{subsec:stress-lifting-spectral-transport}

We prove that the stack construction preserves nonzero nonnegative
equilibrium stresses.  It therefore produces stationary configurations
in a sphere of one higher dimension.  We then show that their diameters
approach the base diameter as the number of layers grows.

\begin{theorem}[Lov\'asz-stack geometry and equilibrium stresses]
\label{thm:stationary-stack-lifting}
Let \(P=\{p_\xi:\xi\in A\}\subset\Sph^m\) be a finite set indexed
without repetition, with diameter \(\alpha\in(\pi/2,\pi)\), and put
\(c:=\cos\alpha\).  Let \(k\geq1\) be an integer.
For the canonical walk and stack in
\eqref{eq:general-stack-walk}--\eqref{eq:general-spherical-stack},
\(\Stack_k(P)\) has the following properties.
\begin{enumerate}[label=\textup{(\roman*)}]
\item \emph{Geometry.} The stack \(\Stack_k(P)\) has \(k|P|+1\) points and
\[
 \diam(\Stack_k(P))
 =\Lambda_k(\alpha)
 =\varphi_{2k}
 =\pi-\varphi_1,
 \qquad
 \frac{\pi}{2}<\diam(\Stack_k(P))<\alpha.
\]
Its diameter pairs are exactly
\begin{equation}
 \{Z,v_k^\xi\},
 \qquad
 \{v_i^\xi,v_j^\eta\}
 \quad\text{with}\quad
 p_\xi\cdot p_\eta=c,\quad i+j\in\{k,k+1\}.
 \label{eq:general-stack-active-pairs}
\end{equation}
In the second family, \(\xi,\eta\in A\) and \(1\leq i,j\leq k\)
range independently, subject to the displayed conditions.  Thus, when
\(i\neq j\), each base diameter pair \(\{p_\xi,p_\eta\}\) contributes both
\(\{v_i^\xi,v_j^\eta\}\) and \(\{v_i^\eta,v_j^\xi\}\).
Moreover,
\begin{equation}
 0<\alpha-\diam(\Stack_k(P))<\frac{2\alpha-\pi}{2k+1}
 <\frac{\alpha}{2k+1},
 \qquad \diam(\Stack_k(P))\longrightarrow\alpha.
 \label{eq:general-stack-limit}
\end{equation}

\item \emph{Stress lifting.} Suppose additionally that the diameter pairs of
\(P\) carry symmetric
weights \(\bigl(w_{\xi\eta}\bigr)_{\xi,\eta\in A}\), with
\(w_{\xi\eta}=w_{\eta\xi}\geq0\), not all zero,
with \(w_{\xi\eta}=0\) whenever \(\xi=\eta\) or
\(\{p_\xi,p_\eta\}\) is not a diameter pair, such that
\begin{equation}
 \sum_{\eta:p_\xi\cdot p_\eta=c}
 w_{\xi\eta}(p_\eta-c\,p_\xi)=0
 \qquad(\xi\in A).
 \label{eq:base-stack-equilibrium}
\end{equation}
Then the stack carries a nonzero nonnegative equilibrium stress, and therefore
\[
 \Lambda_k(\alpha)\in
 \Sigma_{k|P|+1}^{\specC}(\Sph^{m+1}).
\]
If
\(w_{\xi\eta}>0\) on every base diameter pair and every base vertex has positive
stressed degree, the lifted stress is positive on every diameter pair of the
stack.
\end{enumerate}
\end{theorem}

\begin{proof}
\noindent\emph{Diameter bounds.}
Put
\[
 D_k:=\pi-\varphi_1=\Lambda_k(\alpha).
\]
We first show that this candidate value is the diameter of the stack.
The strict ordering of the layer colatitudes makes the \(k\) layers
pairwise disjoint; each is a copy of \(P\), and none contains the apex.
Thus \(\Stack_k(P)\) has \(k|P|+1\) points.
The midpoint relation \(\varphi_k+\varphi_{k+1}=\pi\) and
\eqref{eq:general-stack-walk} give
\[
 \cos\varphi_1=-c-(1-c)\cos^2\varphi_k<-c=\cos\beta.
\]
Strict ordering gives \(\varphi_1\leq\varphi_k<\tfrac{\pi}{2}\).  Thus
\(\varphi_1>\beta\) and
\(\tfrac{\pi}{2}<D_k=\pi-\varphi_1<\alpha\).

\par\smallskip\noindent\emph{Diameter pairs.}
We next determine the diameter pairs.  For a fixed source layer \(i\), put
\[
 g_i(y):=c_i\cos y+c\,s_i\sin y-\cos D_k.
\]
Apply \eqref{eq:general-stack-walk} to the two steps adjacent
to \(\varphi_{2i}\), then use \eqref{eq:general-stack-symmetry} to
reflect their odd-indexed endpoints.  This gives the two consecutive
roots \(\ell_{k-i}\) and \(\ell_{k+1-i}\) of \(g_i\), where
\(\ell_0:=0\).  Moreover,
\[
 g_i(0)=\cos\ell_i-\cos D_k\geq0,
 \qquad
 g_i(\pi)=\cos\varphi_1-\cos\ell_i>0.
\]
A phase-shifted cosine has at most two roots on \([0,\pi]\).  It can
therefore be negative only between these consecutive roots, where no layer
colatitude occurs.  Since every base pair satisfies \(p_\xi\cdot p_\eta\geq c\),
\[
 v_i^\xi\cdot v_j^\eta
 \geq c_ic_j+c\,s_is_j\geq\cos D_k.
\]
Equality holds precisely for the layer and base pairs in
\eqref{eq:general-stack-active-pairs}.  The apex distances are the
colatitudes \(\ell_i\leq\ell_k=D_k\), with equality only on the last
layer.  This identifies the diameter and the pairs that realize it.

\par\smallskip\noindent\emph{An estimate uniform in the base diameter.}
It remains to prove the quantitative assertion in part~\textup{(i)}.  Choose
unit vectors \(e_0,e_1\) with \(e_0\cdot e_1=\cos\beta\), and on the two
corresponding meridians put
\[
 x_\nu(u):=(\sin u\,e_\nu,\cos u),\qquad \nu\in\{0,1\}.
\]
Equation~\eqref{eq:general-stack-walk} says that consecutive
alternating-meridian points \(x_{r\bmod2}(\varphi_r)\) and
\(x_{(r+1)\bmod2}(\varphi_{r+1})\) have
\(d_{m+1}\)-distance \(\varphi_1\).
Set \(\delta(u):=d_{m+1}(x_0(u),x_1(u))\) for points at equal
colatitude \(u\).  Then
\(\delta(u)\leq\beta\), since
\[
 \cos\delta(u)=\cos^2u+\cos\beta\sin^2u\geq\cos\beta.
\]
No colatitude \(\varphi_r\) equals \(\tfrac{\pi}{2}\): otherwise symmetry would give the
same value at the distinct index \(2k+1-r\), contrary to strict ordering.
Thus \(\delta(\varphi_r)<\beta\) for \(1\leq r\leq2k-1\).  The triangle
inequality gives
\[
 \varphi_{r+1}-\varphi_r
 \geq\varphi_1-\delta(\varphi_r)>\varphi_1-\beta
 \qquad(1\leq r\leq2k-1).
\]
Symmetry also gives
\(\varphi_{2k+1}-\varphi_{2k}
=\varphi_1-\varphi_0=\varphi_1\).  Consequently,
\[
 \pi
 >2\varphi_1+(2k-1)(\varphi_1-\beta)
 =2\beta+(2k+1)(\varphi_1-\beta).
\]
Since \(\varphi_1-\beta=\alpha-D_k\) and
\(\pi-2\beta=2\alpha-\pi\), this proves
\eqref{eq:general-stack-limit}.

\par\smallskip\noindent\emph{Constructing the lifted stress.}
The base equilibrium equations cancel the components tangent
to each parallel; we choose positive multipliers for the layer pairs
recursively to balance the meridional components.
Now suppose the weights in part~\textup{(ii)} are given.  Put
\[
 h_\xi:=\sum_\eta w_{\xi\eta}.
\]
Summing \eqref{eq:base-stack-equilibrium} over \(\xi\), and using symmetry,
gives
\begin{equation}
  \sum_\xi h_\xi p_\xi=0.
 \label{eq:base-stack-weighted-center}
\end{equation}
For a base diameter pair, write
\[
 u^P_{\xi\to\eta}:=\frac{p_\eta-c\,p_\xi}{\sin\alpha}
 \in T_{p_\xi}\Sph^m
\]
for its outgoing unit tangent at \(p_\xi\).
For \(1\leq i,j\leq k\) with \(i+j\in\{k,k+1\}\),
define the coefficients that will give the meridional components of
the diameter-edge unit tangents:
\[
 \gamma_{ij}
 :=\frac{s_ic_j-c\,c_is_j}{\sin D_k}.
\]
These coefficients satisfy
\begin{equation}
 \gamma_{i,k-i}>0>\gamma_{i,k+1-i}\quad(1\leq i<k),
 \qquad \gamma_{k1}<0.
 \label{eq:general-stack-gamma-signs}
\end{equation}
Indeed, write
\(Z_i(y):=s_i\cos y-c\,c_i\sin y\).  At the two roots
\(y_-:=\ell_{k-i}\) and \(y_+:=\ell_{k+1-i}\) of \(g_i\),
\[
 Z_i(y_-)Z_i(y_+)
 =\frac{\cos^2D_k(s_i^2+c^2\,c_i^2)-c^2}
        {c_i^2+c^2\,s_i^2}<0.
\]
Since \(\tfrac{\pi}{2}<D_k<\alpha\), we have
\(\cos^2D_k<c^2\); also \(s_i^2+c^2c_i^2\leq1\), so the
numerator is negative.  The nonzero linear combination \(Z_i\) of
sine and cosine has exactly one zero in \((0,\pi)\).  Together with
\(Z_i(0)>0>Z_i(\pi)\), this identifies the signs in their colatitude
order.

Put
\[
 r_0:=k,\quad r_1:=1,\quad r_2:=k-1,\quad r_3:=2,\quad\ldots,
 \quad r_{k-1}:=\left\lceil\frac{k}{2}\right\rceil.
\]
The active layer pairs in
\eqref{eq:general-stack-active-pairs} occur in the order
\[
 L_0-L_{r_0}-L_{r_1}-\cdots-L_{r_{k-1}}-L_{r_{k-1}}.
\]
Each edge records all diameter edges between the indicated
layers; the final edge records diameter edges within that layer.
We assign one symmetric factor \(\mu_{ij}\) to each pair of lifted layers.
At a fixed vertex, each corresponding base edge contributes once.

Start with \(\mu_{k1}:=1\) and
\(\kappa:=-\gamma_{k1}>0\), which balances the meridional component
at layer \(k\).  Then follow the displayed order.  At each successive
layer, the incoming factor is already fixed, and the following equation
determines the one remaining factor:
\begin{equation}
 \sum_{\substack{j\geq1:\ i+j\in\{k,k+1\}}}
 \mu_{ij}\gamma_{ij}
 +\mathbf1_{\{i=k\}}\kappa=0.
 \label{eq:general-stack-layer-balance}
\end{equation}
The two meridional coefficients at each successive layer
have opposite signs by \eqref{eq:general-stack-gamma-signs}, so every
new factor is positive, including the factor for edges within the final
layer.  When \(k=1\), the initial choices already give the within-layer
factor \(\mu_{11}=1\) and the spoke factor \(\kappa\); no recursive
step remains.

Give an active lifted edge the weight \(\mu_{ij}w_{\xi\eta}\), and give the
spoke \(\{Z,v_k^\xi\}\) weight \(\kappa h_\xi\).  At \(v_i^\xi\), let
\(\widehat m_{i,\xi}:=(-c_ip_\xi,s_i)\) be the meridional unit tangent.
The unit tangent toward \(v_j^\eta\) along a diameter edge is
\[
 \gamma_{ij}\widehat m_{i,\xi}
 +\frac{s_j\sin\alpha}{\sin D_k}
   (u^P_{\xi\to\eta},0).
\]
For each fixed target layer \(j\), the components tangent
to the parallel through \(v_i^\xi\) cancel by
\eqref{eq:base-stack-equilibrium}: every base term has the same factor
\(\mu_{ij}s_j\sin\alpha/\sin D_k\).  The remaining meridional
coefficient is \(h_\xi\) times the left side of
\eqref{eq:general-stack-layer-balance}, so it also vanishes.  At the apex the spoke tangents are
\((p_\xi,0)\), so \eqref{eq:base-stack-weighted-center} gives equilibrium.
The lifted stress is nonzero and nonnegative. If every base diameter pair has positive weight and every base vertex has positive stressed degree, then every diameter pair of the stack has positive weight.  Since \(D_k<\pi\),
\cref{prop:clarke-first-order} implies that
\(D_k\in\Sigma_{k|P|+1}^{\specC}(\Sph^{m+1})\).
\end{proof}

\begin{proposition}[Diameter transforms of Lov\'asz stacks]
\label{prop:canonical-lovasz-transform}
For \(\alpha\in(\tfrac{\pi}{2},\pi)\), the diameter transforms defined above
satisfy
\begin{equation}
 \frac{\pi}{2}<\Lambda_1(\alpha)<\Lambda_2(\alpha)<\cdots<\alpha.
 \label{eq:lovasz-transform-monotonicity}
\end{equation}
The function \(\Lambda_k\) is continuous on
\((\tfrac{\pi}{2},\pi)\) and extends continuously to \(\alpha=\pi\) by
\begin{equation}
 \Lambda_k(\pi):=\delta_k
 =\pi-\frac{\pi}{2k+1}.
 \label{eq:lovasz-transform-antipodal-extension}
\end{equation}
The estimate
\[
 0<\alpha-\Lambda_k(\alpha)<\frac{\alpha}{2k+1}
 <\frac{\pi}{2k+1}
\]
is uniform in the base diameter.  In particular, for any sequence
\(\alpha_k\in(\tfrac{\pi}{2},\pi)\),
\(\alpha_k-\Lambda_k(\alpha_k)\to0\).
\end{proposition}

The uniform estimate is \eqref{eq:general-stack-limit}; monotonicity and continuity are proved in
\cref{subsec:stack-transform-proof}.  The appendix also gives a
sharper convergence rate when the base diameter is fixed. For \(n\geq2\), the stack construction associates to each
Clarke critical diameter \(\alpha\in(0,\pi)\) of \(\Sph^{n-1}\) a strictly
increasing sequence of Clarke critical diameters \(\Lambda_j(\alpha)\)
of \(\Sph^n\) converging to \(\alpha\).  Thus each such \(\alpha\) becomes
an accumulation point of the spectrum in the next dimension.

\begin{theorem}[Spherical spectra under the Lov\'asz-stack construction]
\label{thm:lovasz-stack-spectrum-transport}
Let \(n\geq2\) and \(N\geq2\).
Equatorial inclusion preserves the Clarke critical values in
\((0,\pi)\) realized by \(N\)-tuples:
\begin{equation}
 \Sigma_N^{\specC}(\Sph^{n-1})\cap(0,\pi)
 \subseteq
 \Sigma_N^{\specC}(\Sph^n)\cap(0,\pi).
 \label{eq:equatorial-spectrum-nesting}
\end{equation}
If \(\alpha\in\Sigma_N^{\specC}(\Sph^{n-1})\cap(0,\pi)\), then for
every \(j\geq1\),
\begin{equation}
 \Lambda_j(\alpha)\in\Sigma_{jN+1}^{\specC}(\Sph^n),
 \qquad
 \Lambda_1(\alpha)<\Lambda_2(\alpha)<\cdots<\alpha,
 \qquad
 \Lambda_j(\alpha)\longrightarrow\alpha.
 \label{eq:fixed-label-stack-transport}
\end{equation}
Consequently, every value in
\(\Sigma^{\specC}(\Sph^{n-1})\cap(0,\pi)\) is both a Clarke critical value
on \(\Sph^n\) and an accumulation point of
\(\Sigma^{\specC}(\Sph^n)\cap(0,\pi)\):
\begin{equation}
 \Sigma^{\specC}(\Sph^{n-1})\cap(0,\pi)
 \subseteq
 \bigl[\Sigma^{\specC}(\Sph^n)\cap(0,\pi)\bigr]
 \cap
 \operatorname{Acc}\bigl(\Sigma^{\specC}(\Sph^n)\cap(0,\pi)\bigr)
 \qquad(n\geq2),
 \label{eq:stationary-spectrum-stack-lift}
\end{equation}
\end{theorem}

\phantomsection
\label{proof:lovasz-stack-hierarchy}
\begin{proof}[\normalfont\bfseries Proof of \cref{thm:lovasz-stack-hierarchy-intro} and \cref{thm:lovasz-stack-spectrum-transport}]
Fix
\[
 \alpha\in\Sigma_N^{\specC}(\Sph^{n-1})\cap(0,\pi).
\]
Since \(0<\alpha<\pi\), the characterization by an
equilibrium stress on the distinct points in
\cref{prop:spherical-spectrum-coincidence} applies to a Clarke-critical
realization.
The underlying set of this tuple is
\[
 P\subset\Sph^{n-1},
 \qquad
 |P|\leq N,
 \qquad
 \diam(P)=\alpha,
\]
whose diameter pairs carry a nonzero nonnegative equilibrium stress.

\par\smallskip\noindent\emph{Equatorial inclusion.}
The equatorial embedding \(p\mapsto(p,0)\) from
\(\Sph^{n-1}\) to \(\Sph^n\) preserves distances.  Since
\(\alpha<\pi\), every active minimizing arc is unique in both spheres and
its ambient unit tangent is the image of its equatorial unit tangent.
Thus the same stress remains in equilibrium in \(\Sph^n\).  Hence
\cref{cor:spherical-gordan-stress,prop:clarke-first-order} make \(P\)
Clarke critical in \(\Sph^n\).
If \(|P|<N\), repeat entries to obtain an \(N\)-tuple. Repetition preserves Clarke criticality and diameter by \cref{thm:fixed-label-clarke-spectrum}, so \(\alpha\in\Sigma_N^{\specC}(\Sph^n)\), proving \eqref{eq:equatorial-spectrum-nesting}.

\par\smallskip\noindent\emph{{Stack diameters.}}
The sharp positive minimum at \(N\) labels in
\cref{thm:first-spherical-clarke-value} gives
\begin{equation}
 \alpha\geq\zeta_{\min\{n-1,N-2\}}
 \geq\zeta_{n-1}
 =\arccos\left(-\frac{1}{n}\right)>\frac{\pi}{2},
 \label{eq:stack-transport-domain}
\end{equation}
In particular,
\(\alpha\in(\tfrac{\pi}{2},\pi)\), so the Lov\'asz-stack construction
applies.  By
\cref{thm:stationary-stack-lifting,prop:canonical-lovasz-transform},
\(\Stack_j(P)\) has diameter \(\Lambda_j(\alpha)\), has
\(j|P|+1\) points, and these diameters increase strictly to \(\alpha\).
The stack has \(j|P|+1\leq jN+1\) points. Repeating entries as necessary gives a Clarke-critical tuple with \(jN+1\) labels and the same diameter. Hence \(\Lambda_j(\alpha)\in\Sigma_{jN+1}^{\specC}(\Sph^n)\), as asserted in \eqref{eq:fixed-label-stack-transport}.

Taking the union over \(N\), and combining equatorial nesting with the
strictly increasing stack sequence, proves
\eqref{eq:stationary-spectrum-stack-lift}.
\end{proof}

\begin{corollary}[Iterated accumulation and simplex values]
\label{cor:stack-derived-consequences}
Iterating the Lov\'asz-stack construction gives:
\begin{equation}
 \bigl[\Sigma^{\specC}(\Sph^{n-\mu})\cap(0,\pi)\bigr]^{(\nu)}
 \subseteq
 \bigl[\Sigma^{\specC}(\Sph^n)\cap(0,\pi)\bigr]^{(\nu+\mu)}
 \qquad\bigl(n\geq2,\ \mu\in\{1,\ldots,n-1\},\
 \nu\in\mathbb Z_{\geq0}\bigr).
 \label{eq:stationary-spectrum-iterated-lift}
\end{equation}
In particular, for every \(m\geq1\),
\begin{equation}
 \zeta_{m-\nu}\in\bigl[\Sigma^{\specC}(\Sph^m)\bigr]^{(\nu)}
 \qquad\bigl(\nu\in\{0,\ldots,m\}\bigr).
 \label{eq:derived-zeta-membership}
\end{equation}
\end{corollary}

\begin{proof}
The accumulation operator preserves inclusion, directly from its
sequence definition.  Apply it successively to
\eqref{eq:stationary-spectrum-stack-lift} to obtain
\[
 \bigl[\Sigma^{\specC}(\Sph^{n-1})\cap(0,\pi)\bigr]^{(\nu)}
 \subseteq
 \bigl[\Sigma^{\specC}(\Sph^n)\cap(0,\pi)\bigr]^{(\nu+1)}
 \qquad\bigl(\nu\in\mathbb Z_{\geq0}\bigr),
\]
and iteration across \(\mu\) dimensions proves
\eqref{eq:stationary-spectrum-iterated-lift}.

For \(\nu=0\), the regular-simplex value \(\zeta_m\) belongs to
\(\Sigma^{\specC}(\Sph^m)\cap(0,\pi)\) by
\cref{thm:first-spherical-clarke-value}.  For \(1\leq\nu<m\), the same theorem places
\(\zeta_{m-\nu}\) in the positive nonantipodal spectrum of
\(\Sph^{m-\nu}\).  Applying
\eqref{eq:stationary-spectrum-iterated-lift} with \(n=m\), dimension shift \(\nu\),
and initial derivative order zero
gives \eqref{eq:derived-zeta-membership}.  For \(\nu=m\),
\cref{cor:fixed-label-circle-spectra} gives values
\(\delta_k\in\Sigma^{\specC}(\Sph^1)\cap(0,\pi)\) converging to
\(\pi=\zeta_0\).  Thus
\[
 \pi\in\bigl[\Sigma^{\specC}(\Sph^1)\cap(0,\pi)\bigr]^{(1)},
\]
and
\eqref{eq:stationary-spectrum-iterated-lift}, with \(n=m\),
dimension shift \(m-1\), and initial derivative order one, gives
\(\pi\in[\Sigma^{\specC}(\Sph^m)\cap(0,\pi)]^{(m)}\) when \(m\geq2\).
The case \(m=1\) is the initial endpoint statement.
\end{proof}

Unlike equatorial inclusion alone, stacking makes each fixed circle
value \(\delta_k\) an accumulation point, not merely another stationary
value.  We next compute the complete hierarchy within the constructed
family.

\subsection{Iterated stacks and their derived sets}
\label{subsec:explicit-stack-trees}

Starting with the regular odd polygons, we allow every layer number and
repeat the construction on each resulting base.  For this explicit
family, all iterated accumulation sets can be determined exactly.
These are the finite-order Cantor--Bendixson derivatives.
Recall that \(E^{(0)}=E\) and
\(E^{(\nu+1)}=\operatorname{Acc}(E^{(\nu)})\) for
\(\nu\in\mathbb Z_{\geq0}\), with accumulation taken in
\([0,\pi]\).  The resulting formulas also give inclusions for the
full spherical spectrum.

\begin{example}[Odd-polygon stacks]
\label{cor:odd-polygon-stack-accumulation}
For \(j,k\geq1\), the stack \(\Stack_j(R_{2k+1})\) has
\((2k+1)j+1\) points and diameter \(\Lambda_j(\delta_k)\).
Equal weights on the polygon's diameter cycle lift to a stress positive
on every diameter edge of the stack.  In particular,
\[
 \Lambda_j(\delta_k)\in\Sigma_{(2k+1)j+1}^{\specC}(\Sph^2),\qquad
 \Lambda_1(\delta_k)<\Lambda_2(\delta_k)<\cdots\nearrow\delta_k.
\]
These are the conclusions of
\cref{thm:stationary-stack-lifting,prop:canonical-lovasz-transform}.
\end{example}

The diameters obtained by iterating the construction form the sets
\[
 \mathcal T_0:=\{\pi\},
 \qquad
 \mathcal T_1:=\{\delta_k:k\geq1\},
 \qquad
 \mathcal T_{n+1}
 :=\{\Lambda_j(\alpha):\alpha\in\mathcal T_n,\ j\geq1\}
 \quad(n\geq1).
\]
The recursion has two distinct limiting mechanisms.  When the layer number
tends to infinity, the elementary estimate
\[
 0<\alpha-\Lambda_j(\alpha)<\frac{\pi}{2j+1}
\]
is uniform over the nonantipodal base diameter
\(\alpha\in(\tfrac{\pi}{2},\pi)\).  When the layer number is fixed and the
sequence of base diameters converges, convergence instead follows from continuity of
the extended map \(\Lambda_j\), including the endpoint identity
\(\Lambda_j(\pi)=\delta_j\).  These two mechanisms are summarized in
\cref{fig:lovasz-stack-tree}.

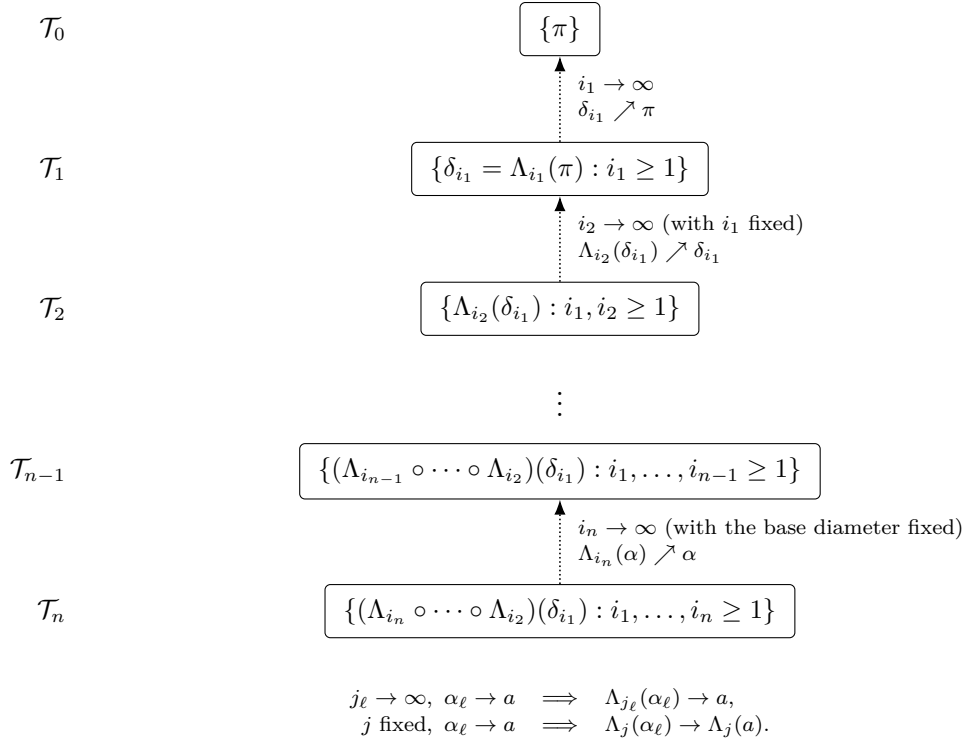
\begin{figure}[htbp!]
\centering
\begin{tikzpicture}[
  leveltag/.style={anchor=east,font=\small\bfseries},
  levelbox/.style={draw,rounded corners=2pt,
                   inner xsep=7pt,inner ysep=5pt,font=\small},
  convergence/.style={densely dotted,-{Latex[length=2mm]},semithick},
  limitlabel/.style={fill=white,inner sep=1.5pt,xshift=5pt,
                     font=\scriptsize,align=left},
  mechanism/.style={font=\scriptsize,align=center,anchor=north}
]
\node[leveltag] at (-6.4,0) {$\mathcal T_0$};
\node[levelbox] (T0) at (0,0) {$\{\pi\}$};

\node[leveltag] at (-6.4,-1.85) {$\mathcal T_1$};
\node[levelbox] (T1) at (0,-1.85)
 {$\{\delta_{i_1}=\Lambda_{i_1}(\pi):i_1\geq1\}$};
\draw[convergence] (T1.north)--
 node[limitlabel,right] {$i_1\to\infty$\\
                         $\delta_{i_1}\nearrow\pi$} (T0.south);

\node[leveltag] at (-6.4,-3.7) {$\mathcal T_2$};
\node[levelbox] (T2) at (0,-3.7)
 {$\{\Lambda_{i_2}(\delta_{i_1}):i_1,i_2\geq1\}$};
\draw[convergence] (T2.north)--
 node[limitlabel,right] {$i_2\to\infty$ (with $i_1$ fixed)\\
 $\Lambda_{i_2}(\delta_{i_1})\nearrow\delta_{i_1}$} (T1.south);

\node at (0,-4.85) {$\vdots$};

\node[leveltag] at (-6.4,-5.85) {$\mathcal T_{n-1}$};
\node[levelbox] (Tnm) at (0,-5.85)
 {$\{(\Lambda_{i_{n-1}}\circ\cdots\circ\Lambda_{i_2})
       (\delta_{i_1}):i_1,\ldots,i_{n-1}\geq1\}$};

\node[leveltag] at (-6.4,-7.7) {$\mathcal T_n$};
\node[levelbox] (Tn) at (0,-7.7)
 {$\{(\Lambda_{i_n}\circ\cdots\circ\Lambda_{i_2})
       (\delta_{i_1}):i_1,\ldots,i_n\geq1\}$};
\draw[convergence] (Tn.north)--
 node[limitlabel,right] {$i_n\to\infty$ (with the base diameter fixed)\\
 $\Lambda_{i_n}(\alpha)\nearrow\alpha$} (Tnm.south);
\node[mechanism] at (0,-8.55)
 {\(\begin{array}{@{}rcl@{}}
 j_\ell\to\infty,\ \alpha_\ell\to a
   &\Longrightarrow&
   \Lambda_{j_\ell}(\alpha_\ell)\to a,\\
 j\ \text{fixed},\ \alpha_\ell\to a
   &\Longrightarrow&
   \Lambda_j(\alpha_\ell)\to\Lambda_j(a).
 \end{array}\)};
\end{tikzpicture}
\caption{Diameters obtained by iterating the Lov\'asz-stack construction.
For each \(\alpha\in\mathcal T_{m-1}\), the values
\(\Lambda_j(\alpha)\in\mathcal T_m\) increase to \(\alpha\).
The dotted arrows indicate these limits.  The two implications below the
diagram distinguish a growing layer number from a fixed layer number.}
\label{fig:lovasz-stack-tree}
\end{figure}

\begin{theorem}[Derived sets of iterated stack diameters]
\label{thm:explicit-lovasz-stack-tree}
For every \(n\geq1\), the set \(\mathcal T_n\) is countably infinite
and is contained in \(\Sigma^{\specC}(\Sph^n)\).  Moreover,
\begin{equation}
 \operatorname{Acc}(\mathcal T_n)=\overline{\mathcal T_{n-1}},
 \label{eq:explicit-stack-tree-first-derived}
\end{equation}
where closure is taken in \([0,\pi]\).  Its iterated derived sets satisfy
\begin{equation}
 [\mathcal T_n]^{(\nu)}
 =\overline{\mathcal T_{n-\nu}}
 \quad\bigl(\nu\in\{1,\ldots,n\}\bigr),
 \qquad
 [\mathcal T_n]^{(n+1)}=\varnothing,
 \label{eq:explicit-stack-tree-derived-sets}
\end{equation}
and their minima are
\begin{equation}
 \min\bigl([\mathcal T_n]^{(\nu)}\bigr)=\zeta_{n-\nu}
 \qquad\bigl(\nu\in\{0,\ldots,n\}\bigr).
 \label{eq:explicit-stack-tree-derived-minima}
\end{equation}
\end{theorem}

The first identity also gives the closure formula
\begin{equation}
 \overline{\mathcal T_n}
 =\mathcal T_n\cup\mathcal T_{n-1}\cup\cdots\cup
  \mathcal T_1\cup\{\pi\},
 \label{eq:explicit-stack-tree-closure}
\end{equation}
Thus the compact countable set \(\overline{\mathcal T_n}\)
has Cantor--Bendixson rank \(n+1\), where the rank denotes the least
derivative order at which the set becomes empty.  Indeed,
\[
 [\overline{\mathcal T_n}]^{(n)}=\{\pi\},
 \qquad
 [\overline{\mathcal T_n}]^{(n+1)}=\varnothing.
\]

\begin{proof}
By \cref{cor:fixed-label-circle-spectra}, the circle values
\(\delta_k\) lie in \(\Sigma^{\specC}(\Sph^1)\cap(0,\pi)\).
Induction using
\eqref{eq:fixed-label-stack-transport} gives
\(\mathcal T_n\subseteq\Sigma^{\specC}(\Sph^n)\cap(0,\pi)\).  For a nonantipodal base value,
\eqref{eq:lovasz-transform-monotonicity} and
\eqref{eq:general-stack-limit} give a strictly increasing sequence of
stack values converging to that base value.  At \(n=1\),
\eqref{eq:lovasz-transform-antipodal-extension} gives the sequence
\(\Lambda_j(\pi)=\delta_j\) converging to \(\pi\).  Thus each
element of \(\mathcal T_{n-1}\) is the limit of its strictly increasing
\(\Lambda_j\)-sequence in \(\mathcal T_n\).  Since a derived set is
closed, this proves
\[
 \overline{\mathcal T_{n-1}}
 \subseteq\operatorname{Acc}(\mathcal T_n).
\]

For the reverse inclusion, the case \(n=1\) is
\(\operatorname{Acc}(\mathcal T_1)=\{\pi\}=\mathcal T_0\).
Now let \(n\geq2\) and assume the identity at level \(n-1\).  Let
\[
 r_\ell:=\Lambda_{j_\ell}(\alpha_\ell)\in\mathcal T_n
\]
be a sequence of distinct values converging to \(r\).  After passing to a
subsequence, either \(j_\ell\to\infty\) or \(j_\ell=j\) is constant.

\smallskip
\noindent\emph{The layer number tends to infinity.}
By \eqref{eq:general-stack-limit},
\[
 0<\alpha_\ell-r_\ell
 <\frac{\alpha_\ell}{2j_\ell+1}
 <\frac{\pi}{2j_\ell+1}.
\]
Thus \(\alpha_\ell-r_\ell\to0\), uniformly in the varying base diameter
\(\alpha_\ell\).  Since \(r_\ell\to r\), it follows that
\(\alpha_\ell\to r\).  Because
\(\alpha_\ell\in\mathcal T_{n-1}\), one has
\(r\in\overline{\mathcal T_{n-1}}\).

\smallskip
\noindent\emph{The layer number is fixed.}
Suppose \(j_\ell=j\).  Since the values \(r_\ell\) are distinct, the base diameters
\(\alpha_\ell\) are distinct.  Compactness of \([0,\pi]\) gives, after
passing to a further subsequence,
\[
 \alpha_\ell\longrightarrow\alpha.
\]
Because the base diameters are distinct,
\(\alpha\in\operatorname{Acc}(\mathcal T_{n-1})
=\overline{\mathcal T_{n-2}}\).
The positive spherical gap in
\cref{thm:first-spherical-clarke-value} gives
\(\alpha\geq\zeta_{n-1}>\pi/2\), so the limiting base diameter
remains in the domain of the extended transform.  Its continuity from
\cref{prop:canonical-lovasz-transform}, including
\(\Lambda_j(\pi)=\delta_j\) when \(\alpha=\pi\), gives
\[
 r=\Lambda_j(\alpha).
\]
Choose \(\widetilde\alpha_\mu\in\mathcal T_{n-2}\) with
\(\widetilde\alpha_\mu\to\alpha\).  A second application of continuity gives
\[
 \Lambda_j(\widetilde\alpha_\mu)\longrightarrow\Lambda_j(\alpha)=r.
\]
The recursion puts
\(\Lambda_j(\widetilde\alpha_\mu)\in\mathcal T_{n-1}\) when
\(n\geq3\).  When \(n=2\), the approximating sequence is the constant
value \(\pi\in\mathcal T_0\), and
\eqref{eq:lovasz-transform-antipodal-extension} instead puts its image
\(\delta_j\) in \(\mathcal T_1\).  In both cases,
\(r\in\overline{\mathcal T_{n-1}}\), proving
\eqref{eq:explicit-stack-tree-first-derived}.

For every subset \(E\) of a metric space,
\(\operatorname{Acc}(\overline E)=\operatorname{Acc}(E)\): a nearby point
of \(\overline E\) can be approximated by a point of \(E\) distinct from
the basepoint.  The identity
\(\overline E=E\cup\operatorname{Acc}(E)\), together with
\eqref{eq:explicit-stack-tree-first-derived}, gives
\eqref{eq:explicit-stack-tree-closure}.  Apply the accumulation operator repeatedly to
\eqref{eq:explicit-stack-tree-first-derived}, using
\(\operatorname{Acc}(\overline E)=\operatorname{Acc}(E)\) at each
step.  This gives \eqref{eq:explicit-stack-tree-derived-sets}; the
last derivative is empty because \(\mathcal T_0=\{\pi\}\).  By the sharp positive spherical minimum
\cref{thm:first-spherical-clarke-value}, every element of
\(\mathcal T_n\subseteq\Sigma^{\specC}(\Sph^n)\cap(0,\pi)\) is at least \(\zeta_n\). Equality is attained: start with
\(\delta_1=\zeta_1\) and apply the one-layer transform repeatedly, since,
for every \(j\geq1\),
\[
 \frac{\cos\zeta_j}{1-\cos\zeta_j}
 =\frac{-1/(j+1)}{1+1/(j+1)}
 =-\frac{1}{j+2}
 =\cos\zeta_{j+1}.
\]

Thus \(\zeta_n\in\mathcal T_n\) and \(\min\mathcal T_n=\zeta_n\).
Closure does not change this minimum, so
\eqref{eq:explicit-stack-tree-derived-minima} follows from
\eqref{eq:explicit-stack-tree-derived-sets}.

Countability follows from the recursion.  Every level is infinite because
\(j\mapsto\Lambda_j(\alpha)\) is strictly increasing for each fixed
base diameter \(\alpha\).
\end{proof}

\smallskip\noindent\emph{Consequences for the full spectrum.}
Derived-set monotonicity gives
\begin{equation}
 \mathcal T_{n-\nu}
 \subseteq[\Sigma^{\specC}(\Sph^n)]^{(\nu)}
 \quad\bigl(\nu\in\{0,\ldots,n-1\}\bigr),
 \qquad
 \pi\in[\Sigma^{\specC}(\Sph^n)]^{(n)}.
 \label{eq:explicit-stack-tree-derived-depth}
\end{equation}
For every \(\nu<n\), the \(\nu\)-th derivative of the full spectrum
therefore contains a countably infinite explicit subset, and its
\(n\)-th derivative is nonempty.  The higher derived sets of the
full spectrum are not determined by these inclusions.

Because the round metric is analytic,
\cref{cor:analytic-definable-fixed-label-finiteness} makes every
\(\Sigma_N^{\specC}(\Sph^n)\) finite.  Hence \(\Sigma^{\specC}(\Sph^n)\) and its
positive nonantipodal part \(\Sigma^{\specC}(\Sph^n)\cap(0,\pi)\) are countable;
\cref{thm:explicit-lovasz-stack-tree} shows that both are infinite.
Suppose the spectra for fixed numbers of labels are finite and nested. Along any sequence of distinct values, the minimum number of labels needed to realize each value tends to infinity.
In particular, if \(c_j\in\Sigma^{\specC}(\Sph^n)\) are distinct, then
\begin{equation}
 \min\{N\geq2:c_j\in\Sigma_N^{\specC}(\Sph^n)\}\longrightarrow\infty.
 \label{eq:spherical-accumulation-label-escape}
\end{equation}
Suppose infinitely many of the distinct values \(c_j\) could be realized with at most \(N_0\) labels. Repeating entries would realize each of them with exactly \(N_0\) labels. They would therefore all belong to the finite spectrum \(\Sigma_{N_0}^{\specC}(\Sph^n)\), a contradiction.
Every accumulation sequence is a special case.

\Cref{tab:lovasz-stack-derived-depth} illustrates these inclusions in the
first three dimensions.  Only the circle row describes the full derived
spectrum exactly.

\begin{table}[H]
\centering
\small
\renewcommand{\arraystretch}{1.35}
\begin{tabular}{
  c
  >{\raggedright\arraybackslash}p{0.31\linewidth}
  >{\raggedright\arraybackslash}p{0.50\linewidth}}
\toprule
Sphere & Explicit values in the spectrum & Values forced in successive derived sets\\
\midrule
\(\Sph^1\)
&
\(\delta_k\in\mathcal T_1\subseteq\Sigma^{\specC}(\Sph^1)\)
&
\([\Sigma^{\specC}(\Sph^1)]^{(1)}=\{\pi\}\)
\\
\(\Sph^2\)
&
\(\Lambda_j(\delta_k)\in\mathcal T_2
  \subseteq\Sigma^{\specC}(\Sph^2)\)
&
\(\{\delta_k:k\geq1\}\subseteq[\Sigma^{\specC}(\Sph^2)]^{(1)}\), and
\(\pi\in[\Sigma^{\specC}(\Sph^2)]^{(2)}\)
\\
\(\Sph^3\)
&
\(\Lambda_\ell(\Lambda_j(\delta_k))\in\mathcal T_3
  \subseteq\Sigma^{\specC}(\Sph^3)\)
&
\(\mathcal T_2\subseteq[\Sigma^{\specC}(\Sph^3)]^{(1)}\),
\(\mathcal T_1\subseteq[\Sigma^{\specC}(\Sph^3)]^{(2)}\), and
\(\pi\in[\Sigma^{\specC}(\Sph^3)]^{(3)}\)
\\
\bottomrule
\end{tabular}
\caption{The first three levels of the Lov\'asz-stack accumulation
hierarchy.  The general inclusions are given by
\eqref{eq:explicit-stack-tree-derived-depth}.
The equality in the circle row follows from
\cref{cor:fixed-label-circle-spectra}; in dimensions two
and three the table records explicit inclusions, not complete descriptions
of the derived sets of the full spherical spectra.}
\label{tab:lovasz-stack-derived-depth}
\end{table}

\begin{question}[Least points of the derived spectra]
\label{ques:derived-zeta-hierarchy}
For \(m\geq1\), we proved
\[
 \min\bigl(\Sigma^{\specC}(\Sph^m)\setminus\{0\}\bigr)=\zeta_m,\qquad
 \min\Bigl(\bigl[\Sigma^{\specC}(\Sph^m)\bigr]^{(1)}\Bigr)=\zeta_{m-1},
\]
as well as the membership
\(\zeta_{m-\nu}\in\bigl[\Sigma^{\specC}(\Sph^m)\bigr]^{(\nu)}\) for every
\(\nu\in\{0,\ldots,m\}\).
For \(\nu\in\{2,\ldots,m\}\), is
\[
 \min\Bigl(\bigl[\Sigma^{\specC}(\Sph^m)\bigr]^{(\nu)}\Bigr)=\zeta_{m-\nu}
\]
valid?
\end{question}

\begin{remark}[Accumulation and spectra in lower dimensions]
\label{rem:reverse-dimensional-boundary}
One potential approach to establishing the minima proposed in
the preceding question is to show that, for every \(m\geq2\),
\begin{equation}
 \operatorname{Acc}\bigl(\Sigma^{\specC}(\Sph^m)\bigr)\cap(0,\pi)
 \subseteq\Sigma^{\specC}(\Sph^{m-1})\cap(0,\pi)
 \label{eq:conjectural-spherical-accumulation-descent}
\end{equation}
The reverse inclusion already follows from the stack
construction (\cref{thm:lovasz-stack-spectrum-transport}). Thus this would
identify the accumulation values in \((0,\pi)\) with the stationary values
on the sphere one dimension lower.

If the proposed inclusion holds for every \(m\geq2\), then
repeated application shows that every point of the \(\nu\)-th derived set
lying in \((0,\pi)\) belongs to the spectrum of \(\Sph^{m-\nu}\), for
\(1\leq\nu<m\). Such a point would therefore be at
least \(\zeta_{m-\nu}\), the known least positive stationary value in that
dimension. Zero is isolated in the original spectrum and therefore belongs
to none of its positive-order derived sets. The final case \(\nu=m\) would reduce to the fact that the
circle spectrum has only the accumulation point \(\pi=\zeta_0\). Since the
stack construction already supplies these values
(\cref{cor:stack-derived-consequences}), this would establish the minima
proposed in the preceding question.

The inclusion remains open. To prove it, one must realize
each accumulation value \(\alpha\in(0,\pi)\) as the diameter of a finite
subset of \(\Sph^{m-1}\) carrying a nonzero nonnegative equilibrium stress
on its diameter graph.
\end{remark}

\begin{question}[Vietoris--Rips detection of Lov\'asz-stack values]
\label{ques:lovasz-stack-rips-detection}
Which of the transported values
\(\Lambda_j(\alpha)\), with
\(\alpha\in\Sigma^{\specC}(\Sph^{n-1})\cap(0,\pi)\), mark changes in the
strict Vietoris--Rips homotopy type of \(\Sph^n\)?
It is already known that the canonical inclusion \(\VR{\Sph^n}{r}\to\VR{\Sph^n}{s}\) fails to be a homotopy equivalence whenever \(0<r\leq\zeta_n<s\), as
recalled after \cref{cor:optimal-spherical-initial-chamber}.
On \(\Sph^2\), the value \(\Lambda_1(\delta_1)=\zeta_2\) is also a known homotopy
transition \cite[Corollary~7.5]{LimMemoliOkutan2024}.
The first remaining test family is
\(\Lambda_j(\delta_k)\) with \((j,k)\neq(1,1)\): which of these
values, or of their limiting scales \(\delta_k\), are transition values?
\end{question}

\section{Topological consequences and limits of diameter criticality}
\label{sec:spectral-mechanisms}

For compact metric spaces,
\cref{thm:metric-stationary-chamber-intro} turns gaps in the stationary
diameter spectrum into homotopy equivalences of canonical Vietoris--Rips
maps.  If \(Y\) is a one-Lipschitz retract of \(X\), weak-slope
stationary configurations in \(Y\) remain stationary in \(X\).
If a canonical Vietoris--Rips inclusion for \(Y\) fails to be a homotopy
equivalence, the corresponding inclusion for \(X\) also fails
(\cref{prop:weak-slope-nonexpansive-retract}).
A smooth distance-preserving embedding transfers stresses when the
ambient diameter pairs avoid the cut locus
(\cref{prop:stress-transfer-metric-isometric-embedding}).

The snowflaked interval in
\cref{ex:snowflake-stationary-without-transition} shows why stationarity
alone is insufficient to detect a topological transition: every attainable
diameter is stationary, although every canonical inclusion between
positive scales is a homotopy equivalence.

\subsection{Retractions and stress transfer}
A configuration that is stationary for diameter in a subspace need not remain stationary in the ambient space, where deformations may leave the subspace. We give two criteria ensuring that stationarity persists: one uses a one-Lipschitz retraction, while the other transfers equilibrium stresses through a smooth distance-preserving embedding
provided that the ambient distance function is smooth near every diameter pair.
\subsubsection{One-Lipschitz retractions} A one-Lipschitz retraction turns every local diameter descent
in the ambient space into a descent on the subspace, as in
\cref{lem:weak-slope-retraction}.  Stationarity in the subspace therefore
forces stationarity in the ambient space.  The inclusion and retraction
also induce maps of Vietoris--Rips complexes at each scale;
\cref{prop:weak-slope-nonexpansive-retract} uses these maps to transfer
failures of
the canonical Vietoris--Rips inclusions to be homotopy equivalences.

\begin{proposition}[Stationarity and Vietoris--Rips maps under a one-Lipschitz retraction]
\label{prop:weak-slope-nonexpansive-retract}
Let \((X,d_X)\) be a metric space, let \(Y\subseteq X\) carry the
restricted metric, and suppose that there is a one-Lipschitz retraction
\(\omega:X\to Y\).  Write \(\diam_N^X\) and \(\diam_N^Y\) for the diameter
functions on \(X^N\) and \(Y^N\), respectively.

\emph{Stationarity.} For every
\(N\geq2\) and \(y\in Y^N\),
\[
 \bigl|\diff\diam_N^Y\bigr|(y)
 \geq
 \bigl|\diff\diam_N^X\bigr|(y),
\]
where on the right \(y\) is regarded as a tuple in \(X^N\).  Consequently,
\[
 \Sigma_N^{\specWS}(Y)\subseteq\Sigma_N^{\specWS}(X),
 \qquad
 \Sigma^{\specWS}(Y)\subseteq\Sigma^{\specWS}(X).
\]

\emph{Canonical maps.} For every \(0<a<b\),
\[
 \begin{gathered}
  \VR{X}{a}\longrightarrow\VR{X}{b}
  \text{ is a homotopy equivalence}\\
  \Longrightarrow\\
  \VR{Y}{a}\longrightarrow\VR{Y}{b}
  \text{ is a homotopy equivalence}.
 \end{gathered}
\]

\end{proposition}

\begin{proof}
Let \(\iota_N:Y^N\to X^N\) and \(\omega_N:X^N\to Y^N\) be the coordinatewise
inclusion and retraction.  For the maximum product metrics, \(\iota_N\) is
isometric, \(\omega_N\) is one-Lipschitz, and
\(\omega_N\iota_N=\operatorname{id}_{Y^N}\).  Moreover,
\[
 \diam_N^Y(\omega_Nz)\leq\diam_N^X(z)
 \qquad(z\in X^N).
\]
Since also \(\diam_N^X\iota_N=\diam_N^Y\),
\cref{lem:weak-slope-retraction} gives the slope inequality.
A stationary tuple in \(Y^N\) is therefore stationary in \(X^N\);
preservation of diameter gives the spectral inclusions.

For the filtration assertion, let
\[
 I_t:\VR{Y}{t}\longrightarrow\VR{X}{t},
 \qquad
 P_t:\VR{X}{t}\longrightarrow\VR{Y}{t}
\]
be induced by the inclusion and the retraction, and let
\[
 F:\VR{X}{a}\longrightarrow\VR{X}{b},
 \qquad
 G:\VR{Y}{a}\longrightarrow\VR{Y}{b}
\]
be
the canonical inclusions.  Naturality gives
\[
 FI_a=I_bG,
 \qquad
 P_bF=GP_a,
 \qquad
 P_tI_t=\operatorname{id}.
\]
If \(h:\VR{X}{b}\to\VR{X}{a}\) is a homotopy inverse of \(F\), then
\(P_ahI_b\) is a homotopy inverse of \(G\):
\[
 (P_ahI_b)G=P_ahFI_a\simeq P_aI_a=\operatorname{id},
 \qquad
 G(P_ahI_b)=P_bFhI_b\simeq P_bI_b=\operatorname{id}.
\]

\end{proof}

\begin{corollary}[Circle-retract criterion]
\label{cor:circle-retract-critical-values}
Suppose the intrinsic circle of circumference \(L>0\) embeds isometrically
in a metric space \(X\),
and there is a one-Lipschitz retraction
from \(X\) onto its image.  Put \(r_k:=Lk/(2k+1)\).  For every \(k\geq1\),
\[
 r_k\in\Sigma_{2k+1}^{\specWS}(X),
\]
and every canonical inclusion
\[
 \VR{X}{r_k}\longrightarrow\VR{X}{u},\qquad u>r_k,
\]
fails to be a homotopy equivalence.
\end{corollary}

\begin{proof}
Rescaling \cref{cor:fixed-label-circle-spectra} gives
\(r_k=(L/(2\pi))\delta_k\) as a weak-slope stationary value at
\(2k+1\) labels.  The retraction transfers this value by
\cref{prop:weak-slope-nonexpansive-retract}.
For \(r_k<u<L/2\), the circle's canonical inclusion is not an
equivalence by \cref{cor:circle-canonical-rips-maps}.  For \(u\geq L/2\),
its target is contractible by \cref{rem:circle-antipodal-scale}, whereas
its source has homotopy type \(\Sph^{2k-1}\)
by
\cite[Theorem~7.4]{AdamaszekAdams2017}.  The retraction proposition
therefore excludes the ambient equivalence in every case.
\end{proof}

In \cref{cor:circle-retract-critical-values}, the values
\(r_k\) increase strictly to \(L/2\), so \(\Sigma^{\specWS}(X)\) is
infinite.  \Cref{lem:weak-slope-repetition-nesting} places \(r_k\) in
every \(\Sigma_N^{\specWS}(X)\) with \(N\geq2k+1\).
If \(X\) is a connected finite-dimensional Riemannian manifold,
\cref{prop:clarke-first-order} also places these values in
\(\Sigma_N^{\specC}(X)\).

\begin{example}[Circle products]
\label{ex:circle-product-canonical-obstructions}
Let \(L>0\), and let \(Z\) be a complete connected finite-dimensional Riemannian manifold
and fix \(z_*\in Z\).  In the Riemannian product
\[
 X=\left(\Sph^1,\frac{L}{2\pi}d_1\right)\times Z,
\]
the horizontal circle has the intrinsic circumference-\(L\) metric.
Projection \((\theta,z)\mapsto(\theta,z_*)\) is a one-Lipschitz
retraction, since circle-coordinate distance is at most product distance.
Thus \cref{cor:circle-retract-critical-values} supplies the stationary
values \(r_k=Lk/(2k+1)\) and
shows that \(\VR{X}{r_k}\to\VR{X}{u}\) fails to be a homotopy equivalence for every \(u>r_k\).  The values belong to both spectra for
\(N\geq2k+1\).  Taking \(Z\) to be another circumference-\(L\) circle
gives the square flat torus.
\end{example}

\subsubsection{Stress transfer and isometrically embedded circles}
A smooth distance-preserving embedding preserves equilibrium stresses,
with the same edge weights, when every ambient diameter pair avoids the
cut locus (\cref{prop:stress-transfer-metric-isometric-embedding}).
Distance preservation identifies the minimizing geodesics, so applying the differential of the embedding to each equilibrium equation gives the corresponding equation in the ambient tangent space.  The resulting ambient configurations are stationary.
This alone does not imply that any canonical Vietoris--Rips inclusion
fails to be a homotopy equivalence.

\begin{proposition}[Stress transfer under a distance-preserving embedding]
\label{prop:stress-transfer-metric-isometric-embedding}
Let \(M\) and \(\widetilde M\) be complete connected finite-dimensional
Riemannian manifolds, and let
\[
 \iota:M\longrightarrow\widetilde M
\]
be a smooth embedding satisfying
\[
 d_{\widetilde M}\bigl(\iota(x),\iota(x')\bigr)=d_M(x,x')
 \qquad(x,x'\in M).
\]
Fix \(N\geq2\), and let \(x\in M^N\) have positive
diameter and be first-order stationary for diameter.  Suppose that every
diameter pair of \(\iota^N(x)\) is off the cut locus of \(\widetilde M\).
Then \(\iota^N(x)\) is first-order stationary, weak-slope stationary, and
Clarke critical for diameter on \(\widetilde M^N\).  In particular,
\[
 \diam_N(x)
 \in
 \Sigma_N^{\specWS}(\widetilde M)\cap\Sigma_N^{\specC}(\widetilde M).
\]
Every nonzero nonnegative equilibrium stress on \(x\) transfers to
\(\iota^N(x)\) with the same edge weights.
\end{proposition}

\begin{proof}
Because \(\iota\) preserves all distances, \(x\) and \(\iota^N(x)\) have the same diameter, realized by the same pairs of indices.  Each active intrinsic
distance function is the pullback of the corresponding smooth ambient
distance function, so it is smooth near the diameter pair.  Write
\(u^M_{i\to j}\) and \(u^{\widetilde M}_{i\to j}\) for the initial
unit tangents of the intrinsic and ambient minimizing geodesics,
respectively.  By the convex-hull criterion in
\cref{prop:clarke-first-order} and the
first-variation formula, first-order stationarity of \(x\) supplies
symmetric coefficients \(\bigl(w_{ij}\bigr)_{1\leq i,j\leq N}\), with
\(w_{ij}=w_{ji}\geq0\), not all zero
and supported on its diameter pairs, such that
\[
 \sum_{j\neq i}w_{ij}u^M_{i\to j}=0
 \qquad(1\leq i\leq N).
\]

For each diameter pair, the image under \(\iota\) of its minimizing geodesic
in \(M\) has length equal to the ambient distance between its endpoints,
and is therefore minimizing in \(\widetilde M\).
Since neither endpoint lies in the cut locus of the other, that ambient minimizing geodesic is unique.  Hence
\[
 u^{\widetilde M}_{i\to j}
 =\diff\iota_{x_i}\bigl(u^M_{i\to j}\bigr).
\]
Applying \(\diff\iota_{x_i}\) to the preceding balance gives
\[
 \sum_{j\neq i}w_{ij}u^{\widetilde M}_{i\to j}
 =\diff\iota_{x_i}
   \left(\sum_{j\neq i}w_{ij}u^M_{i\to j}\right)
 =0
 \qquad(1\leq i\leq N).
\]
The first-variation formula now turns these equations into
a linear relation among the active distance differentials on \(\widetilde M^N\) with nonnegative coefficients, not all zero.  Thus the image tuple has no common first-order descent
direction.  The weak-slope and Clarke conclusions follow from
\cref{prop:clarke-first-order}.
\end{proof}

\begin{corollary}[Odd-polygon values from an isometrically embedded circle]
\label{prop:off-cut-geodesic-circle-ladder}
Let \(M\) be a complete connected finite-dimensional Riemannian manifold.
Suppose
that, for some \(L>0\), there is a smooth embedding
\[
 \gamma:\Sph^1\longrightarrow M
\]
which is an isometry from
\(\bigl(\Sph^1,\tfrac{L}{2\pi}d_1\bigr)\) onto its image, equipped with
the restricted metric \(d_M\), and such that
\[
 \gamma(z')\notin\operatorname{Cut}_M(\gamma(z))
 \qquad
 \text{whenever }0<d_1(z,z')<\pi,
\]
where \(\operatorname{Cut}_M(x)\) denotes the cut locus of \(x\) in \(M\).
Then, for every \(k\geq1\),
\[
 \tfrac{L}{2\pi}\delta_k
 \in
 \Sigma_{2k+1}^{\specWS}(M)\cap\Sigma_{2k+1}^{\specC}(M).
\]
The values increase strictly to \(L/2\), an accumulation point of both
all-label spectra.
\end{corollary}

\begin{proof}
After rescaling \cref{cor:fixed-label-circle-spectra}, the regular
\((2k+1)\)-gon in
\(\bigl(\Sph^1,\tfrac{L}{2\pi}d_1\bigr)\) is first-order stationary and
has diameter \(\tfrac{L}{2\pi}\delta_k<L/2\).  Every one of its diameter pairs has \(d_1\)-distance \(\delta_k<\pi\), so its image under \(\gamma\)
lies off the ambient cut locus by hypothesis.  Apply
\cref{prop:stress-transfer-metric-isometric-embedding}.
Strict monotonicity and convergence follow from \(\delta_k\nearrow\pi\).
\end{proof}

For a particular \(k\), it suffices to check the ambient cut-locus
condition only on the diameter pairs of the chosen regular polygon.

\begin{example}[An equatorial circle without a retraction]
\label{ex:equatorial-stress-without-retraction}
The equatorial inclusion \(\iota:\Sph^1\hookrightarrow\Sph^2\) preserves
the intrinsic round distances.  Every nonantipodal equatorial pair lies
off the cut locus of \(\Sph^2\), so
\cref{prop:off-cut-geodesic-circle-ladder} gives
\[
 \delta_k\in
 \Sigma_{2k+1}^{\specWS}(\Sph^2)\cap
 \Sigma_{2k+1}^{\specC}(\Sph^2)
 \qquad(k\geq1).
\]
There is no continuous retraction \(\omega:\Sph^2\to\Sph^1\) of this
inclusion.  Otherwise \(\omega\iota=\operatorname{id}_{\Sph^1}\) would imply
\[
 \omega_\ast\iota_\ast
 =\operatorname{id}_{H_1(\Sph^1;\mathbb Z)},
\]
whereas \(\iota_\ast\) factors through
\(H_1(\Sph^2;\mathbb Z)=0\).
Thus the stress-transfer proposition applies, although the retraction criterion cannot be used here to show that a canonical Vietoris--Rips inclusion fails to be a homotopy equivalence.
\end{example}

\subsection{Limits of the criterion and quantitative contractions}

On \([0,1]\) with \(d(s,t)=|s-t|^\alpha\), \(0<\alpha<1\), every
tuple is weak-slope stationary for diameter, while every canonical
Vietoris--Rips inclusion between positive scales is a homotopy equivalence.

\begin{example}[Weak-slope stationarity without Vietoris--Rips transitions]
\label{ex:snowflake-stationary-without-transition}
Let \(0<\alpha<1\) and
\[
 X=([0,1],d_X),\qquad d_X(s,t)=|s-t|^\alpha.
\]
For every \(N\geq2\), every tuple in \(X^N\) is weak-slope stationary
for diameter, and
\[
 \Sigma_N^{\specWS}(X)=\Sigma^{\specWS}(X)=[0,1].
\]
Nevertheless, \(\VR{X}{r}\) is contractible for every \(r>0\), and
every canonical inclusion \(\VR{X}{r}\to\VR{X}{s}\), \(0<r<s\),
is a homotopy equivalence.

\par\smallskip\noindent To prove weak-slope stationarity,
we bound the possible decrease in diameter under a small displacement.
Fix \(x=(x_1,\ldots,x_N)\in X^N\) and let
\(a=\max_{i,j}|x_i-x_j|\) be its Euclidean diameter.  If \(a=0\),
then \(x\) is a global minimum of \(\diam_N\).  Suppose \(a>0\).
For any \(z\in X^N\) satisfying
\(d_{X,\infty}(z,x)\leq u\), each coordinate moves by at most
\(u^{1/\alpha}\) in the Euclidean metric.  A pair realizing \(a\)
therefore gives, for all sufficiently small \(u>0\),
\[
 \diam_N(x)-\diam_N(z)
 \leq a^\alpha-\bigl(a-2u^{1/\alpha}\bigr)^\alpha
 =O(u^{1/\alpha})=o(u).
\]
A weak-slope deformation with rate \(\sigma>0\), evaluated at \(x\),
would instead give a decrease of at least \(\sigma u\) with displacement
at most \(u\), a contradiction.  Thus every tuple is stationary.  For
\(N\geq2\), every value in \([0,1]\) is realized by a pair
\((0,a)\), with repeated labels if necessary, since its diameter is
\(a^\alpha\).

\par\smallskip\noindent To determine the Vietoris--Rips
homotopy type, we contract the interval while keeping pairwise distances
nonincreasing.
The continuous homotopy
\(F_t(x)=(1-t)x\) fixes \(0\), starts at the identity, ends at the
constant map, and satisfies
\[
 d_X(F_t(x),F_t(x'))=(1-t)^\alpha d_X(x,x').
\]
It is a crushing onto \(\{0\}\) in Hausmann's sense, with time
reversed.  Hence \(\VR{X}{r}\) is contractible for every \(r>0\)
by \cite[Proposition~2.2]{Hausmann1995}.  Every map between nonempty
contractible spaces is a homotopy equivalence, which proves the assertion
for the canonical inclusions.
\end{example}

The snowflake metric is not a length metric: any curve from \(0\) to \(1\),
subdivided at successive crossings of \(j/m\), has length at least
\(m^{1-\alpha}\) for every integer \(m\geq1\), and hence infinite length.
This contrasts with the closed smooth Riemannian setting.
\Cref{prop:clarke-first-order,thm:fixed-label-clarke-spectrum} give
\(\Sigma_N^{\specWS}(M)\subseteq\Sigma_N^{\specC}(M)\) and
\(\dim_{\mathrm H}\Sigma_N^{\specC}(M)=0\) for every \(N\geq2\).
Thus the interval of stationary values exhibited here cannot occur in
that setting.  This restriction concerns the size of the stationary
spectrum; it does not assert that every stationary value marks a
Vietoris--Rips transition.

\begin{remark}[Crushing versus controlled descent]
\label{rem:hausmann-crushing-comparison}
After reversing time, a Hausmann crushing of \(X\) onto \(A\subseteq X\)
is a continuous homotopy \(F_t:X\to X\) with
\[
 F_0=\operatorname{id}_X,\qquad F_1(X)\subseteq A,\qquad
 F_t|_A=\operatorname{id}_A,
\]
whose pairwise distances never increase:
\[
 d_X(F_s(x),F_s(x'))\leq d_X(F_t(x),F_t(x'))
 \qquad(0\leq t\leq s\leq1).
\]
Hausmann's theorem makes \(\VR{A}{r}\to\VR{X}{r}\) a homotopy
equivalence for every \(r>0\)
\cite[Definition and Proposition~2.2]{Hausmann1995}.
Crushing imposes no linear relation between decrease and displacement.
The snowflake example shows that it can coexist with zero weak slope at
every tuple.  The next criterion adds precisely such quantitative control.
\end{remark}

\subsubsection{{Excluding positive stationary values by contractions}}
 To use a contraction to exclude stationarity at a positive
diameter, we must control both diameter decrease and displacement of the
labels.  In \cref{prop:uniform-contraction-no-stationarity}, a
uniform contraction of pairwise distances supplies the decrease, while a
linear displacement bound permits the time rescaling required by
\cref{def:weak-slope}.

\begin{proposition}[Quantitative contractions exclude positive stationary diameters]
\label{prop:uniform-contraction-no-stationarity}
Let \((X,d_X)\) be a nonempty metric space.  Let
\(A,\lambda,\tau>0\) satisfy \(\lambda\tau\leq1\), and suppose there is a
continuous map
\[
 C:X\times[0,\tau]\longrightarrow X
\]
such that, writing \(C_t(x):=C(x,t)\), one has
\(C_0=\operatorname{id}_X\) and, for every \(x,x'\in X\) and
\(0\leq t\leq\tau\),
\begin{equation}
 d_X(C_t(x),x)\leq At,
 \qquad
 d_X(C_t(x),C_t(x'))\leq(1-\lambda t)d_X(x,x').
 \label{eq:uniform-metric-contraction}
\end{equation}
Then, for every \(N\geq2\) and every \(x\in X^N\) of positive diameter,
\[
 |\diff\diam_N|(x)\geq\frac{\lambda\diam_N(x)}{A}>0.
\]
In particular,
\[
 \Sigma_N^{\specWS}(X)=\{0\}
 \quad(N\geq2),
 \qquad
 \Sigma^{\specWS}(X)=\{0\}.
\]
\end{proposition}

\begin{proof}
Put \(D:=\diam_N(x)>0\).
Let \(0\leq\sigma<\lambda D/A\).  Choose \(\delta>0\) so small that
\[
 \delta\leq A\tau,
 \qquad
 \frac{\lambda(D-2\delta)}{A}>\sigma.
\]
For \(y=(y_1,\ldots,y_N)\in B_\delta(x)\) and
\(0\leq u\leq\delta\), set
\[
 H(y,u):=
 \bigl(C_{u/A}(y_1),\ldots,C_{u/A}(y_N)\bigr).
\]
The first inequality in \eqref{eq:uniform-metric-contraction} gives
\[
 d_{X,\infty}(H(y,u),y)\leq u.
\]
Since \(\diam_N\) is two-Lipschitz for the maximum product metric,
\(\diam_N(y)\geq D-2\delta\).  The second inequality in
\eqref{eq:uniform-metric-contraction} therefore gives
\begin{align*}
 \diam_N(H(y,u))
 &\leq\left(1-\frac{\lambda u}{A}\right)\diam_N(y)\\
 &\leq\diam_N(y)-\frac{\lambda(D-2\delta)}{A}u
 \leq\diam_N(y)-\sigma u.
\end{align*}
Thus \(\sigma\leq|\diff\diam_N|(x)\).  Letting
\(\sigma\nearrow\lambda D/A\) proves the estimate.  Constant tuples have
diameter zero and are weak-slope stationary, which proves the spectral
conclusion.
\end{proof}

The proof only uses a displacement bound near the tuple under
consideration.  Radial geodesic contractions provide such local bounds,
even when the space is unbounded.

\begin{corollary}[
Geodesic contractions exclude positive stationary diameters]
\label{prop:busemann-contraction-no-stationarity}
Let \(X\) be a nonempty metric space with a point \(x_*\) and a continuous
map \(\Gamma:X\times[0,1]\to X\).
Suppose that, for every \(x\in X\), the map
\(s\mapsto\Gamma(x,s)\) is a constant-speed geodesic from \(x_*\) to
\(x\), and that
\begin{equation}
 d_X(\Gamma(x,s),\Gamma(x',s))\leq s\,d_X(x,x')
 \qquad(x,x'\in X,\ 0\leq s\leq1).
 \label{eq:conical-geodesic-contraction}
\end{equation}
Then \(\Sigma_N^{\specWS}(X)=\{0\}\) for every \(N\geq2\), and hence
\(\Sigma^{\specWS}(X)=\{0\}\), without a boundedness assumption.
\end{corollary}

\begin{proof}
Fix a tuple \(x\) of diameter \(D>0\) and choose
\(R>\max_i d_X(x_*,x_i)\).  On a sufficiently small tuple neighborhood,
every coordinate remains within distance \(R\) of \(x_*\).
The time rescaling in the proof of
\cref{prop:uniform-contraction-no-stationarity} applies locally to
\(C_t(y)=\Gamma(y,1-t)\), with \(A=R\) and \(\lambda=1\).
Explicitly, fix \(\sigma<D/R\) and choose the neighborhood radius
\(0<\delta\leq R\) so that \((D-2\delta)/R>\sigma\).  For
\(y\in B_\delta(x)\) and \(0\leq u\leq\delta\), the map
\[
 H(y,u)=\bigl(\Gamma(y_1,1-u/R),\ldots,\Gamma(y_N,1-u/R)\bigr)
\]
has displacement at most \(u\) and diameter at most
\((1-u/R)\diam_N(y)\).  Since \(\diam_N(y)\geq D-2\delta\), it
decreases diameter at rate at least \(\sigma\).  Thus every
\(\sigma<D/R\) is admissible and \(|\diff\diam_N|(x)\geq D/R>0\).
\end{proof}

\begin{remark}[Examples and the role of boundedness]
\label{rem:metric-tree-boundedness}
A convex subset of a normed vector space uses
\(\Gamma(x,s)=(1-s)x_*+sx\), with equality in \eqref{eq:conical-geodesic-contraction}.
A \emph{Busemann-convex} geodesic space is one in which the distance
between any two constant-speed geodesics is convex in their common
parameter \cite[Section~1]{DescombesLang2015}.
This implies uniqueness of geodesics by applying convexity to two
geodesics with the same endpoints.  With the common initial point
\(x_*\), it gives \eqref{eq:conical-geodesic-contraction}.  Joint
continuity follows from this bound and constant speed:
\[
 d_X(\Gamma(x,s),\Gamma(y,t))
 \leq s\,d_X(x,y)+|s-t|\,d_X(x_*,y).
\]
Thus \cref{prop:busemann-contraction-no-stationarity}
excludes all positive weak-slope stationary diameter values in these
spaces.  In particular, \(\operatorname{CAT}(0)\)
spaces are Busemann convex
\cite[Proposition~II.2.2]{BridsonHaefliger1999}.
An \(\R\)-tree is a uniquely geodesic metric space whose geodesic
triangles are tripods \cite[Example~II.1.15(5)]{BridsonHaefliger1999}; along the tripod
rooted at \(x_*\), radial contraction satisfies the same inequality.  This argument does not require completeness or boundedness.

If \(X\) is bounded, \(C_t(x)=\Gamma(x,1-t)\) also satisfies
\eqref{eq:uniform-metric-contraction} with \(\lambda=\tau=1\) and any positive
\(A\geq\sup_x d_X(x_*,x)\).  Conversely, the global criterion forces
boundedness, since for \(0<t\leq\tau\),
\[
 d_X(x,x')\leq2At+(1-\lambda t)d_X(x,x')
 \quad\Longrightarrow\quad\diam(X)\leq2A/\lambda.
\]
Unbounded examples therefore use the local corollary, not the global
proposition.
\end{remark}

For a compact space \(X\) satisfying the global quantitative
criterion, \cref{prop:uniform-contraction-no-stationarity,%
thm:metric-stationary-chamber-intro} imply that every canonical
Vietoris--Rips inclusion
\[
 \VR{X}{r}\hookrightarrow\VR{X}{s},
 \qquad 0<r<s,
\]
is a homotopy equivalence.  Comparing with a scale larger than
\(\diam(X)\), where the complex is a full simplex, shows that every
positive-scale complex is contractible.

\subsubsection{Spaces with no nonconstant paths}

The snowflaked interval has nonconstant paths, but its metric prevents
linear descent relative to displacement.  The absence of nonconstant
paths gives another reason for every weak slope to vanish.
\begin{proposition}[Spaces with no nonconstant paths]
\label{prop:path-rigidity-weak-slope}
Let \(X\) be a metric space having no nonconstant continuous path.
Then every continuous function \(f:X\to\R\) satisfies
\begin{equation}
 |\diff f|\equiv0.
 \label{lem:weak-slope-path-rigidity}
\end{equation}
If \(X\) is nonempty, then, for every \(N\geq1\), every tuple in
\(X^N\) is weak-slope stationary for \(\diam_N\), and
\[
 \Sigma_1^{\specWS}(X)=\{0\},\qquad
 \Sigma_N^{\specWS}(X)=\{d_X(x,x'):x,x'\in X\}
 \quad(N\geq2).
\]
If \(X\) is infinite and compact, then
the set of distances in \(X\) is infinite and has \(0\) as an
accumulation point.
\end{proposition}

\begin{proof}
For any weak-slope deformation \(H\) near a point of \(X\), each map
\(u\mapsto H(y,u)\) is a continuous path starting at \(y\), by the
displacement inequality at \(u=0\).  It must be constant.  A positive
rate \(\sigma\) would then force
\(f(y)\leq f(y)-\sigma u\) for \(u>0\), a contradiction.  This proves
\eqref{lem:weak-slope-path-rigidity}.

Every path in \(X^N\) is also constant, since all its coordinate
projections are constant.  Applying the preceding argument to
\(\diam_N:X^N\to\R\) proves stationarity of every tuple.  For
\(N\geq2\), a tuple's diameter is one of its pairwise distances;
conversely, every distance is realized by a pair with repeated labels if
necessary.  This gives the spectral identities.
Finally, an infinite compact metric space has a nonisolated point
\(x\).  Distinct points \(x_j\to x\) give positive distances
\(d_X(x_j,x)\to0\), with distinct values after passing to a subsequence.
\end{proof}

\begin{remark}[Totally disconnected spaces and the limits of the criterion]
\label{rem:path-rigidity-limitation}
Every totally disconnected metric space has no nonconstant path.  For the
middle-thirds Cantor set \(C\subset[0,1]\), one has
\(C-C=[-1,1]\).  Indeed, every \(t\in[-1,1]\) has a balanced ternary
expansion \(t=2\sum_{j\geq1}\varepsilon_j3^{-j}\), with
\(\varepsilon_j\in\{-1,0,1\}\): successively choose one of the three
closed subintervals of relative length \(1/3\).  Writing
\(\varepsilon_j=a_j-b_j\), with \(a_j,b_j\in\{0,1\}\), expresses
\(t\) as the difference of two points of \(C\).  Therefore
\[
 \Sigma_N^{\specWS}(C)=[0,1]
 \qquad(N\geq2).
\]
For these spaces, the hypothesis of
\cref{thm:metric-stationary-chamber-intro} says that no distance belongs
to \([r,s)\); hence \(\VR{X}{r}=\VR{X}{s}\).
Finite metric spaces are examples.  Equivalences between different
complexes in their filtrations require another argument, such as a
simplicial collapse.

The examples isolate three different situations.  Quantitative
contraction excludes every positive stationary value.  When every path
is constant, every attained diameter is stationary.  For the snowflaked interval, every attained diameter is again stationary, although nonconstant paths exist and its Vietoris--Rips complexes are contractible at every positive scale.
Thus the spectrum measures the availability of controlled descent; it is
not an exact detector of changes in Vietoris--Rips topology.
\end{remark}

\appendix
\crefalias{section}{appendix}
\crefalias{subsection}{subappendix}
\crefalias{subsubsection}{subsubappendix}

\section{Weak-slope comparisons and the antipodal case}
\label{app:weak-clarke-comparison}

We prove the comparisons stated in
\cref{thm:weak-implies-clarke,lem:weak-slope-bilipschitz,prop:weak-slope-finite-smooth-maximum}.
We then give the smooth metric construction of
\cref{prop:cantor-clarke-spectrum} and the antipodal argument for
\cref{lem:two-label-spherical-spectrum}.

Recall the Clarke subdifferential from
\cref{def:clarke-subdifferential} and its support-function description
\cite[Section~2.1]{Clarke1990}.  For a \(C^1\) function it is
\(\{\diff f_x\}\); for a locally Lipschitz function it records limiting
differentials.  In a Euclidean chart its support-function description is
\[
 h^\circ(z;v):=\limsup_{\substack{y\to z\\t\downarrow0}}
       \frac{h(y+tv)-h(y)}{t},
 \qquad
 \partial^{\mathrm C}h(z)
 =\{\xi:\xi[v]\leq h^\circ(z;v)\text{ for every }v\}.
\]
This latter definition also applies in a Banach space, with
\(\xi\in E^*\).

\begin{proof}[Proof of \cref{thm:weak-implies-clarke}]
We transfer Degiovanni--Marzocchi's Banach-space inequality
\cite[Theorem~2.17]{DegiovanniMarzocchi1994} to a normal chart at \(x\)
and compare the chart metric with \(d_M\).  The cited theorem states
that, if \(A\) is open in a real Banach space \(E\) and
\(h:A\to\R\) is locally Lipschitz, then
\[
 |\diff h|_{\lVert\cdot\rVert_E}(z)
 \geq\min_{\xi\in\partial^{\mathrm C}h(z)}\lVert\xi\rVert_{E^*}.
\]
Here the subdifferential uses the generalized directional derivative and
\(\lVert\cdot\rVert_{E^*}\) is the dual norm.

Fix \(\varepsilon>0\).  Choose a sufficiently small normal chart
\(\phi:U\to\R^{\dim M}\) centered at \(x\), with
\(\diff\phi_x\) an isometry, so that both \(\phi\) and
\(\phi^{-1}\) are \((1+\varepsilon)\)-Lipschitz.
Indeed, the metric coefficients approach their value at the center;
on a smaller convex normal neighborhood, integration along straight
coordinate segments and minimizing geodesics gives the two distance
bounds.  Transfer a weak-slope deformation for \(h=f\circ\phi^{-1}\)
back to \(M\), shrinking its domain to remain in this chart and replacing
its time parameter \(u\) by \(u/(1+\varepsilon)\).  Its displacement is
then at most \(u\), so
\[
 |\diff f|_{d_M}(x)
 \geq\frac{|\diff h|_{\mathrm{Eucl}}(\phi(x))}{1+\varepsilon}
 \geq\frac{\operatorname{dist}_{g_x^*}
              (0,\partial^{\mathrm C}f(x))}{1+\varepsilon}.
\]
By \cref{def:clarke-subdifferential}, the pullback
\((\diff\phi_x)^*\) identifies
\(\partial^{\mathrm C}h(\phi(x))\) with
\(\partial^{\mathrm C}f(x)\).  This identification preserves the
cotangent norms because \(\diff\phi_x\) is an isometry.
Let \(\varepsilon\downarrow0\).
\end{proof}

\begin{proof}[Proof of \cref{lem:weak-slope-bilipschitz}]
Choose a neighborhood on which \(a\,d\leq d'\leq b\,d\), with
\(a,b>0\).  If \(H\) witnesses a positive rate \(\sigma\) for \(d\),
restrict its spatial and time domains so that its image remains in
that neighborhood.  On a sufficiently small \(d'\)-ball the map
\(H'(y,u)=H(y,u/b)\) is defined, moves points by at most \(u\) in
\(d'\), and lowers \(f\) by at least \(\sigma u/b\).
Thus positive \(d\)-slope implies positive \(d'\)-slope.
Interchanging the metrics proves the converse.
\end{proof}

\subsection{Finite maxima of smooth functions}
\label{sec:finite-smooth-maxima}

Let \(f=\max_{\alpha\in\mathcal A}f_\alpha\), with active set
\(I_f(z)\) and convex hull \(K_z\) of active differentials as in
\cref{prop:weak-slope-finite-smooth-maximum}.

\begin{proof}[Proof of \cref{prop:weak-slope-finite-smooth-maximum}]
After excluding the inactive branches, Taylor expansion along
\(\exp_z(hv)\) gives \eqref{eq:dini-finite-maximum}.  At a
differentiability point \(y\) of \(f\), each active branch satisfies
\(\diff(f_\alpha)_y=\diff f_y\), since \(f-f_\alpha\) has a local
minimum at \(y\).  Since there are finitely many branches and their differentials vary continuously, it follows that \(\partial^{\mathrm C}f(z)\subseteq K_z\).
For the reverse inclusion, each active branch satisfies
\(f-f_\alpha\geq0\) with equality at \(z\), so its differential is
bounded above in every direction by the ordinary directional derivative
of \(f\), and hence by the Clarke generalized directional derivative.
The defining support inequalities for the Clarke subdifferential imply
\(\diff(f_\alpha)_z\in\partial^{\mathrm C}f(z)\).  Convexity gives
equality.  This is the finite smooth maximum rule
\cite[Proposition~2.3.12]{Clarke1990}.

Put \(\mu=\operatorname{dist}(0,K_z)\).  Convex duality gives
\[
 \mu=\max_{\lVert v\rVert\leq1}
           \left\{-\max_{\xi\in K_z}\xi[v]\right\}.
\]
To see the identity directly, let \(\xi_0\in K_z\) have least norm.
For every \(v\) with \(\lVert v\rVert\leq1\),
\[
 -\max_{\xi\in K_z}\xi[v]\leq-\xi_0[v]\leq\lVert\xi_0\rVert.
\]
If \(\xi_0=0\), the choice \(v=0\) gives equality.  Otherwise use the
vector dual to \(-\xi_0/\lVert\xi_0\rVert\) and the nearest-point
inequality \(\langle\xi-\xi_0,\xi_0\rangle\geq0\) for \(\xi\in K_z\).

\Cref{thm:weak-implies-clarke} supplies the lower bound:
\[
 |\diff f|(z)\geq\operatorname{dist}
 (0,\partial^{\mathrm C}f(z))=\mu.
\]
For the upper bound, let a weak-slope deformation \(H\) witness an
admissible rate \(\sigma>0\).  Choose \(u_k\downarrow0\) and put
\[
 v_k:=u_k^{-1}\exp_z^{-1}H(z,u_k).
\]
The displacement bound gives \(\lVert v_k\rVert\leq1\) for all large
\(k\).  Passing to a subsequence, let \(v_k\to v\).  For every
\(\alpha\in I_f(z)\),
\[
 f_\alpha(H(z,u_k))\leq f(H(z,u_k))\leq f(z)-\sigma u_k.
\]
Taylor expansion, division by \(u_k\), and passage to the limit give
\(\diff(f_\alpha)_z[v]\leq-\sigma\).  The duality identity implies
\(\sigma\leq\mu\).  Taking the supremum over admissible rates proves
\eqref{eq:weak-slope-finite-maximum}; separation of the compact convex
set \(K_z\) from zero gives the stated equivalences.
\end{proof}

\subsection{A smooth metric with uncountably many critical diameter values}
\label{app:cantor-clarke-spectrum}

We construct the metric in \cref{prop:cantor-clarke-spectrum}
using a warping function that is flat on a Cantor set and takes distinct
values there.  The corresponding parallels are geodesics.  We prove
that their semicircles minimize globally, so the two orientations give
opposite limiting distance differentials and make the endpoint pairs
Clarke critical.

\begin{proof}
\noindent\emph{Constructing the warping function.}
Use coordinates \([0,10]\times S^1\), with the two boundary circles
collapsed, and a metric
\[
 (\diff s)^2+f(s)^2(\diff\theta)^2.
\]
We construct \(f\) on \([2,8]\) so that
\begin{equation}
 1\leq f\leq1.1,\qquad\lVert f''\rVert_\infty\leq0.1,
 \label{eq:cantor-warp-bounds}
\end{equation}
while \(f'=0\) on a Cantor set \(C\subset[4,5]\) and \(f|_C\) is
injective.  Take the middle-thirds Cantor set in \([4,5]\), put
\[
 u_0(s):=\exp\!\left(-\frac{1}{(s-a)(b-s)}\right)
 \quad\text{on each complementary gap }(a,b),
\]
and set \(u_0=0\) on \(C\) and outside \([4,5]\).   On each complementary gap \((a,b)\), \(u_0\) is flat at both endpoints. For each fixed \(j\geq0\), \(\sup_{s\in(a,b)}|u_0^{(j)}(s)|\) tends to zero faster than any power of \(b-a\) as \(b-a\to0\).  These estimates show that
\(u_0\) extends smoothly by zero, with all derivatives zero on \(C\).
Let \(A=\int_4^5u_0(s)\,\diff s>0\), choose a nonnegative
\(\beta\in C_c^\infty((6,7))\) with integral one, and set
\[
 H(s):=\int_2^s\bigl(u_0(t)-A\beta(t)\bigr)\,\diff t,
 \qquad f(s):=1+\varepsilon H(s).
\]
Here \(0\leq H\leq A\); choosing \(\varepsilon>0\) sufficiently
small gives \eqref{eq:cantor-warp-bounds}.  Between any two distinct
points of \(C\) there is a complementary gap, so \(f|_C\) is strictly
increasing.  The function equals one near both ends of \([2,8]\).
Extend it by positive smooth caps, with \(f(s)=s\) near zero and
\(f(s)=10-s\) near ten.  These formulas make the metric smooth at the
poles.  Its Gaussian curvature vanishes on an open cylindrical band
and does not vanish everywhere; thus the metric is not analytic.

\par\smallskip\noindent\emph{Global minimality of the semicircles.}
For \(c\in C\), put
\[
 x_c=(c,0),\qquad y_c=(c,\pi),\qquad r_c=\pi f(c).
\]
Since \(f'(c)=0\), the parallel \(s=c\) is a geodesic.  We claim
that both of its semicircles from \(x_c\) to \(y_c\) minimize globally.
For a piecewise smooth competitor, let
\(R=\max_t|s(t)-c|\).  Its radial variation is at least \(2R\).
If \(R\geq r_c/2\), its length is therefore at least \(r_c\).
Otherwise \(R<1.1\pi/2<2\), so the path remains in \([2,8]\times S^1\)
and its angular variation is at least \(\pi\).  Taylor's theorem gives
\(f(s(t))\geq f(c)-0.05R^2>0\).  The triangle inequality for the
integral of the nonnegative vector
\((|\dot s|,f(s)|\dot\theta|)\) now gives
\begin{align*}
 L^2
 &\geq 4R^2+\pi^2\bigl(f(c)-0.05R^2\bigr)^2\\
 &=\pi^2f(c)^2+
       \bigl(4-0.1\pi^2f(c)\bigr)R^2+0.0025\pi^2R^4
 \geq r_c^2,
\end{align*}
because \(4-0.1\pi^2f(c)>2.9\).  This proves the claim.

\par\smallskip\noindent\emph{Criticality and additional labels.}
Every proper initial subarc of either minimizing semicircle ends before
its cut time from \(x_c\).  The distance is smooth at those ordered
pairs.  First variation shows that, as the second endpoint tends to
\(y_c\) along the two semicircles,
the differentials of \(d_M\) on \(M\times M\)
converge to opposite covectors in
\(T_{x_c}^*M\times T_{y_c}^*M\).  Both are reachable differentials of
\(d_M\); their midpoint is zero.  Hence
\[
 0\in\partial^{\mathrm C}d_M(x_c,y_c),
 \qquad r_c\in\Sigma_2^{\specC}(M).
\]
Continuity and injectivity of \(f|_C\) show that
\(\{r_c:c\in C\}\) is a Cantor set.

For \(N>2\), add \(N-2\) copies of the {additional point}
\((c,\pi/2)\).  Their distances to either endpoint are at most
\(r_c/2\), so \((x_c,y_c)\) is still the unique labelled diameter pair.
Locally the diameter equals that distance function, whose two opposite
reachable covectors extend by zero on the {additional coordinates}.  The
tuple is therefore Clarke critical.  For every \(N\geq2\), moving \(y_c\) slightly toward
\(x_c\) along one semicircle strictly reduces the active distance.
When additional labels are present, sufficiently small movements keep all other pairwise distances strictly below the distance between these two endpoints.
Thus the diameter decreases, and none of these tuples is a local
minimum.
\end{proof}

\subsection{Weak-slope stationarity at an antipodal pair}
\label{subsec:antipodal-weak-slope}

We complete the proof of \cref{lem:two-label-spherical-spectrum} by
treating the antipodal distance \(\pi\), where the smooth first-order
criterion does not apply.  Near an antipodal pair, suitable coordinates
write the distance as \(\pi-\|w\|\), with \(w\in\R^m\) measuring the
displacement from antipodality.  The limiting differentials of this
local expression give Clarke criticality.  For weak-slope stationarity,
we show that a deformation decreasing distance at a positive uniform
rate would, after normalization, extend a map homotopic to the identity
on \(\Sph^{m-1}\) across a ball.  The no-retraction argument rules out
such an extension.

\begin{proof}[Proof of \cref{lem:two-label-spherical-spectrum}]
At a pair of distance in \((0,\pi)\), moving one endpoint toward the
other strictly decreases the smooth distance function.  Thus
\cref{prop:spherical-spectrum-coincidence} excludes both kinds of
stationarity.  Constant pairs are global minima of diameter and account
for zero in both spectra.

Fix an antipodal pair \((x,-x)\).  We reduce distance near this
pair to the function \(\pi-\|w\|\): its limiting differentials give
Clarke criticality, while the no-retraction argument below excludes a
weak-slope deformation with positive descent rate.
For a nearby pair \((y,y')\), choose normal coordinates
\(u\in\R^m\) for \(y\) near \(x\), and write
\(-y'=\exp_y(w)\).  A smooth orthonormal frame near \(x\) identifies
\(w\in T_y\Sph^m\) with a vector of \(\R^m\).  In these local
coordinates,
\[
 d_m(y,y')=\pi-\|w\|.
\]
At points with \(w\ne0\), the differential is
\((0,-w/\|w\|)\).  Its limiting covectors as \(w\to0\) include
\((0,v)\) for every unit vector \(v\).  Their convex hull contains
zero, so the antipodal pair is Clarke critical.

The coordinate metric and the maximum product metric are locally
bi-Lipschitz.  By \cref{lem:weak-slope-bilipschitz}, it suffices to prove
weak-slope stationarity of \(\pi-\|w\|\) in the Euclidean product
metric.  Suppose a weak-slope deformation \(H\) has positive
rate \(\sigma\), as in \cref{def:weak-slope}.  Choose a closed ball
\(\overline B_r^m\) in the slice \(u=0\), small enough to lie in the
spatial domain of \(H\), and a time \(0<\tau<r\) in its time domain,
small enough that all images remain in the coordinate neighborhood.
For \(w\in\overline B_r^m\), let \(W_\tau(w)\in\R^m\) be the
\(w\)-coordinate of \(H((0,w),\tau)\).  The descent and displacement
bounds give
\[
 \|W_\tau(w)\|\geq\|w\|+\sigma\tau>0,
 \qquad
 \|W_\tau(w)-w\|\leq\tau.
\]
Consequently, \(W_\tau/\|W_\tau\|\) maps the whole ball into
\(\Sph^{m-1}\).  On its boundary this map is homotopic to
\(w\mapsto w/r\): the straight-line homotopy from \(w\) to
\(W_\tau(w)\) avoids zero because \(\tau<r\).
For \(m\geq2\), a boundary map of degree one cannot extend over the
ball, by the no-retraction argument
\cite[Corollary~2.15]{Hatcher2002}.  For \(m=1\), a
continuous map from the interval into \(\Sph^0\) cannot take the two
different endpoint values.  Thus no such deformation exists.
\end{proof}

\section{Equivalence of the ellipse and stack-walk coordinates}
\label{app:ellipse-stack-walk-coordinate-bridge}

The spherical walk in
\eqref{eq:general-stack-walk}--\eqref{eq:general-stack-symmetry} and
the planar ellipse construction give the same parallels for the stack
over a base \(P\) in \eqref{eq:general-spherical-stack}.  Axial
rescaling sends the ellipse to a unit meridian circle.  Under this map,
two adjacent walk equations are equivalent to tangency of the chord
joining the corresponding even-indexed points.  Fix
\[
 \alpha\in\left(\frac{\pi}{2},\pi\right),
 \qquad
 c:=\cos\alpha,
 \qquad
 \beta:=\pi-\alpha,
 \qquad
 \varrho:=-c=\cos\beta\in(0,1),
\]
and put
\[
 \mathcal E_\beta
 :=\left\{(x,y)\in\R^2:\frac{x^2}{\varrho^2}+y^2=1\right\},
 \qquad
 \Gamma_t:=\left\{(x,y)\in\R^2:x^2+y^2=t^2\right\}.
\]
\begin{proposition}[Equivalence of the ellipse and stack-walk coordinates]
\label{prop:ellipse-stack-walk-equivalence}
Let \(\varphi_0,\ldots,\varphi_{2k+1}\) be a symmetric \(k\)-layer
stack walk and set
\[
 t:=\cos\varphi_1,
 \qquad
 \theta_i:=\frac{\pi}{2}-\varphi_{2i},
 \qquad
 q_i:=(\varrho\cos\theta_i,\sin\theta_i)
 \quad(0\leq i\leq k).
\]
\emph{From the walk to the ellipse.} Then \(0<t<\varrho\).  The points \(q_0,\ldots,q_k\) form the right-hand chain of a convex \((2k+1)\)-gon inscribed in
\(\mathcal E_\beta\) and circumscribed about \(\Gamma_t\): the edges \(q_iq_{i+1}\), \(0\leq i<k\), are tangent to
\(\Gamma_t\).  The bottom edge joins \(q_k\) to its reflection in the
vertical axis; it lies on \(y=-t\) and is tangent to \(\Gamma_t\) at
\((0,-t)\).  Reflection completes the chain to the full odd polygon.
Under the axial rescaling
\[
 (x,y)\longmapsto\left(\frac{x}{\varrho},y\right),
\]
the image of \(q_i\) lies on the unit meridian circle at
colatitude \(\varphi_{2i}\), measured from \((0,1)\).  Hence the
selected parallels in the ellipse construction have exactly the layer
colatitudes
\[
 \ell_i=\varphi_{2i}\qquad(1\leq i\leq k).
\]

\emph{From the ellipse to the walk.}  Conversely, let \(Q\) be a convex \((2k+1)\)-gon inscribed in
\(\mathcal E_\beta\) and circumscribed about \(\Gamma_t\), where
\(0<t<\varrho\).  Assume that \((0,1)\) is a vertex, and enumerate the
vertices clockwise from \((0,1)\) in a single traversal of the boundary.
Then the same change of variables determines a unique symmetric
\(k\)-layer stack walk.  Thus the ellipse and symmetric stack-walk
constructions give precisely the same layer positions.
\end{proposition}

\begin{figure}[htbp]
\centering
\begingroup
\newcommand{\figmath}[1]{\ensuremath{#1}}
\begin{tikzpicture}[
  >={Latex[length=1.8mm,width=1.25mm]},
  every node/.style={font=\scriptsize,text=black,inner sep=1.5pt},
  geometry/.style={draw=black!70,line width=.55pt},
  point/.style={circle,fill=black!85,inner sep=1.2pt},
  layerpoint/.style={circle,fill=black!85,inner sep=1.05pt},
  guide/.style={draw=black!35,line width=.35pt,dash pattern=on 2pt off 2pt},
  backarc/.style={draw=black!45,line width=.45pt,dash pattern=on 1.5pt off 1.5pt},
  leader/.style={draw=black!60,line width=.4pt},
  panel/.style={font=\footnotesize\bfseries,anchor=base,align=center},
  stepnote/.style={font=\scriptsize,align=center,inner sep=2pt},
  line cap=round,line join=round
]
\node[panel] at (0,1.95) {(a) Tangential polygon};
\node[panel] at (4.70,1.95) {(b) Selected colatitudes};
\node[panel] at (9.40,1.95) {(c) Two-layer stack};

\begin{scope}[scale=1.30]
  \draw[geometry] (0,0) ellipse[x radius=.75,y radius=1];
  \draw[geometry,draw=black!55,line width=.45pt]
    (0,0) circle[radius=.67772];
  \coordinate (q0) at (0,1);
  \coordinate (q1) at (.73435,.20323);
  \coordinate (q2) at (.55149,-.67772);
  \coordinate (q3) at (-.55149,-.67772);
  \coordinate (q4) at (-.73435,.20323);
  \draw[draw=black!85,line width=.95pt]
    (q0)--(q1)--(q2)--(q3)--(q4)--cycle;
  \foreach \q in {q0,q1,q2,q3,q4}{\node[point] at (\q) {};}
  \node[anchor=south west] at (.10,1.05) {\figmath{q_0}};
  \node[anchor=west] at (.88,.25) {\figmath{q_1}};
  \node[anchor=west] at (.76,-.75) {\figmath{q_2}};
  \node[anchor=east] at (-.93,.67) {\figmath{\mathcal E_\beta}};
  \draw[leader] (-.87,.67)--(-.56,.665);
  \node at (0,.02) {\figmath{\Gamma_t}};
\end{scope}

\draw[geometry,->] (1.65,.02)--(3.15,.02)
  node[midway,above=5pt,stepnote]{axial\\rescaling};

\begin{scope}[shift={(4.70,0)},scale=1.30]
  \draw[guide] (0,-1)--(0,1);
  \draw[guide] (-.97913,.20323)--(.97913,.20323);
  \draw[guide] (-.73532,-.67772)--(.73532,-.67772);
  \draw[geometry,draw=black!50,line width=.4pt]
    (.97913,.20323)--(0,0)--(.73532,-.67772);
  \draw[geometry] (0,0) circle[radius=1];
  \node[point] at (0,1) {};
  \node[point] at (.97913,.20323) {};
  \node[point] at (.73532,-.67772) {};
  \node[anchor=south west] at (.10,1.05) {\figmath{\varphi_0=0}};
  \node[anchor=south] at (-.43,.27) {\figmath{\ell_1}};
  \node[anchor=south] at (-.36,-.61) {\figmath{\ell_2}};
\end{scope}

\draw[geometry,->] (6.25,.02)--(7.85,.02)
  node[midway,above=5pt,stepnote]
    {place \figmath{P} on\\each parallel};

\begin{scope}[shift={(9.40,0)},scale=1.30]
  \draw[guide] (0,-1)--(0,1);
  \draw[geometry] (0,0) circle[radius=1];
  \foreach \s/\c in {.97913/.20323,.73532/-.67772}{
    \pgfmathsetmacro{\cy}{\c*cos(10)}
    \pgfmathsetmacro{\ry}{\s*sin(10)}
    \draw[backarc] (\s,\cy)
      arc[start angle=0,end angle=180,x radius=\s,y radius=\ry];
    \draw[geometry] (-\s,\cy)
      arc[start angle=180,end angle=360,x radius=\s,y radius=\ry];
    \foreach \a in {18,106,211,302}{
      \node[layerpoint] at ({\s*cos(\a)},{\cy+\ry*sin(\a)}) {};
    }
  }
  \node[point] at (0,{cos(10)}) {};
  \node[anchor=south west] at (.10,1.05) {\figmath{Z}};
  \draw[leader] (.97913,{.20323*cos(10)})--(1.10,.40)--(1.18,.40);
  \node[anchor=west] at (1.20,.40) {\figmath{L_1}};
  \draw[leader] (.73532,{-.67772*cos(10)})--(1.10,-.55)--(1.18,-.55);
  \node[anchor=west] at (1.20,-.55) {\figmath{L_2}};
\end{scope}

\node[anchor=north,inner sep=0pt] at (4.70,-1.72) {\figmath{
  \ell_i=\varphi_{2i},\qquad
  (s_i,c_i)=(\sin\ell_i,\cos\ell_i),\qquad
  L_i=s_iP\times\{c_i\}\quad(i=1,2)
}};
\end{tikzpicture}
\endgroup
\caption{The ellipse construction and the corresponding two-layer stack.
In (a), a symmetric pentagon with vertices on \(\mathcal E_\beta\) is
circumscribed about \(\Gamma_t\).  In (b), axial rescaling sends the
vertices corresponding to the selected even-indexed walk points to
\((s_i,c_i)=(\sin\ell_i,\cos\ell_i)\), where
\(\ell_i=\varphi_{2i}\), for \(i=1,2\).  The intervening odd-indexed walk
points are not shown.  In (c), each selected parallel carries
\(L_i=s_iP\times\{c_i\}\), and \(Z=(0,1)\) is adjoined.  The repeated
marker pattern, including the displayed base directions, is schematic; it
records only that the same labelled base \(P\) is placed on both layers.}
\label{fig:lovasz-coordinate-bridge}
\end{figure}
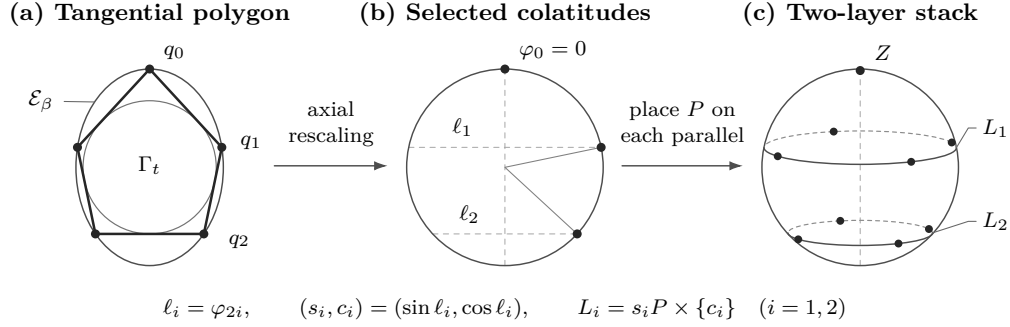

\begin{proof}[Proof of \cref{prop:ellipse-stack-walk-equivalence}]
\noindent\emph{Forward direction: tangency and the selected parallels.}
The symmetry and strict ordering of the walk give
\[
 \varphi_k<\frac{\pi}{2}<\varphi_{k+1},
 \qquad
 \varphi_k+\varphi_{k+1}=\pi.
\]
In particular, \(t>0\).  Applying
\eqref{eq:general-stack-walk} to this middle pair gives
\[
 t
 =\varrho-(1+\varrho)\cos^2\varphi_k
 <\varrho.
\]

Fix \(0\leq i<k\) and put
\[
 u:=\varphi_{2i},\qquad v:=\varphi_{2i+1},\qquad
 w:=\varphi_{2i+2},
 \qquad
 \mu:=\frac{u+w}{2},\qquad a:=\frac{w-u}{2}.
\]
Since \(\varrho=-c\), the two adjacent walk equations are
\[
 t=\cos u\cos v+\varrho\sin u\sin v,
 \qquad
 t=\cos v\cos w+\varrho\sin v\sin w.
\]
Subtracting these equations gives
\(\sin\mu\cos v=\varrho\cos\mu\sin v\).
Since \(0<\mu,v<\pi\), the vector \((\cos v,\sin v)\) is a
positive multiple of \((\varrho\cos\mu,\sin\mu)\).  Normalizing gives
\[
 \cos v
 =\frac{\varrho\cos\mu}
 {\sqrt{\varrho^2\cos^2\mu+\sin^2\mu}},
 \qquad
 \sin v
 =\frac{\sin\mu}
 {\sqrt{\varrho^2\cos^2\mu+\sin^2\mu}}.
\]
Substitution into either adjacent equation yields
\begin{equation}
 t
 =\frac{\varrho\cos a}
 {\sqrt{\varrho^2\cos^2\mu+\sin^2\mu}}.
 \label{eq:digon-ellipse-tangency}
\end{equation}
On the other hand, the distance from the origin to the
chord's supporting line is
\begin{align*}
 \operatorname{dist}\bigl(0,\aff\{q_i,q_{i+1}\}\bigr)
 =\frac{|\det(q_i,q_{i+1})|}{\|q_{i+1}-q_i\|}
 =\frac{\varrho\cos a}
 {\sqrt{\varrho^2\cos^2\mu+\sin^2\mu}}.
\end{align*}
Thus \eqref{eq:digon-ellipse-tangency} says exactly that every
{edge in the right-hand chain} is tangent to \(\Gamma_t\).  Walk symmetry also gives
\[
 \varphi_{2k}=\pi-\varphi_1,
 \qquad
 \sin\theta_k=-\cos\varphi_1=-t,
\]
so the horizontal edge from \(q_k\) to its vertical-axis reflection is
tangent along \(y=-t\).  Reflecting the right-hand chain completes the
convex odd polygon.  More explicitly, if \(R(x,y):=(-x,y)\), put
\[
 q_{2k+1-i}:=R(q_i)\quad(1\leq i\leq k),
 \qquad q_{2k+1}:=q_0.
\]
The completed vertices occur once in cyclic order on the ellipse, and every
consecutive parameter gap is less than \(\pi\).  Hence, for each tangent line, the origin lies in the open half-plane containing the interior of the polygon.  The closed
disk bounded by \(\Gamma_t\) is therefore contained in the polygon, so the
polygon is indeed circumscribed about \(\Gamma_t\).  Finally,
\[
 \left(\frac{(q_i)_1}{\varrho},(q_i)_2\right)
 =(\cos\theta_i,\sin\theta_i)
 =(\sin\varphi_{2i},\cos\varphi_{2i}),
\]
which proves the claim about its selected parallels.

\par\smallskip\noindent\emph{Converse: reconstructing the walk.}
We first recover the polygon's reflection symmetry from its
tangent edges.  Orient the polygon clockwise and enumerate it from
\(q_0=(0,1)\).  At each point of \(\mathcal E_\beta\), there are
exactly two tangents to the strictly interior circle \(\Gamma_t\).
Choose the tangent chord directed so that the circle lies on its right;
its other endpoint defines a map
\(F:\mathcal E_\beta\to\mathcal E_\beta\).
Convexity, cyclic ordering, and containment of \(\Gamma_t\) imply
that each polygonal edge runs from a vertex \(q_i\) to \(F(q_i)\).
Choosing the other tangent gives \(F^{-1}\).  Reflection in the
vertical axis interchanges the two tangent choices, so it conjugates
\(F\) to \(F^{-1}\) and fixes \(q_0\).  Since \(q_i=F^i(q_0)\) and
\(F^{2k+1}(q_0)=q_0\),
\[
 R(q_i)=RF^i(q_0)=F^{-i}(q_0)=F^{2k+1-i}(q_0)=q_{2k+1-i},
\]
where \(R(x,y):=(-x,y)\).  Thus
\[
 q_{2k+1-i}=(- (q_i)_1,(q_i)_2)
 \qquad(0\leq i\leq2k+1).
\]
In particular, the middle edge is horizontal and, being the lower supporting
edge, is tangent along \(y=-t\).

Write the right-hand vertices uniquely as
\[
 q_i=(\varrho\cos\theta_i,\sin\theta_i),
 \qquad
 \frac{\pi}{2}=\theta_0>\theta_1>\cdots>
 \theta_k>-\frac{\pi}{2},
\]
and set \(\varphi_{2i}:=\tfrac{\pi}{2}-\theta_i\).  For
\(0\leq i<k\), put
\[
 \mu_i:=\frac{\varphi_{2i}+\varphi_{2i+2}}{2},
 \qquad
 a_i:=\frac{\varphi_{2i+2}-\varphi_{2i}}{2},
\]
and define the intervening angle by
\[
 (\cos\varphi_{2i+1},\sin\varphi_{2i+1})
 :=\frac{(\varrho\cos\mu_i,\sin\mu_i)}
 {\sqrt{\varrho^2\cos^2\mu_i+\sin^2\mu_i}}.
\]
The chord-distance calculation and tangency to \(\Gamma_t\) give
\eqref{eq:digon-ellipse-tangency}, with \(\mu=\mu_i\) and \(a=a_i\).
Direct substitution gives both adjacent equations in
\eqref{eq:general-stack-walk} with common value \(t\).
It remains to show that the reconstructed odd-indexed angle lies
strictly between its two neighboring even-indexed angles.
With \(\varepsilon_i:=\varphi_{2i+1}-\mu_i\), this is the
inequality \(|\varepsilon_i|<a_i\).  We have
\[
 \sin^2\varepsilon_i
 =\frac{(1-\varrho)^2\sin^2\mu_i\cos^2\mu_i}
 {\varrho^2\cos^2\mu_i+\sin^2\mu_i}
 \leq(1-\varrho^2)\cos^2\mu_i.
\]
Moreover,
\[
 \cos\varepsilon_i
 =\frac{\varrho\cos^2\mu_i+\sin^2\mu_i}
 {\sqrt{\varrho^2\cos^2\mu_i+\sin^2\mu_i}}>0,
\]
so \(|\varepsilon_i|<\tfrac{\pi}{2}\).
Since \(t<\varrho\), \eqref{eq:digon-ellipse-tangency} gives
\[
 \sin^2a_i>(1-\varrho^2)\cos^2\mu_i.
\]
Here \(0<a_i<\tfrac{\pi}{2}\) and
\(|\varepsilon_i|<\tfrac{\pi}{2}\), so \(|\varepsilon_i|<a_i\), which is equivalent to
\[
 \varphi_{2i}<\varphi_{2i+1}<\varphi_{2i+2}.
\]
For \(i=0\), the first walk equation gives
\(t=\cos\varphi_1\).  The lower horizontal edge gives
\(\varphi_{2k}=\pi-\varphi_1\); set
\(\varphi_{2k+1}:=\pi\).  The last walk equation then holds as well.

\noindent\emph{Symmetry and uniqueness.}
It remains only to verify symmetry.  {The sequence}
\[
 \psi_r:=\pi-\varphi_{2k+1-r}
\]
is strictly increasing, satisfies the same walk recurrence, and begins with
\(\psi_0=0\) and \(\psi_1=\varphi_1\).  Starting with this common pair,
suppose the two sequences agree through index \(r\geq1\).
{For fixed \(\varphi_r\)}, the left-hand side of
\eqref{eq:general-stack-walk} is a phase-shifted cosine in the next
angle, with amplitude at least \(\varrho>t\).  The recurrence
therefore has exactly two roots modulo \(2\pi\).
The preceding {colatitude} \(\varphi_{r-1}\) is one root; strict ordering
selects the other root uniquely as the following {colatitude} in \([0,\pi]\).  Induction therefore gives
\(\psi_r=\varphi_r\) for every \(r\), which is precisely
\eqref{eq:general-stack-symmetry}.  The same argument shows uniqueness of
the reconstructed walk for the given tangential polygon.
\end{proof}

\section{Shooting and convergence details for canonical stack walks}
\label{app:stack-walk-shooting-convergence}

The canonical walk is determined by its first step.  A shooting
argument selects the step for which the two middle colatitudes sum
to \(\pi\), so reflection completes the walk.  We then use  the maps that determine successive colatitudes to prove monotonicity in the layer number and continuity in
the base diameter of the transform \(\Lambda_k\) in
\cref{prop:canonical-lovasz-transform}.
The final two subsections give supplementary convergence estimates.
The derived-set argument uses the estimate
\eqref{eq:general-stack-limit}, which is uniform in the base diameter.

\subsection{Shooting and closing}

Fix a base diameter \(\alpha\in(\pi/2,\pi)\) and a layer number
\(k\geq1\).  To prove \cref{prop:canonical-symmetric-stack-walk}, we
vary the first step and solve \eqref{eq:general-stack-walk} forward
from the north pole.  We identify an interval of trial first steps
on which the iteration reaches the two middle colatitudes.  Their sum
is continuous and strictly increasing on this interval, with endpoint
values on opposite sides of \(\pi\).  The unique first step for which
this sum equals \(\pi\) allows us to complete the walk by reflection
across the equator.

\begin{proof}[Proof of \cref{prop:canonical-symmetric-stack-walk}]

Put
\[
 \beta:=\pi-\alpha,
 \qquad
 a:=\cos\beta=-\cos\alpha\in(0,1),
\]
and define
\[
 H(u,v):=\cos u\cos v+a\sin u\sin v.
\]
\noindent\emph{Admissible first steps.}
We first determine the possible first step of a symmetric walk.  If
\(\tau:=\varphi_1\), then symmetry gives
\[
 \varphi_{k+1}=\pi-\varphi_k,
 \qquad
 0<\tau\leq\varphi_k<\frac{\pi}{2}.
\]
The middle recurrence therefore gives
\begin{equation}
 \cos\tau
 =H(\varphi_k,\pi-\varphi_k)
 =a-(1+a)\cos^2\varphi_k
 <a=\cos\beta.
 \label{eq:lovasz-middle-equation}
\end{equation}
Thus every symmetric \(k\)-layer stack walk has
\[
 \tau\in\left(\beta,\frac{\pi}{2}\right).
\]

\par\smallskip\noindent\emph{{The map \(F_\tau\) and its monotonicity.}}
For a trial first step
\(\tau\in(\beta,\tfrac{\pi}{2})\), we solve the walk recurrence
forward from the north pole.  Given a current colatitude
\(0\leq u\leq\tfrac{\pi}{2}\), let
\(F_\tau(u)\in(u,\pi)\) be the unique next colatitude \(v\)
satisfying
\begin{equation}
 H(u,v)=\cos\tau.
 \label{eq:lovasz-shooting-step}
\end{equation}
Indeed,
\[
 H(u,u)\geq a>\cos\tau,
 \qquad
 H(u,\pi)=-\cos u<\cos\tau,
\]
and, for \(u<v<\pi\),
\[
 H_v(u,v)=-\cos u\sin v+a\sin u\cos v<0.
\]
For \(v\geq\tfrac{\pi}{2}\) this is immediate; for
\(u<v<\tfrac{\pi}{2}\), it follows from
\(\tan v>\tan u\geq a\tan u\).

To compare the next colatitude as either \(u\) or \(\tau\)
increases, we also need the sign \(H_u(u,v)>0\) at a solution of
\eqref{eq:lovasz-shooting-step}.  This sign is immediate when
\(v\geq\tfrac{\pi}{2}\).  If \(v<\tfrac{\pi}{2}\), set
\[
 v_0:=\arctan\frac{\tan u}{a}.
\]
Then
\[
 H(u,v_0)
 =a\sqrt{\frac{1+\tan^2u}{a^2+\tan^2u}}
 >a>\cos\tau=H(u,v).
\]
Since \(H(u,\mathord\cdot)\) is strictly decreasing, \(v>v_0\), which is
equivalent to \(H_u(u,v)>0\).  Implicit differentiation gives
\begin{equation}
 \frac{\partial F_\tau}{\partial u}
 =-\frac{H_u}{H_v}>0,
 \qquad
 \frac{\partial F_\tau}{\partial\tau}
 =-\frac{\sin\tau}{H_v}>0.
 \label{eq:lovasz-step-monotonicity}
\end{equation}

Put
\[
 u_0(\tau):=0,
 \qquad
 u_{r+1}(\tau):=F_\tau(u_r(\tau))
\]
whenever the input \(u_r(\tau)\) is at most \(\tfrac{\pi}{2}\).
Each iterate is continuous in \(\tau\) on its domain, and is strictly increasing in \(\tau\) when its index is positive. For fixed \(\tau\), the iterates are strictly increasing with the index wherever they are defined.

\par\smallskip\noindent\emph{Endpoint parameters.}
The construction extends continuously to \(\tau=\beta\).  If
\(u<\tfrac{\pi}{2}\), then
\[
 H(u,u)>a,
 \qquad
 H\left(u,\frac{\pi}{2}\right)=a\sin u<a,
\]
so
\[
 u<F_\beta(u)<\frac{\pi}{2}.
\]
Consequently every finite iterate \(u_r(\beta)\) lies below
\(\tfrac{\pi}{2}\).  At the other endpoint we use the continuous extension
\[
 F_{\pi/2}\left(\frac{\pi}{2}\right)=\pi.
\]

\par\smallskip\noindent\emph{When each iterate reaches the equator.}
The iterates are defined only while their inputs to \(F_\tau\)
are at most \(\pi/2\).  We therefore locate, by induction, the
parameter at which each iterate first reaches the equator.  These
parameters will give an interval on which the closing condition is
defined.  Set
\[
 \sigma_1:=\frac{\pi}{2}.
\]
Suppose \(j\geq2\) and \(\sigma_{j-1}\) has been defined so that
\[
 u_{j-1}(\sigma_{j-1})=\frac{\pi}{2}.
\]
The function \(u_j\) is defined, continuous, and strictly increasing on
\([\beta,\sigma_{j-1}]\), and
\[
 u_j(\beta)<\frac{\pi}{2},
 \qquad
 u_j(\sigma_{j-1})
 =F_{\sigma_{j-1}}\left(\frac{\pi}{2}\right)
 >\frac{\pi}{2}.
\]
Hence there is a unique
\[
 \sigma_j\in(\beta,\sigma_{j-1})
 \qquad\text{such that}\qquad
 u_j(\sigma_j)=\frac{\pi}{2}.
\]
In particular, \(\sigma_k\) is defined, and
\[
 u_0(\sigma_k),\ldots,u_{k-1}(\sigma_k)
 <\frac{\pi}{2}=u_k(\sigma_k).
\]

\par\smallskip\noindent\emph{Closing and reflecting the walk.}
The first half of the walk joins its reflection exactly when
its two middle colatitudes sum to \(\pi\).  On
\([\beta,\sigma_k]\), the closing function
\[
 \Psi_k(\tau):=u_k(\tau)+u_{k+1}(\tau)
\]
is continuous and strictly increasing.  At its endpoints,
\[
 \Psi_k(\beta)<\pi,
 \qquad
 \Psi_k(\sigma_k)
 =\frac{\pi}{2}
  +F_{\sigma_k}\left(\frac{\pi}{2}\right)
 >\pi.
\]
There is therefore a unique
\(\tau_*\in(\beta,\sigma_k)\) such that
\[
 u_k(\tau_*)+u_{k+1}(\tau_*)=\pi.
\]

Define
\[
 \varphi_r:=u_r(\tau_*)
 \qquad(0\leq r\leq k+1)
\]
and complete the sequence by reflection:
\[
 \varphi_{2k+1-r}:=\pi-\varphi_r
 \qquad(0\leq r\leq k).
\]
The closing equation makes the two prescriptions agree at the middle and
gives
\[
 \varphi_k<\frac{\pi}{2}<\varphi_{k+1}.
\]
Thus the completed sequence is strictly increasing.  Its first half
satisfies the recurrence by construction, and its reflected half does also,
because
\[
 H(\pi-v,\pi-u)=H(u,v).
\]
It is therefore a symmetric \(k\)-layer stack walk.

\par\smallskip\noindent\emph{Uniqueness.}
Conversely, the first step \(\tau\) of any symmetric \(k\)-layer stack walk
lies in \((\beta,\tfrac{\pi}{2})\).  Strict ordering forces its first half to
be the successive increasing roots
\[
 \varphi_r=u_r(\tau)
 \qquad(0\leq r\leq k+1),
\]
and \(\varphi_k<\pi/2\) implies \(\tau<\sigma_k\), so
\(\Psi_k\) is defined at this parameter.  Middle symmetry is exactly
\[
 \Psi_k(\tau)=\pi.
\]
The strict monotonicity of \(\Psi_k\) forces \(\tau=\tau_*\), proving
uniqueness.
\end{proof}

\subsection{Monotonicity and continuity of the stack diameters}
\label{subsec:stack-transform-proof}

We now prove the monotonicity and continuity assertions of
\cref{prop:canonical-lovasz-transform}.  The stack diameter is
\(\Lambda_k(\alpha)=\pi-\tau_k\), where \(\tau_k\) is the first step
of the canonical \(k\)-layer walk.  At fixed base diameter, comparing
the iterates of  the maps \(F_\tau\) from
\eqref{eq:lovasz-shooting-step} shows that adding a layer decreases
the first step and hence increases the diameter.  For continuity at
fixed \(k\), we pass to limits of the colatitude sequences and use
uniqueness to identify the limiting walk when \(\alpha<\pi\).
At \(\alpha=\pi\), the limiting recurrence forces equal increments,
giving the extension \(\Lambda_k(\pi)=\pi-\pi/(2k+1)\).

\begin{proof}[Proof of \cref{prop:canonical-lovasz-transform}]
\noindent\emph{Monotonicity in the layer number.}
Put \(\beta:=\pi-\alpha\) and \(a:=\cos\beta=-\cos\alpha\).
The uniform estimate is \eqref{eq:general-stack-limit}.
Use {the maps} \(F_\tau\) defined by
\eqref{eq:lovasz-shooting-step}.  Their strict
monotonicity in the current colatitude and in \(\tau\) is
\eqref{eq:lovasz-step-monotonicity}.
Write \(\varphi_j^{(k)}\) for the \(j\)-th colatitude at layer count \(k\),
and put \(\tau_k:=\varphi_1^{(k)}\).  If \(\tau_{k+1}\geq\tau_k\), apply
\eqref{eq:lovasz-step-monotonicity} successively through index
\(k+1\).  In both walks, the input colatitudes through index \(k\)
lie below \(\pi/2\), so  the maps \(F_\tau\) are defined and give
\[
 \varphi_{k+1}^{(k+1)}\geq\varphi_{k+1}^{(k)}.
\]
Symmetry gives the opposite strict order
\[
 \varphi_{k+1}^{(k+1)}<\frac{\pi}{2}
 <\varphi_{k+1}^{(k)}=\pi-\varphi_k^{(k)},
\]
a contradiction.  Hence \(\tau_{k+1}<\tau_k\), and
\(\Lambda_k(\alpha)=\pi-\tau_k\) proves
\eqref{eq:lovasz-transform-monotonicity}.

\par\smallskip\noindent\emph{Continuity in the base diameter.}
We prove continuity in \(\alpha\), including the endpoint
extension, by identifying every subsequential limit of the walks.
For an interior value of \(\alpha\),
we must first show that the limiting first step lies strictly between \(\beta\) and \(\pi/2\), and that the limiting colatitudes are distinct.  Let \(\alpha_\nu\to\alpha\in(\tfrac{\pi}{2},\pi]\), and write
\(\varphi^{(\nu)}_0,\ldots,\varphi^{(\nu)}_{2k+1}\) for the corresponding
walks.  Every subsequence has a further subsequence on which all colatitudes
converge, and the recurrence and symmetry identities pass to the limit.
Write \(\varphi_j\) for these limits and \(\tau:=\varphi_1\).
If \(\alpha<\pi\), suppose first that \(\tau=\beta\).  Set
\(H(u,v):=\cos u\cos v+a\sin u\sin v\).  For \(0\leq u<\pi/2\),
\[
 H(u,u)=a+(1-a)\cos^2u>a,
 \qquad H(u,\pi/2)=a\sin u<a,
\]
and \(H(u,\cdot)\) is strictly decreasing on \([u,\pi/2]\).
The limiting colatitudes
\(\varphi_0,\ldots,\varphi_k\) are nondecreasing and at most
\(\pi/2\).  Starting at \(\varphi_0=0\), the recurrence
\(H(\varphi_j,\varphi_{j+1})=a\) therefore forces each next
colatitude through index \(k\) into \((\varphi_j,\pi/2)\).  In particular \(\varphi_k<\pi/2\), whereas
the middle equation \eqref{eq:lovasz-middle-equation} with
\(\tau=\beta\) forces \(\cos\varphi_k=0\), a contradiction.
The inequalities for the approximating walks give
\(\beta\leq\tau\leq\tfrac{\pi}{2}\).  We have just excluded equality on
the left.  Equality on the right would make
\eqref{eq:lovasz-middle-equation} give
\(\cos^2\varphi_k=a/(1+a)>0\), contradicting
\(\tau=\varphi_1\leq\varphi_k\).  Hence
\(\beta<\tau<\tfrac{\pi}{2}\).  The middle equation gives
\[
 \cos^2\varphi_k=\frac{a-\cos\tau}{1+a}>0,
\]
so {the limiting colatitude \(\varphi_k\)} remains strictly below
\(\tfrac{\pi}{2}\).  For \(0\leq u\leq\pi/2\), one has
\(H(u,u)\geq a>\cos\tau\), so consecutive limiting colatitudes
cannot coincide while the current colatitude is in this interval.
The nondecreasing limiting sequence therefore follows the unique
increasing roots \(F_\tau(u)\) in
\eqref{eq:lovasz-shooting-step} through the middle pair.
Equation~\eqref{eq:general-stack-symmetry} determines the remaining
colatitudes by reflection.  Thus the limit is a symmetric stack walk
for \(\alpha\), and \cref{prop:canonical-symmetric-stack-walk}
identifies it with the canonical walk.

\par\smallskip\noindent\emph{The antipodal endpoint.}
If \(\alpha=\pi\), then \(a\to1\), and the limiting recurrence becomes
\[
 \cos(\varphi_{r+1}-\varphi_r)=\cos\varphi_1.
\]
The limiting colatitudes are nondecreasing and lie in \([0,\pi]\).  Each
increment and \(\varphi_1\) therefore lies in \([0,\pi]\), where cosine is
injective, and hence
\(\varphi_{r+1}-\varphi_r=\varphi_1\).  Summing the \(2k+1\) increments gives
\(\varphi_1=\pi/(2k+1)\), and therefore
\(\Lambda_k(\alpha_\nu)\to\delta_k\).  Every convergent subsequence has the
same limit, proving continuity and
\eqref{eq:lovasz-transform-antipodal-extension}.

\end{proof}

\subsection{{Geometric convergence for a fixed base diameter}}

The diameter error decreases geometrically when the base diameter is
fixed.  This sharper estimate is not uniform as the base diameter approaches
\(\pi\), and is not needed for the derived-set calculation.

For a fixed \(\alpha\in(\tfrac{\pi}{2},\pi)\), one has
\begin{equation} \label{eq:lovasz-transform-geometric}
 0<\alpha-\Lambda_k(\alpha)
 <\tan\frac{\alpha}{2}\,
       \operatorname{sech}^2\!\left(
       2k\operatorname{artanh}\tan\frac{\pi-\alpha}{2}\right)
 =\frac{
 4\tan(\frac{\alpha}{2})
 \tan^{2k}(\frac{2\alpha-\pi}{4})}
 {\left(1+\tan^{2k}(\frac{2\alpha-\pi}{4})\right)^2}.
\end{equation}

To prove \eqref{eq:lovasz-transform-geometric}, we compare the
canonical walk with the orbit whose first step is \(\pi-\alpha\).
An explicit formula for that orbit bounds the middle colatitude, which
controls the diameter error through
\eqref{eq:lovasz-middle-equation}. Fix \(\alpha\in(\tfrac{\pi}{2},\pi)\), put
\[
 \beta:=\pi-\alpha,
 \qquad
 \eta_\alpha:=\operatorname{artanh}\tan\frac{\pi-\alpha}{2},
\]
and let \(\tau_k\) be the first step of the canonical \(k\)-layer walk,
whose colatitudes we write as \(\varphi_j^{(k)}\).
Using the continuous extension of \(F_\tau\) to \(\tau=\beta\), define the limiting orbit
\(\bar\varphi_0=0\),
\(\bar\varphi_{j+1}=F_\beta(\bar\varphi_j)\).
If
\[
 b:=\tan\frac{\beta}{2},
 \qquad
 t_j:=\tan\frac{\bar\varphi_j}{2},
\]
then the half-angle form of
\eqref{eq:lovasz-shooting-step} at \(\tau=\beta\) gives the first
identity below.  Starting from \(t_0=0\), the addition formula for
hyperbolic tangent gives the second:
\[
 t_{j+1}=\frac{t_j+b}{1+bt_j},
 \qquad
 t_j=\tanh(j\eta_\alpha).
\]
Thus
\begin{equation}
 \cos\bar\varphi_j=\operatorname{sech}(2j\eta_\alpha).
 \label{eq:lovasz-limit-orbit}
\end{equation}
Since \(\tau_k>\beta\),
\eqref{eq:lovasz-step-monotonicity} gives
\(\varphi_k^{(k)}>\bar\varphi_k\).  By \eqref{eq:lovasz-middle-equation},
\[
 \cos\beta-\cos\tau_k
 =(1+\cos\beta)\cos^2\varphi_k^{(k)}
 <(1+\cos\beta)\operatorname{sech}^2(2k\eta_\alpha).
\]
On the other hand,
\[
 \cos\beta-\cos\tau_k
 =\int_\beta^{\tau_k}\sin u\,\diff u
 \geq(\tau_k-\beta)\sin\beta.
\]
Since
\[
 \alpha-\Lambda_k(\alpha)=\tau_k-\beta,
 \qquad
 \frac{1+\cos\beta}{\sin\beta}=\tan\frac{\alpha}{2},
\]
this proves the first line of
\eqref{eq:lovasz-transform-geometric}.  Finally,
\[
 e^{-2\eta_\alpha}
 =\frac{1-b}{1+b}
 =\tan\frac{2\alpha-\pi}{4},
\]
which gives its second line.

\subsection{{Convergence rate for the iterated one-layer stacks}}

For the diameters \(\theta_{m,k}\) of the iterated one-layer
stacks in \cref{thm:stationary-iterated-one-layer-stacks}, the approach
to \(\zeta_{m-1}=\arccos(-1/m)\) has the following rate.
For \(m\geq2\),
\[
 \zeta_{m-1}-\theta_{m,k}
 =\frac{\pi^2}
 {2m\sqrt{m^2-1}\,(2k+1)^2}+O(k^{-4}),
\]
while \(\pi-\theta_{1,k}=\tfrac{\pi}{2k+1}\).

Indeed, with \(a_k:=\cos(\pi/(2k+1))\),
\[
 a_k=1-\frac{\pi^2}{2(2k+1)^2}+O(k^{-4}),
\]
and substitution into
\eqref{eq:iterated-one-layer-stack-diameter}, followed by Taylor
expansion of \(\arccos\) at \(-\tfrac{1}{m}\), gives the displayed
asymptotic for \(m\geq2\).  For \(m=1\), the same diameter formula
gives the exact expression above.

\bibliographystyle{alpha}
\bibliography{references}

\end{document}